\documentclass[12pt,sans]{article}
\usepackage{fullpage}
\usepackage{url}
\usepackage{transparent,url,examplep}
\usepackage[colorlinks = true,
            linkcolor = blue,
            urlcolor  = blue,
            citecolor = blue,
            anchorcolor = blue]{hyperref}
\usepackage[utf8]{inputenc} 
\usepackage[T1]{fontenc}    
\usepackage{lmodern}
\usepackage{mathrsfs}
\usepackage[sans]{dsfont}   
\usepackage{libertine}
\usepackage{booktabs}       
\usepackage{nicefrac}       
\usepackage{microtype}      

\usepackage{amsmath,amsthm,amssymb,amsfonts}   
\usepackage{latexsym}
\usepackage{optidef}
\usepackage{graphicx}
\usepackage{caption}
\usepackage{subcaption}
\usepackage{epstopdf}
\usepackage{cleveref}
\usepackage{autonum}
\usepackage{color}
\usepackage{algorithm,algorithmic}
\usepackage{tcolorbox}
\usepackage{geometry}
\usepackage{accents}

\newtheorem{theorem}{Theorem}[section]
\newtheorem{lemma}[theorem]{Lemma}
\newtheorem{corollary}{Corollary}
\newtheorem{proposition}{Proposition}
\newtheorem{assumption}{Assumption}
\newtheorem{setup}{Set-up}
\newtheorem{definition}[theorem]{Definition}
\newtheorem*{fact}{Fact}
\theoremstyle{remark}
\newtheorem{remark}{Remark}

\makeatletter
\DeclareRobustCommand{\varamalg}{%
  \mathbin{\mathpalette\var@malg\perp}%
}

\newcommand\var@malg[2]{%
  \rlap{$\m@th#1#2$}\mkern6mu{#1#2}%
}
\makeatother

\makeatletter
\def\@tvsp{\mathchoice{{}\mkern-4.5mu}{{}\mkern-4.5mu}{{}\mkern-2.5mu}{}}
\def\ltrivert{\left|\@tvsp\left|\@tvsp\left|}
\def\rtrivert{\right|\@tvsp\right|\@tvsp\right|}
\makeatother

\makeatletter
\def\@tvsp{\mathchoice{{}\mkern-4.5mu}{{}\mkern-4.5mu}{{}\mkern-2.5mu}{}}
\def\llangle{\langle\@tvsp\langle}
\def\rrangle{\rangle\@tvsp\rangle}
\makeatother

\newcommand{\esp}{\mathbb{E}}
\newcommand{\prob}{\mathbb{P}}
\newcommand{\probn}{\mathbf{P}}

\DeclareMathOperator*{\var}{\mathbb{V}}

\newcommand{\unit}{\mathbf{1}}
	
\newcommand{\re}{\mathbb{R}}

\newcommand{\tr}{\mathsf{tr}}

\newcommand{\KL}{\mathsf{KL}}

\DeclareMathOperator*{\diag}{diag}

\newcommand{\RFHP}{\mathsf{P}}
\DeclareMathOperator*{\MW}{\mathtt{MW}}
\DeclareMathOperator*{\Round}{\mathtt{ROUND}}
\DeclareMathOperator*{\Algorithm}{\mathtt{Algorithm}}
\DeclareMathOperator*{\Pruning}{\mathtt{Pruning}}

\DeclareMathOperator*{\RobustRegression}{\mathtt{Robust-Regression}}
\DeclareMathOperator*{\RobustDirection}{\mathtt{Robust-Direction}}
\DeclareMathOperator*{\AdaptiveRobustRegression}{\mathtt{Adaptive-Robust-Regression}}

\def\bfDelta{\boldsymbol{\Delta}}

\def\bfSigma{\boldsymbol{\Sigma}}

\def\bfXi{\boldsymbol{\Xi}}
\def\bfA{\mathbf{A}}
\def\bfB{\mathbf{B}}
\def\bfC{\mathbf{C}}

\def\bfI{\mathbf{I}}
\def\bfM{\mathbf{M}}

\def\bfS{\mathbf{S}}

\def\bfZ{\mathbf{Z}}

\def\bb{\boldsymbol b}

\def\be{\boldsymbol e}

\def\bx{\boldsymbol x}
\def\by{\boldsymbol y}
\def\bq{\boldsymbol q}
\def\bu{\boldsymbol u}
\def\bv{\boldsymbol v}
\def\bw{\boldsymbol w}
\def\bz{\boldsymbol z}
\def\bf0{\mathbf{0}}

\def\bmu{\boldsymbol\mu}

\def\btheta{\boldsymbol\theta}

\def\bpi{\boldsymbol\pi}

\def\calA{\mathcal A}
\def\calB{\mathcal B}

\def\calD{\mathcal D}
\def\calE{\mathcal E}

\def\calG{\mathcal G}
\def\calI{\mathcal I}
\def\calK{\mathcal K}

\def\calM{\mathcal M}

\def\calO{\mathcal O}

\def\calS{\mathcal S}

\def\calU{\mathcal U}

\def\calW{\mathcal W}

\def\sfC{\mathsf{C}}

\def\sa{\mathsf{a}}
\def\sb{\mathsf{b}}

\def\mbB{\mathbb{B}}
\def\mbC{\mathbb{C}}
\def\mbS{\mathbb{S}}

\def\mbZ{\mathbb{Z}}

\newcommand{\vertiii}[1]{{\left\vert\kern-0.25ex\left\vert\kern-0.25ex\left\vert #1 
    \right\vert\kern-0.25ex\right\vert\kern-0.25ex\right\vert}}
\newcommand{\verti}[1]{{\left\vert\kern-0.4ex #1 
\kern-0.4ex\right\vert}}

\DeclareMathOperator*{\argmin}{argmin}
\DeclareMathOperator*{\argmax}{argmax}

\title{A spectral least-squares-type method for heavy-tailed corrupted regression with unknown covariance \& heterogeneous noise}
\date{}

\author{
Roberto I. Oliveira \and
Zoraida F. Rico \and
Philip Thompson \and
}

\newcommand{\Addresses}{{
  \bigskip
  \footnotesize

Roberto I.~Oliveira, \textsc{IMPA,
    Rio de Janeiro, RJ, Brazil}\par\nopagebreak
  \textit{E-mail address}: \texttt{rbimfo@impa.br} 

\and 

Zoraida F.~Rico, \textsc{Columbia University,
    New York, NY}\par\nopagebreak
  \textit{E-mail address}: \texttt{zoraida.f.rico@columbia.edu} 

\and

  P.~Thompson, \textsc{Krannert School of Management, Purdue University,
    West Lafayette, Indiana}\par\nopagebreak
  \textit{E-mail address}: \texttt{thompsp@purdue.edu}
}}

\begin{document}

\maketitle
\Addresses

\begin{abstract}
We revisit heavy-tailed corrupted least-squares linear regression assuming to have a  corrupted $n$-sized label-feature sample of at most $\epsilon n$ arbitrary outliers. We wish to estimate $\bb^*\in\re^p$ given such sample of a label-feature pair $(y,\bx)\in\re\times\re^p$ satisfying
$
y=\langle\bx,\bb^*\rangle+\xi,
$
with heavy-tailed $(\bx,\xi)$. We only assume $\bx$ is $L^4-L^2$ hypercontractive  with constant $L>0$ and has covariance matrix $\bfSigma$ with minimum eigenvalue 
$\nicefrac{1}{\mu^2(\mbB_2)}>0$ and bounded condition number $\kappa>0$. The noise $\xi\in\re$ can be arbitrarily dependent on $\bx$ and nonsymmetric as long as $\xi\bx$ has finite covariance matrix $\bfXi$. We propose a near-optimal computationally tractable estimator, based on the power method, assuming no knowledge on 
$(\bfSigma,\bfXi)$ nor the operator norm of $\bfXi$. With probability at least $1-\delta$, our proposed estimator attains the statistical rate 
$
\mu^2(\mbB_2)\Vert\bfXi\Vert^{1/2}(\frac{p}{n}+\frac{\log(1/\delta)}{n}+\epsilon)^{1/2}
$
and breakdown-point 
$
\epsilon\lesssim\frac{1}{L^4\kappa^2},
$
both optimal in the $\ell_2$-norm, assuming the near-optimal minimum sample size 
$
L^4\kappa^2(p\log p + \log(1/\delta))\lesssim n
$, up to a log factor. To the best of our knowledge, this is the first computationally tractable algorithm satisfying simultaneously all the mentioned properties. Our estimator is based on a two-stage Multiplicative Weight Update algorithm. The first stage estimates a descent direction $\hat\bv$ with respect to the (unknown) pre-conditioned inner product $\langle\bfSigma(\cdot),\cdot\rangle$. The second stage estimate the descent direction $\bfSigma\hat\bv$ with respect to the (known) inner product 
$\langle\cdot,\cdot\rangle$, without knowing nor estimating $\bfSigma$. 
\end{abstract}

\section{Introduction}

Least-squares regression is a fundamental problem in statistics and machine learning, either from a practical or theoretical standpoint. However, classical methodologies for this  problem assume the collected data is clean and light-tailed. Robust Statistics \cite{1964huber, 2005hampel:ronchetti:rousseeuw:stahel,  2006maronna:martin:yohai, 2011huber:ronchetti} aim in addressing robust estimation when either the sample is  corrupted or the data generating distribution is too heavy-tailed. 

In recent work, the minimax optimality of several robust estimation problems have been attained \cite{2016chen:gao:ren,2018chen:gao:ren} and \cite{2020gao}. The construction of these estimators, however, is based on Tukey's depth, a hard computational problem in higher dimensions. Fundamental recent works \cite{2016diakonikolas:kane:karmalkar:price:stewart, 2016lai:rao:vempala} have proposed alternative estimators that are both computationally tractable and statistically (near) optimal. For instance, near optimal robust estimators for the mean of a high-dimensional vector can be computed in nearly-linear time \cite{2019cheng:diakonikolas:ge, 2019dong:hopkins:li, 2019depersin:lecue, 2020hopkins:li:zhang, 2020dalalyan:minasyan}.  We refer to  \cite{2019diakonikolas:kane,2019lugosi:mendelson-survey} for extensive surveys. 

In this work, we revisit the problem of \emph{heavy-tailed} least-squares regression assuming to have an \emph{adversarially corrupted} $n$-sized sample. Here, ``adversarial'' means that the sample, corrupted in both labels and features, has at most $\epsilon n$ arbitrary outliers for some contamination fraction 
$\epsilon\in(0,1/2)$. In particular, the adversary mechanism can depend on the (unobserved) clean iid sample. The goal of this paper is to establish near-optimal statistical rates for this problem with a computationally tractable estimator and minimal assumptions. Precisely, our main result can be resumed as follows:
\begin{itemize}
\item[\rm a)] \emph{Optimality in high-probability}. We assume that the feature vector $\bx$ has finite covariance matrix $\bfSigma$, with minimum eigenvalue $1/\mu^2(\mbB_2)>0$, maximum eigenvalue $\Vert\bfSigma\Vert$ and condition number $\kappa<\infty$, and satisfy the $L^4-L^2$ hypercontractive property with constant $L>0$. Moreover, the noise-feature multiplier vector $\xi\bx$ is assumed to have finite covariance matrix $\bfXi$, with maximum eigenvalue $\Vert\bfXi\Vert<\infty$. Under these standard assumptions, with probability at least $1-\delta$, our proposed estimator achieves the $\ell_2$-norm estimation rate $\mu^2(\mbB_2)\Vert\bfXi\Vert^{1/2}(\sqrt{p/n}+\sqrt{\log(1/\delta)/n}+\sqrt{\epsilon})$  with sample size of at least $n\ge C L^4\kappa^2(p\log p+\log(1/\delta))$ and contamination fraction of at most $\epsilon\le\frac{1}{CL^4\kappa^2}$. Here, $C>0$ is a absolute constant. The mentioned rate and cut-offs in $(n,\epsilon)$ are all optimal in $(n,p,\delta,\epsilon)$, including the constants $(\mu(\mbB_2),\Vert\bfSigma\Vert,\Vert\bfXi\Vert)$, up to a log factor. Using a ``least-squares methodology'', the statistical and optimization rates of our algorithm do not depend on the 
$\ell_2$-norm of the ground truth parameter, only on the condition numbers $(\mu^2(\mbB_2)\Vert\bfXi\Vert^{1/2},\kappa)$.
\item[\rm b)] \emph{Heterogeneous noise}. We are mainly concerned with the statistical learning framework over the linear class in mean least-square sense. In this set-up, the noise $\xi$ can be arbitrarily dependent on $\bx$ and does not need to be symmetric. 
\item[\rm c)] \emph{Tractability via spectral methods}. The seminal works \cite{2016diakonikolas:kane:karmalkar:price:stewart, 2016lai:rao:vempala} were the first to suggest that, to construct computationally tractable robust mean estimators, one must exploit the eigenstructure of the sample covariance matrix. Various approaches have been developed since then for robust linear regression. Some current approaches make use of significantly more time consuming approaches such as semi-definite programming (SDP) or sum-of-squares algorithms. Our estimator is computationally tractable by means of faster spectral methods \cite{2020lei:luh:venkat:zhang,2020depersin}. The main computational bottleneck is to run a logarithmic number of iterations of a Multiplicative Weight Update algorithm (MWU) \cite{2012karnin:liberty:lovett:schwartz:weinstein,2012arora:hazan:kale} in which every iteration requires to approximately solve a maximum eigenvalue problem. This can be done e.g. via a randomized power method. 
\item[\rm d)] \emph{Unknown covariances and noise level}. With a light-tailed iid clean sample, the least squares estimator is known to be optimal without knowing the covariance matrices $(\bfSigma, \bfXi)$ nor the noise variance $\sigma^2$. Likewise, our estimator satisfy (a)-(c) without the need to know $(\bfSigma, \bfXi)$ nor $\Vert\bfXi\Vert$.
\end{itemize}

To the best of our knowledge, as discussed next, we believe this is the first analysis with a computationally tractable algorithm, based on a least-squares methodology instead of a gradient estimation methodology, satisfying simultaneously all the mentioned properties.  

\subsection{Related work}

Outlier-robust linear regression has already been subject to a lot of research since the seminal work of Huber \cite{1964huber}. In the particular model of label contamination,  estimators based on Huber-type losses are optimal. Unlike the label-feature contamination model, optimal estimators for the label-contamination model can be tuned adaptively to $(\epsilon,\delta)$ and have $\kappa$-free breakdown points. Also, optimal estimators are asymptotically consistent in case the model is oblivious. See e.g. \cite{2019dalalyan:thompson,2020thompson} for an extensive review.  

The more general problem of label-feature corrupted linear regression has been previously considered in the works 	\cite{2019diakonikolas:kamath:kane:li:steinhardt:stewart, 2019diakonikolas:kong:stewart, 2018prasad:suggala:balakrishnan:ravikumar, 2020zhu:jiao:steinhardt, 2021bakshi:prasad, 
2020depersin, 2020cherapanamjeri:aras:tripuraneni:jordan:flammarion:bartlett,
2020pensia:jog:loh, 
2021jambulapati:li:schramm:tian}. \cite{2019diakonikolas:kamath:kane:li:steinhardt:stewart, 2019diakonikolas:kong:stewart, 2018prasad:suggala:balakrishnan:ravikumar} focused on the subgaussian corrupted model. \cite{2020zhu:jiao:steinhardt,2021bakshi:prasad} considered assumptions and algorithmic approaches that are statistically optimal with polynomial time complexity. Still, they require more restrictively sample complexity and distribution assumptions. For instance, \cite{2021bakshi:prasad} is based on sum-of-squares methodology which is more time consuming than spectral methods. \cite{2020depersin,2020pensia:jog:loh}, as this work, are based on a least-squares methodology, but they require full knowledge of the feature covariance matrix $\bfSigma$.  \cite{2020cherapanamjeri:aras:tripuraneni:jordan:flammarion:bartlett, 2021jambulapati:li:schramm:tian} do not require knowledge of  
$\bfSigma$ but their estimators, like \cite{2019diakonikolas:kong:stewart, 2018prasad:suggala:balakrishnan:ravikumar}, follow a different approach, based on robust gradient estimation. \cite{2020cherapanamjeri:aras:tripuraneni:jordan:flammarion:bartlett} is based on SDP, a more time consuming approach. Also, their optimization complexity depends, unlike this work, on the $\ell_2$-norm of the ground truth parameter and they assume independence between noise and feature vector. \cite{2021jambulapati:li:schramm:tian}, like this work, are constructed with spectral methods. Still, they assume independence between noise and the feature vector, require knowledge of the noise variance $\sigma^2$ and their complexity depend on the $\ell_2$-distance between the ground truth and the initial estimate. \cite{2020cherapanamjeri:aras:tripuraneni:jordan:flammarion:bartlett, 2021jambulapati:li:schramm:tian} are not concerned with optimality with respect to $\delta$. Also, their minimal requirement on the sample size is of order $\tilde\calO(p)/\epsilon\lesssim n$; we require $\tilde\calO(p)+\log(1/\delta)\lesssim n$, independently of $\epsilon$. 	

It is instructive to conclude this section with a discussion between two methodologies used in heavy-tailed corrupted estimation. In a nutshell, computationally tractable estimators for this problem are based on two frameworks. The first solve (approximately) the semi-definite programming given a set of points $\{\bz_i\}_{i=1}^n$ and $k\in\{1,\ldots,n\}$:
\begin{align}
\min_{w\in\Delta_{n,k}}\sup_{\bv\in\mbB_2}\sum_{i=1}^nw_i\bz_i\bz_i^\top,\label{problem:above}
\end{align}
where $\Delta_{n,k}:=\{w\in\re^n:w_i\ge0,\sum_{i=1}^nw_i=1,w_i\le\frac{1}{n-k}\}$ and $\mbB_2$ is the Euclidean unit ball. This is the approach followed by \cite{2019dong:hopkins:li, 2019depersin:lecue, 2020hopkins:li:zhang,  2020cherapanamjeri:aras:tripuraneni:jordan:flammarion:bartlett,2021jambulapati:li:schramm:tian}. A complementary approach, initiated for robust mean estimation in \cite{2020hopkins,2020lei:luh:venkat:zhang}, is to consider tractable relaxations of the combinatorial problem  
\begin{maxi}
  {(\theta,\bv,\bq)\in\re\times \mbB_2\times\re^K}{\theta}{}{}{}
  \addConstraint{\bq_i|\langle\bz_i,\bv\rangle|}{\ge\bq_i\theta,}{i=0,\ldots K}
  \addConstraint{\sum_{i=1}^\calK\bq_i}{> K-k}
  \addConstraint{\bq_i}{\in\{0,1\},\quad}{i=0,\ldots K.}
  \label{problem:below}
 \end{maxi}
Here, $\{\bz_i\}_{i=1}^K$ are initially pre-processed from data using a Median-of-Means framework. See Section \ref{ss:benchmark:problem} and \cite{2020hopkins, 2020lei:luh:venkat:zhang} for further discussion on the motivation for why studying this problem.  Most closely to our work are \cite{2020lei:luh:venkat:zhang,2020depersin}. \cite{2020lei:luh:venkat:zhang} is focused on robust mean estimation.  One important difference between robust mean estimation and linear regression is that, unlike approaches based on \eqref{problem:above}, methods based on \eqref{problem:below} do not require high-probability concentration bounds for the 4th order tensor. See \cite{2020cherapanamjeri:aras:tripuraneni:jordan:flammarion:bartlett} for further discussion on this issue. To the best of our knowledge, \cite{2020depersin} is the first work aiming in generalizing the approach of \cite{2020hopkins, 2020lei:luh:venkat:zhang} to heavy-tailed corrupted linear regression. Still, one important limitation is that \cite{2020depersin} requires full knowledge of the covariance matrix 
$\bfSigma$; this explains why their rate is independent of the condition number $\kappa$. \cite{2020depersin} requires the number of buckets $K$ to be depend on the dimension $p$ while our estimator uses $K$ independent of $p$. One key development in our analysis is to show that a \emph{two-stage algorithm} based on the MWU algorithm is enough to avoid the need of the knowledge of $\bfSigma$. The first stage estimates a descent direction $\hat\bv$ with respect to the (unknown) pre-conditioned inner product $\langle\bfSigma(\cdot),\cdot\rangle$. The second stage estimate the descent direction $\bfSigma\hat\bv$ with respect to the (known) inner product 
$\langle\cdot,\cdot\rangle$, without knowing nor estimating $\bfSigma$. 

Finally, one important difference between robust mean estimation and linear regression concerns the initialization. For instance, the easy to compute coordinate-wise median turns out to be sufficient for the initialization of robust mean estimators \cite{2019depersin:lecue, 2020dalalyan:minasyan}. One the other hand, there is no coordinate-wise median counterpart for robust linear regression. To the best of our knowledge, the properties needed for the initialization in \cite{2020depersin}, the most close to our work, are assumed a priori without formal derivation. We formalize guarantees for the initialization of robust linear regression based on Median-of-Least-Squares (MLS) estimators. In that regard, unlike assuming a priori invertability assumptions for the bucket design matrices as in  \cite{2015minsker} or Srivastava-Vershynin condition, a stronger assumption than hypercontractivity, as in \cite{2016hsu:sabato}, we derive sufficient lower bounds in high-probability for MLS estimators with tight dependence on $(K,\delta)$ assuming only hypercontractivity. See Proposition \ref{prop:quad:proc:block:lower:heavy} in Section \ref{ss:bounds:liner:regression}. This is a analog for MLS estimators of the PAC-Bayesian tool developed in \cite{2016oliveira}. We also remark that our initialization is adaptive to any of the parameters 
$(\mu^2(\mbB_2),\Vert\bfSigma\Vert,\sigma^2,\Vert\bfXi\Vert)$. 

We conclude with a minor observation regarding randomized rounding, a needed tool in most of the literature. Our rounding scheme is based on the spherical distribution instead of the Gaussian distribution. This somewhat  simplifies the rounding analysis in \cite{2019depersin:lecue,2020lei:luh:venkat:zhang} and it also seems to imply a larger confidence interval. See Proposition \ref{prop:unif:round}. 

\subsection{Framework}
\label{s:framework}

Let $(y,\bx)\in\re\times\re^p$ be a label-feature pair with centered feature $\bx$. Within a statistical learning framework, we wish to explain $y$ trough $\bx$ via the linear class 
$
F_1(\re^p):=\{\langle\cdot,\bb\rangle:\bb\in\re^p\}.
$ 
Precisely, giving a sample of $(y,\bx)$, we wish to estimate 
\begin{align}
\bb^*\in\argmin_{\bb\in\re^p}\esp\left(y-\langle\bx,\bb\rangle\right)^2.\label{equation:least-squares:regression}
\end{align}
In particular, one has
$
y=\langle\bx,\bb^*\rangle+\xi
$
with $\xi\in\re$ having zero mean and satisfying 
$\esp[\xi\bx]=0$. This is our only assumption on $\xi$:  we do not assume $\xi$ and $\bx$ are independent nor that $\xi$ is symmetric. 

\begin{assumption}[Heavy-tails]\label{assump:L4:L2}
Assume:
\begin{itemize}
\item The feature vector $\bx$ has distribution $\Pi$ and unknown finite non-singular covariance matrix 
$
\bfSigma:=\esp[\bx\bx^\top]
$
with known maximum eigenvalue $\Vert\bfSigma\Vert<\infty$ and minimum eigenvalue $\frac{1}{\mu^2(\mbB_2)}>0$. Moreover, $\bx$ satisfies the $L^{4}-L^2$ norm equivalence condition with unknown constant $L>0$: for all $\bv\in\mbB_2$, 
\begin{align}
\left\{\esp|\langle\bz,\bv\rangle|^{4}\right\}^{1/4}
\le L \left\{\esp|\langle\bz,\bv\rangle|^{2}\right\}^{1/2}.
\end{align}
\item The centered noise $\xi$ satisfy 
$\esp[\xi\bx]=0$,  has finite variance 
and the multiplier vector $\xi\bx$ has unknown finite population covariance matrix
$
\bfXi:=\esp[\xi^2\bx\bx^\top]
$
with unknown maximum eigenvalue $\Vert\bfXi\Vert$. 
\end{itemize}
\end{assumption} 
$L^{4}-L^2$ norm equivalence is also known as \emph{bounded 4th moment}, \emph{bounded kurtosis} or \emph{hypercontractivity} conditions \cite{2020mendelson:zhivotovskiy}.

We consider available a label-feature sample with adversarial contamination.
\begin{assumption}[Label-feature adversarial contamination]\label{assump:contamination}
The contamination fraction will be denoted by 
$\epsilon:=\frac{o}{n}$. This means that it is available a label-feature sample $\{(\tilde y_\ell,\tilde\bx_\ell)\}_{\ell=1}^{2n}$ having an arbitrary subset of exactly $o$ data points differing from the label-feature sample $\{(y_\ell,\bx_\ell)\}_{\ell=1}^{2n}$ which is an independent iid copy of $(y,\bx)\in\re\times\re^p$. We use the notations  
$\xi_\ell:=y_\ell-\langle\bx_\ell,\bb^*\rangle$ and 
$
\tilde\xi_\ell:=\tilde y_\ell-\langle\tilde\bx_\ell,\bb^*\rangle
$
for all $\ell\in[n]$. 
\end{assumption}

We remark that assuming knowledge of 
$(\Vert\bfSigma\Vert,\mu^{-2}(\mbB_2))$ is not restrictive in the heavy-tailed corrupted model of Assumption \ref{assump:contamination}. Using a separate batch of the corrupted sample, there exist tractable robust estimators 
$(\hat\lambda_{\max},\hat\lambda_{\min})$ satisfying, with high-probability,  
$\sa_1\Vert\bfSigma\Vert\le\hat\lambda_{\max}\le\sa_2\Vert\bfSigma\Vert$
and 
$\sa_3\mu^{-2}(\mbB_2)\le\hat\lambda_{\min}\le\sa_4\mu^{-2}(\mbB_2)$
for positive constants $(\sa_1,\sa_2,\sa_3,\sa_4)$. See for instance \cite{2020mendelson:zhivotovskiy}. Using those estimates in our algorithms entail the same rate of Theorem \ref{thm:main} up to changes in absolute constants. 

Next, we formally state our main result. Its full derivation requires several intermediate results developed in the next sections. 

\begin{theorem}\label{thm:main}
Grant Assumptions \ref{assump:L4:L2} and \ref{assump:contamination}. Assume that 
$\max\{1,\Vert\bfXi\Vert\}<\gamma\zeta_0$ for some known
$\gamma\in(0,1)$ and $\zeta_0>0$; let $M:=\lceil\log_{\gamma^{-1}}(\zeta_0)\rceil$.

Then, given any desired probability of failure 
$\delta_0\in(0,1)$, if the sample size and contamination fraction satisfy
\begin{align}
n&\ge(\sfC L^4\kappa^2p\log p)\bigvee(\sfC L^4\kappa^2\log(\nicefrac{\sfC M}{\delta_0})),\\
\epsilon&\le\frac{1}{\sfC L^4\kappa^2}, 
\end{align}
there exists a computationally efficient algorithm (namely $\Algorithm$ \ref{algo:robust:regression:main:adaptive} in Section \ref{ss:adaptive:algo}) with inputs $\{(\tilde y_\ell,\tilde\bx_\ell)\}_{\ell=1}^n\cup\{(\tilde y_\ell,\tilde\bx_\ell)\}_{\ell=n+1}^{2n}$, $\mu^2(\mbB_2)$, 
$\Vert\bfSigma\Vert$, $M$ and $\delta_0$ and $\epsilon$ satisfying, with probability at least $1-\delta_0$, 
\begin{align}
\Vert\hat\bb-\bb^*\Vert_2&\lesssim
\mu^2(\mbB_2)\Vert\bfXi\Vert^{1/2}
\left(
\frac{p}{n}+\frac{1+\log(M)+\log(1/\delta_0)}{n}+\epsilon
\right)^{1/2}.
\end{align}
If one knows $\Vert\bfXi\Vert$, then we can take $M=1$ above. 
\end{theorem}

Throughout the paper we will denote the $\ell$th  residual at $\bb\in\re^p$ by
$
\tilde\xi_\ell(\bb):=\tilde y_\ell-\langle\tilde\bx_\ell,\bb\rangle.
$
Recall the bilinear form
$
\langle\bv,\bw\rangle_\Pi:=\esp[\langle\bx,\bv\rangle\langle\bx,\bw\rangle], 
$
and the $L^2(\Pi)$ pseudo-norm 
$
\Vert\bv\Vert_{\Pi}:=\sqrt{\langle\bv,\bv\rangle_\Pi^2}.
$
Given a cone $\mbC\subset\re^p$, we define the restricted eigenvalue constant 
$
\mu(\mbC):=\sup_{\bv\in\mbC}\frac{\Vert\bv\Vert_2}{\Vert\bv\Vert_\Pi}. 
$

\section{Notation}

\subsection{Basic notation}

Let $[n]:=\{1,\ldots,n\}$. Denote by $\Delta_n$ the $n$-dimensional simplex and, for given $k\in[n-1]$, 
\begin{align}
    \Delta_{n,k}:=\left\{w\in\Delta_n:w_i\le \frac{1}{n-k}\right\}.
\end{align}
We denote by $\KL(p\Vert q)$ the Kullback-Leibler divergence between two distributions in $\Delta_n$. 

We'll write $a\lesssim b$ if $a\le C b$ for an absolute constant and say $a\asymp b$ if $a\lesssim b$ and $b\lesssim a$. We use the usual notations  $a_+:=\max\{0,a\}$, $a_-:=\max\{0,-a\}$, $a\vee b:=\max\{a,b\}$ and $a\wedge b:=\min\{a,b\}$. Given sequence $\{\sigma_i\}_{i=1}^m$ of numbers, 
$\sigma_1^*\le\ldots\le\sigma_m^*$ denotes its non-decreasing order while 
$\sigma_1^\sharp\ge\ldots\ge\sigma_m^\sharp$ denotes its non-increasing order. 
We denote the inner product by $\langle\cdot,\cdot\rangle$,  the $\ell^2$-norm by $\Vert\cdot\Vert_2$, the unit balls $\mbB_2:=\{\bv\in\re^p:\Vert\bv\Vert_2\le1\}$, unit sphere $\mbS_2:=\{\bv\in\re^p:\Vert\bv\Vert_2=1\}$, the $\ell_2$-norm ball with center $a$ and radius $r$ by $\mbB_2(a,r)$.

The canonical basis in $\re^p$ will be denoted by $\{\be_1,\ldots,\be_p\}$ and 
$\bfI_p$ denotes the identity matrix. Given non-zero matrix $\bfA$, we denote its trace by 
$\tr(\bfA)$, its operator norm by 
$\Vert\bfA\Vert$ and its Frobenius norm by $\Vert\bfA\Vert_F$. We also define $\rho_{\bfA}^2=\Vert\diag(\bfA)\Vert_\infty$. The standard inner product on 
$\re^{p\times p}$ will be denote by 
$\llangle\bfA,\bfB\rrangle:=\tr(\bfA^\top\bfB)$, where 
$\top$ the transpose operation. Given two vectors $\bv,\bu\in\re^p$, we let $\bv\otimes\bu:=\bv\bu^\top$. We use the notation $\bfM\succeq\bf0$ for a semi-positive definite symmetric matrix $\bfM\in\re^{p\times p}$. Also, its associated bilinear form and pseudo-norm will be denoted respectively by 
$\langle\bu,\bv\rangle_\bfM:=\langle\bfM\bu,\bv\rangle$ and 
$\Vert\bu\Vert_\bfM:=\Vert\bfM^{1/2}\bu\Vert$. We define the set of matrices $\calM(\calB)$ for some compact convex set $\calB\subset\re^p$ as the  convex hull of the set $\{\bv\bv^\top:\bv\in\calB\}$.

\subsection{Some probabilistic notions}\label{ss:prob:notions}

Let $X$ be a random variable with distribution $\probn$ taking values on a measurable set $\mbB$. We denote by $\{X_i\}_{i=1}^n$ an iid copy of $X$. We reserve the notation of $\{\epsilon_i\}_{i\in[n]}$ to represent an iid sequence of Rademacher random variables. Given a class $F$ of integrable functions $F:\mbB\rightarrow\re$ with respect to $\probn$, the \emph{Rademacher complexity} of $F$ associated to $\{X_i\}_{i=1}^n$ is the quantity
$$
\mathscr {R}_{X,n}(F):=\esp\left[\sup_{f\in F}\sum_{i\in[n]}\epsilon_if(X_i)\right],
$$
where $\{\epsilon_i\}_{i\in[n]}$ is independent of $\{X_i\}_{i=1}^n$. A related quantity is 
$$
\mathscr {D}_{X,n}(F):=\esp\left[\sup_{f\in F}\left|\sum_{i\in[n]}(f(X_i)-\esp[f(X_i))]\right|\right],
$$
noting that, by symmetrization, $\mathscr {D}_{X,n}(F)\asymp \mathscr {R}_{X,n}(F)$. We sometimes use the notation $\probn f:=\esp[f(X)]$. The \emph{wimpy variance} of the class $F$ associated to $X$ is the quantity
\begin{align}
\sigma_{X}^2(F):=\sup_{f\in F}\esp[(f(X)-\probn f)^2].
\end{align}

Let $\mbB$ and $\mbB'$ be compact subsets of $\re^p$. Typical classes we will use are the ``linear classes'' $F_1(\mbB):=\{\langle\cdot,\bv\rangle:\bv\in\mbB\}$ and $|F_1|(\mbB):=\{|\langle\cdot,\bv\rangle|:\bv\in\mbB\}$, the ``quadratic class''  $F_2(\mbB):=\{\langle\cdot,\bv\rangle^2:\bv\in\mbB\}$ and the ``product class'' $F(\mbB,\mbB'):=\{\langle\cdot,\bu\rangle\langle\cdot,\bv\rangle:\bu\in\mbB,\bv\in\mbB'\}$.

Letting $\bz\in\re^p$ be a centered random vector with distribution $\probn$, we define the bilinear form 
$$
\langle\bu,\bv\rangle_\probn:=\esp[\langle\bz,\bu\rangle\langle\bz,\bv\rangle],
$$
and the $L^2(\probn)$ pseudo-norm 
$
\Vert\bv\Vert_{\probn}:=\sqrt{\esp[\langle\bz,\bv\rangle^2]}.
$ 
We will also define the unit ellipsoid 
$\mbB_\probn:=\{\bv\in\re^p:\Vert\bv\Vert_\probn\le1\}$ with border 
$\mbS_\probn:=\{\bv\in\re^p:\Vert\bv\Vert_\probn=1\}$. 
Let $\mbB$ be a compact subset of $\re^p$. The \emph{wimpy variance} of $\mbB$ associated to the distribution of $\bz$ is the quantity
$
\sigma_{\bz}^2(\mbB):=\sigma_{\bz}^2(F_1(\mbB)).
$

\section{Concentration bounds}

\begin{tcolorbox}
In all this section, $X$ is a random variable taking values on some set $\mbB$ with distribution $\probn$ and 
$\{X_i\}_{i=1}^n$ an iid copy of $X$. We also split the sample $\{X_i\}_{i=1}^n$ into $K$ blocks of equal size $B:=n/K$ indexed by the partition 
$\bigcup_{k\in [K]}B_k=[n]$. When $\mbB=\re$, $\mu$ and  
$\sigma_X^2<\infty$ will denote the mean and variance of $X$ respectively and  let $\overline X:= X - \mu$. When $\mbB=\re^p$, $X=\bz$ will be a centered $p$-dimensional random vector with covariance matrix $\bfSigma$  and $\{X_i=\bz_i\}_{i=1}^n$ is an iid copy of $\bz$. 
\end{tcolorbox} 

\subsection{Some general bounds}

We define, for any $\eta\in[0,1]$, its $\eta$-quantile by $Q_{\overline X,\eta}$, that is, 
\begin{align}
Q_{\overline X,\eta}:=\sup\left\{x\in\re:\prob(\overline X\ge x)\ge1-\eta\right\}.
\end{align} 
We will assume without loss on generality that 
$X$ is continuous. In particular,
$
\prob(\overline X\ge Q_{\overline X,\eta})=1-\eta.
$

The following result follows from the proofs in \cite{2019lugosi:mendelson-trim}. We give a proof for completeness in the Appendix. 
\begin{lemma}\label{lemma:count:combinatorial:variable:v2}
Let $\eta\in(0,1/2]$. Then, setting $Q:=Q_{\overline X,1-\eta/2}$, with probability at least $1-\exp(-\eta n/1.8)$,
\begin{align}
\sum_{i=1}^n\unit_{\{\overline X_i \le Q\}}\ge (1-0.75\eta) n.
\end{align}
\end{lemma}

Let $X$ be real valued and 
$
\hat\mu_k:= \frac{1}{B}\sum_{\ell\in B_k}X_\ell.
$
We restate the following well-known bound for MOM of random variables \cite{2011lerasle:oliveira}. We give a proof in the Appendix for completeness. 
\begin{lemma}[Random variable]\label{lemma:count:block:variable}
Let $\alpha\in(0,1)$ and any constant $C_\alpha>0$ satisfying
$$
\frac{1}{3C_\alpha}+\frac{1}{C_\alpha^2}+\frac{\sqrt{2}}{C_\alpha^{2/3}}\le\alpha.
$$

Then, setting $r:=\sigma_X\sqrt{\frac{K}{n}}$, with probability at least $1-e^{-K/C_\alpha}$,
\begin{align}
\sum_{k=1}^K\unit_{\{
|\hat\mu_k-\esp|X||>C_\alpha\cdot r\}}\le \alpha K.
\end{align}
\end{lemma}

By now it is well-known that the previous lemma can be generalized for the Empirical Process over a general class $F$ of integrable functions $F:\mbB\rightarrow\re$ with respect to $\probn$. Define, for each $k\in[K]$ and $f\in F$, the block empirical mean
$
\hat\probn_{B_k}f:= \frac{1}{B}\sum_{\ell\in B_k}f(X_\ell).
$
For ease, we use the following notation. 
\begin{definition}
Let
\begin{align}
r_{X,n,K}(F):=
\frac{\mathscr{D}_{X,n}(F)}{n}\bigvee\sigma_{X}(F)\sqrt{\frac{K}{n}}.
\end{align}
\end{definition}

The following result is Lemma 1 in \cite{2019depersin:lecue}. 
\begin{lemma}[Empirical Process]\label{lemma:count:block:emp:process}
Let $\alpha\in(0,1)$ and any constant $C_\alpha>0$ satisfying
\begin{align}
\frac{28}{4C_\alpha}+\frac{4}{C_\alpha^2}+\frac{\sqrt{2}}{C_\alpha^{2/3}}\le\alpha.
\label{lemma:count:block:emp:process:alpha}
\end{align}

Then, setting $r:=r_{X,n,K}(F)$, with probability at least $1-e^{-K/C_\alpha}$,
\begin{align}
\sup_{f\in F}\sum_{k=1}^K\unit_{\{
|\hat\probn_{B_k}f-\probn f|>C_\alpha\cdot r\}}\le \alpha K.
\end{align}
\end{lemma}
For instance, when $F=F_1(\mbB_2)$ one has $r_{\bz,n,k}=2\sqrt{\frac{\tr(\bfSigma)}{n}}\vee\sqrt{\Vert\bfSigma\Vert\frac{K}{n}}$. When $F=F_1(\mbB_\probn)$ one has $r_{\bz,n,k}=2\sqrt{\frac{p}{n}}\vee\sqrt{\frac{K}{n}}$.

\subsection{Concentration bounds for linear regression}\label{ss:bounds:liner:regression}

For linear regression, we shall need the following bound for the Quadratic Process over the linear class. 
\begin{proposition}[Quadratic Process]\label{prop:quad:proc:block:upper:heavy}
Suppose $\bz\in\re^p$ satisfies the $L^4-L^2$ norm equivalence property with constant $L>0$. Let 
$\alpha\in(0,1)$ and constant $C_\alpha>0$ satisfying \eqref{lemma:count:block:emp:process:alpha}.
Let $C>0$ be an absolute constant in Lemma \ref{lemma:trunc:quad:rademacher:comp} in the Appendix. Set
\begin{align}
r_{n,K}:=C L^2\sqrt{\frac{p\log p}{n}}
\bigvee L^2\sqrt{\frac{K}{n}}.
\end{align}
Then, 
\begin{itemize}
\item[\rm (i)] \textbf{Upper bound}: given 
$\rho\in(0,1/2]$ and setting $C_\rho':=1+\sqrt{2/\rho}$, on an event of probability at least $1-e^{-\rho n/1.8}-e^{-K/C_\alpha}$, for all $\bu\in\mbS_\probn$, for at least $(1-(\alpha+0.75\rho))K$ of the blocks, 
\begin{align}
\frac{1}{B}\sum_{\ell\in B_k}\left(\langle\bz_\ell,\bu\rangle^2-1\right)&\le 
C_\alpha\left[
r_{n,K}\bigvee CC'_\rho\frac{p\log p}{n}
\right].
\end{align}
\item[\rm (ii)] \textbf{Lower bound}: for any $\theta>0$, on an event of probability at least $1-e^{-K/C_\alpha}$, for all $\bu\in\mbS_\probn$, for at least $(1-\alpha)K$ of the blocks, 
\begin{align}
\frac{1}{B}\sum_{\ell\in B_k}\left(
1-\frac{L^4}{\theta}
-\langle\bz_\ell,\bu\rangle^2
\right)&\le C_\alpha\left[
r_{n,K}\bigvee C\theta\frac{p\log p}{n}
\right].
\end{align}
\end{itemize} 
\end{proposition}

\begin{corollary}[Product Process]\label{cor:prod:proc:block:upper:heavy}
Given $\rho\in(0,1/2]$ and $\alpha\in(0,1)$, grant assumptions and definitions in Proposition  \ref{prop:quad:proc:block:upper:heavy}.
Set
\begin{align}
r_{n,\rho,K}:=2r_{n,K}+
CC_\alpha C_\rho'\frac{p\log p}{2n}
+2L^2\sqrt{\frac{CC_\alpha p\log p}{2n}}.
\end{align}

Then, on a event of probability at least $1-e^{-\rho n/1.8}-2e^{-K/C_\alpha}$, for all $[\bu,\bv]\in\mbB_\probn\times\mbB_\probn$, for at least $(1-(2\alpha+0.75\rho))K$ of the blocks, 
\begin{align}
\frac{1}{B}\sum_{\ell\in B_k}
\left(
\langle\bz_\ell,\bu\rangle
\langle\bz_\ell,\bv\rangle
-\langle\bu,\bv\rangle_\probn
\right)&\le r_{n,\rho,K}.
\end{align} 
\end{corollary}

We next present some bounds that are suboptimal with respect to to the confidence level. Nevertheless, they are important to the pre-processing step of linear regression; we make a remark in this regard in the following. 

\begin{proposition}\label{prop:quad:proc:block:lower:heavy}
Suppose that $\bz$ satisfies the $L^4-L^2$ norm equivalence condition with constant $L>0$.

Then, for all $k\in[K]$ and all $t\ge0$, setting
\begin{align}
r_t:= L^2\left(
7\sqrt{\frac{Kt}{n}}+4\sqrt{\frac{p}{n}}
\right), 
\end{align}
on a event of probability at least $1-2e^{-t}$, for all 
$\bv\in\re^p$, 
\begin{align}
\Vert\bv\Vert_\probn^2-\frac{1}{B}\sum_{\ell\in B_k}\langle\bz_\ell,\bv\rangle^2 &\le r_t\Vert\bv\Vert_\probn^2.
\end{align}
\end{proposition}
\begin{remark}
The above lower bound can be seen as a MOM-type lower bound for quadratic processes \cite{2016oliveira}. It has two important features. First, a direct application of Theorem 3.1 in \cite{2016oliveira} leads to a rate of the form $r_t\asymp L^2\sqrt{\frac{K(p+t)}{n}}$; this is not useful in the optimal regime when $K\ge o\vee\log(1/\delta)$ and $p\vee o\vee\log(1/\delta)\lesssim n$ we are interested. Indeed, we will only use Proposition \ref{prop:quad:proc:block:lower:heavy} with fixed confidence $t\asymp1$. Second, the above bound holds \emph{for every block} $k\in[K]$ \emph{uniformly over $\bv\in\mbS_{\probn}$}. While the rate in Proposition \ref{prop:quad:proc:block:lower:heavy} is worse than the one in Proposition \ref{prop:quad:proc:block:upper:heavy} with respect to the confidence parameter $t>0$, the uniformity on $(k,\bv)$ does not hold in item (ii) of Proposition \ref{prop:quad:proc:block:upper:heavy}. Indeed the blocks for which the lower bound in item (ii) of Proposition \ref{prop:quad:proc:block:upper:heavy} holds \emph{depend on} $\bv\in\mbS_{\probn}$. The uniformity property will be fundamental in order to show that the initialization of our algorithm with the Median-of-Least-Squares is bounded in the mentioned regime for $(n,p,o,\delta)$.
\end{remark}

The following lemma is immediate from Markov's inequality and the parallelogram law satisfied by the $\ell_2$-norm.
\begin{lemma}\label{lemma:nem:ineq:vector}
For all $k\in[K]$ and all $\delta\in(0,1)$, with probability at least $1-\delta$, 
\begin{align}
\left\Vert
\frac{1}{B}\sum_{\ell\in B_k}\bz_\ell
\right\Vert_2^2
\le\left(
\frac{1}{B}\esp\Vert\bz\Vert_2^2
\right)(\nicefrac{1}{\delta}).
\end{align}
\end{lemma}	

\subsection{Random spherical rounding}

Only within this section we assume that $\{\bz_i\}_{i=1}^m$ is a fixed (nonrandom) sequence of vectors in $\re^p$. The following proof is inspired by Proposition 1 in \cite{2019depersin:lecue}. Still, we simplify the proof and enhance the probability level significantly by using a spherical distribution instead of the Gaussian distribution.

\begin{proposition}[Spherical rounding of 
$\calM(\mbS_2)$ to $\mbS_2$]\label{prop:unif:round}
Suppose that there exist 
$\bfM\in\calM$ and $D,\sb>0$ such that
\begin{align}
\sum_{i=1}^m\unit_{\{\llangle\bz_i\bz_i^\top,\bfM\rrangle> D\}}> \sb m.
\end{align}

Let $\btheta\sim\calU(\mbS_2)$ be the uniform distribution over $\mbS_2$. Define the random vector 
$\bv_{\btheta}:=\bfM^{1/2}\btheta$ where $\bfM^{1/2}$ is the square root of $\bfM$.

\quad

Then for any $\varphi\in(0,\pi/2)$ and $\sa\in(0,1)$ satisfying
$
 \frac{2\varphi\sb}{\pi\sa}>1,
$
it holds with probability (on the randomness of 
$\btheta$) of at least 
$
\left(\frac{2\varphi\sb}{\pi}-\sa\right)^2
$,
\begin{align}
\sum_{i=1}^m\unit_{\left\{|\langle\bz_i,\bv_{\btheta}\rangle|> (\cos\varphi)\sqrt{D}\right\}}>\sa m.
\end{align}
\end{proposition}
\begin{remark}
We remark that the argument above is invariant to scaling. In particular, if one has the sane assumption of the proposition for some $\bfM\in\calM(R\mbS_2)$ and $R>0$, then for $\btheta\sim\calU(\mbS_2)$, the statement of Proposition \ref{prop:unif:round} still holds for $\bv_{\btheta}=\bfM^{1/2}\btheta\in R\mbB_2$.
\end{remark}

\section{Pre-processing \& probabilistic arguments}
\label{s:pre-processing}

\begin{tcolorbox}
\begin{setup}
We observe the corrupted data set
$\{(\tilde y_\ell,\tilde\bx_\ell)\}_{\ell=1}^{n}\cup\{(\tilde y_\ell,\tilde\bx_\ell)\}_{\ell=n+1}^{2n}$  and denote the unobserved clean data set by 
$\{(y_\ell,\bx_\ell)\}_{\ell=1}^{n}\cup\{(y_\ell,\bx_\ell)\}_{\ell=n+1}^{2n}$. Given $K\in[n]$, we further split the observed first bucket  
$\{(\tilde y_\ell,\tilde\bx_\ell)\}_{\ell=1}^{n}$ into $K$ buckets 
of same size $B:=n/K$ indexed by the partition $\bigcup_{k\in[K]}\tilde B_k^{(1)}=[n]$. Similarly, the unobserved first bucket  
$\{(y_\ell,\bx_\ell)\}_{\ell=1}^{n}$ is split into $K$ buckets of same size $B:=n/K$ indexed by the partition $\bigcup_{k\in[K]}B_k^{(1)}=[n]$. Here we assume $n$ is divisible by $K$ without loss on generality. 
\end{setup}
\end{tcolorbox}

We will use the multivariate notion of median considered by Hsu-Sabato \cite{2016hsu:sabato}. For that purpose, we introduce the following notation. 
\begin{definition}
Given $\calW:=\{\bw_1,\ldots,\bw_K\}\subset\re^p$ and $\bz\in\re^p$, define
\begin{align}
\Delta_{\calW}(\bz):=\min\left\{r\ge0:|\mbB_2(\bz,r)\cap\calW|>\frac{K}{2}\right\}. 
\end{align}
\end{definition}

\begin{algorithm}[H]
\caption{$\Pruning(\calD,K,\eta)$}
\label{algo:naive:pruning:linear-reg}
Input: sample $\{(\tilde y_\ell,\tilde\bx_\ell)\}_{\ell=1}^n\cup
\{(\tilde y_\ell,\tilde\bx_\ell)\}_{\ell=n+1}^{2n}$, number of buckets $K$ \& quantile probability $\eta\in(0,1)$. 

\begin{algorithmic}[1]
  \small
  \STATE Split the first batch $\{(\tilde y_\ell,\tilde\bx_\ell)\}_{\ell=1}^n$ into $K$ buckets/batches of equal size $B:=n/K$ indexed by the partition $\bigcup_{k\in[K]}\tilde B_k^{(1)}=[n]$.
  	\STATE For each $k\in[K]$, compute the bucket least-squares estimator
$
\tilde\bb_k := \argmin_{\bb\in\re^p}\frac{1}{B}\sum_{\ell\in \tilde B_k^{(1)}}
(\tilde y_\ell-\langle\tilde\bx_\ell,\bb\rangle)^2.
$
\STATE  Compute the Hsu-Sabato's multivariate median  $\tilde\bb^{(0)}$ of $\calW:=\{\tilde\bb_1,\ldots,\tilde\bb_K\}$, that is, set 
$
\bar k :=\argmin_{k\in[K]}\Delta_{\calW}(\tilde\bb_k),
$
and $\tilde\bb^{(0)}:=\tilde\bb_{\bar k}$.
\STATE Using the second batch $\{(\tilde y_\ell,\tilde\bx_\ell)\}_{\ell=n+1}^{2n}$, compute the order statistics 
	$
	\tilde R_1^*\le\ldots\le \tilde R_n^*
	$ 
	of the sequence
	$\{\tilde R_\ell:=|\tilde\xi_\ell(\tilde\bb^{(0)})|\vee\Vert\tilde\bx_\ell\Vert_2\}_{\ell=n+1}^{2n}$. 
\STATE Set $m:=(1-\eta)n$.
\STATE Set $\{(\tilde y_i,\tilde\bx_i)\}_{i=1}^m$ (with some abuse of notation) as the subsample obtained by removing from $\{(\tilde y_\ell,\tilde\bx_\ell)\}_{\ell=n+1}^{2n}$ the $n-m=\eta n$ ``top data points'', that is, points such that 
	$\tilde R_\ell>\tilde R^*_m$.
	\RETURN $(\tilde\bb^{(0)},\tilde R^*_m,
	\{(\tilde y_i,\tilde\bx_i)\}_{i=1}^m)$.
\end{algorithmic}
\end{algorithm}

Next we state a lemma ensuring that, with high probability,  the initialization 
$\tilde\bb^{(0)}$ lies at a constant distance to the ground truth $\bb^*$ and the pruned data set is bounded.

\begin{lemma}[Boundedness of initialization \& pruned sample]\label{lemma:high:prob:pruned:data}
Grant Assumption \ref{assump:L4:L2}. Define the quantity
$
r:=2\mu^2(\mbB_2)\sqrt{12\frac{K}{n}\tr(\bfXi)}.
$
Let $K\in[n]$ and $\eta\in(0,1/2]$. Let $m:=(1-\eta)n$. Suppose that 
\begin{align}
o<K/4,\quad
\epsilon&\le\eta/4,\quad
L^2\left(
7\sqrt{\frac{K\log(24)}{n}} + 4\sqrt{\frac{p}{n}}
\right)\le\frac{1}{2}.
\end{align}
Let $Q_\eta:=Q_{|\xi|\vee\Vert\bx\Vert_2,1-\eta/2}$.

\quad

Then on an $\{(y_\ell,\bx_\ell)\}_{\ell=1}^n\cup\{(y_\ell,\bx_\ell)\}_{\ell=n+1}^{2n}$-measurable event $\calE_0$ of probability at least $1-e^{-K/5.4}-\exp(-\eta n/1.8)$, the output 
$(\tilde\bb^{(0)},\tilde R^*_m,\{(\tilde y_\ell,\tilde\bx_\ell)\}_{\ell=1}^m)$ of $\Algorithm$ \ref{algo:naive:pruning:linear-reg} satisfies:
\begin{align}
\Vert\tilde\bb^{(0)}-\bb^*\Vert_2\le 3r
\quad\mbox{and}\quad
\tilde R^*_m=\max_{\ell\in[m]}|\tilde\xi_\ell(\tilde\bb^{(0)})|\vee\Vert\tilde\bx_\ell\Vert_2\le Q_\eta(1+3r).\tag{$\mathtt{BD}(\eta,r)$}
\label{lemma:high:prob:pruned:data:eq0}
\end{align}
\end{lemma}
\begin{proof}
\underline{\textbf{STEP 1}:} Let us define, for each $k\in[K]$, the bucket least-squares estimator 
$$
\bb_k := \argmin_{\bb\in\re^p}\frac{1}{B}\sum_{\ell\in B_k^{(1)}}
(y_\ell-\langle\bx_\ell,\bb\rangle)^2, 
$$
correspondent to the (unobserved) clean sample. We will first prove that on an $\{(y_\ell,\bx_\ell)\}_{\ell=1}^n$-measurable event of probability at least $1-\exp(-K/5.4)$,
\begin{align}
\sum_{k=1}^K\unit_{\{
\Vert\bb_k-\bb^*\Vert_2\le r
\}}\ge 3K/4.\label{lemma:high:prob:pruned:data:eq0'}
\end{align}
Assume first the above claim is correct. If that is the case, let 
$\calK$ be the number of buckets of $\{(\tilde y_\ell,\tilde\bx_\ell)\}_{\ell=1}^n$ without outliers.  Since $o<K/4$, $|\calK^c|<K/4$. We thus conclude that on the same event
\begin{align}
\sum_{k=1}^K\unit_{\left\{
\Vert\tilde\bb_k-\bb^*\Vert_2>r
\right\}}&\le
\sum_{k\in\calK}\unit_{\left\{
\Vert\bb_k-\bb^*\Vert_2>r
\right\}}
+|\calK^c|\\
&\le\sum_{k\in [K]}\unit_{\{
\Vert\bb_k-\bb^*\Vert_2>r
\}}+K/4\\
&<K/2. 
\end{align}
implying that, for $\calW:=\{\tilde\bb_1,\ldots,\tilde\bb_K\}$, 
$
\Delta_{\calW}(\bb^*)\le r.
$
This fact and a well-known property of Hsu-Sabato's multivariate median \cite{2016hsu:sabato} imply that the estimator $\tilde\bb^{(0)}:=\tilde\bb_{\bar k}$ satisfies $\Vert \tilde\bb^{(0)}-\bb^*\Vert_2\le3r$.

We now prove the claim \ref{lemma:high:prob:pruned:data:eq0}. Set 
$Q_k:=\prob(\Vert\bb_k-\bb^*\Vert_2>r)$.
If we show that, for any $k\in[K]$, $Q_k\le\eta/2$ for 
$\eta:=1/3$, then a standard argument based on Bernstein's inequality  (as in the proof of Lemma \ref{lemma:count:combinatorial:variable:v2} in the Appendix) entails the claim  \eqref{lemma:high:prob:pruned:data:eq0'}.

We next prove that $Q_k\le\eta/2$ for any $k\in[K]$. Let 
$\bfDelta_k:=\bb_k-\bb^*$. By optimality, 
\begin{align}
\frac{1}{B}\sum_{\ell\in B_k^{(1)}}\langle\bx_\ell,\bfDelta_k\rangle^2=\frac{1}{B}\sum_{\ell\in B_k^{(1)}}\xi_\ell\langle\bx_\ell,\bfDelta_k\rangle.
\end{align}
Let $\delta\in(0,1)$ to be determined later and assume that, for $r_t$ as defined in Proposition \ref{prop:quad:proc:block:lower:heavy},  $r_{\log(4/\delta)}\le1/2$. \emph{Given any} $k\in[K]$, by Proposition \ref{prop:quad:proc:block:lower:heavy},  
Lemma \ref{lemma:nem:ineq:vector} and an union bound,
on an event of probability $1-\delta$, we have 
\begin{align}
\frac{1}{2}\Vert\bfDelta_k\Vert_\Pi^2\le\frac{1}{B}\sum_{\ell\in B_k^{(1)}}\langle\bx_\ell,\bfDelta_k\rangle^2
\quad\mbox{and}\quad
\frac{1}{B}\sum_{\ell\in B_k^{(1)}}\xi_\ell\langle\bx_\ell,\bfDelta_k\rangle&\le\sqrt{\frac{K}{n}\esp\xi^2\Vert\bx\Vert_2^2\frac{2}{\delta}}\Vert\bfDelta_k\Vert_2.
\end{align}
This and the optimality condition imply 
\begin{align}
\Vert\bfDelta_k\Vert_\Pi\le2\mu(\mbB_2)\sqrt{\frac{K}{n}\esp\xi^2\Vert\bx\Vert_2^2\frac{2}{\delta}}.
\end{align}

We now take $\eta:=2\delta=1/3$ and verify that conditions of the proposition imply $r_{\log(4/\delta)}\le1/2$.

\quad

\underline{\textbf{STEP 2}:} We'll now make use of the order statistics of the unobserved sequences $Y_\ell:=|\xi_\ell|\vee\Vert\bx_\ell\Vert_2$ and $\tilde Y_\ell:=|\tilde\xi_\ell|\vee\Vert\tilde\bx_\ell\Vert_2$. Let $\eta\in(0,1/2]$ and define
$
Q_\eta:=Q_{|\xi|\vee\Vert\bx\Vert_2,1-\eta/2}.
$ 
By Lemma \ref{lemma:count:combinatorial:variable:v2}, on a $\{(y_\ell,\bx_\ell)\}_{\ell=n+1}^{2n}$-measurable event $\calE_3$ of probability at most  $1-\exp(-\eta n/1.8)$, 
\begin{align}
\sum_{\ell=n+1}^{2n}\unit_{\{|\xi_\ell|\vee\Vert\bx_\ell\Vert_2\le Q_\eta\}}> (1-0.75\eta)n.
\end{align}

We now claim that on the event $\calE_3$,
$$
\tilde Y_{(1-\eta)n}^*\le Q_\eta.
$$ 
Indeed, there are at least $(1-0.75\eta) n$ points from the $n$-sized clean sample $\{Y_\ell\}_{\ell=n+1}^{2n}$ satisfying $Y_\ell\le Q_\eta+\eta$. Since $\epsilon\le 0.25\eta$ and there at most $\epsilon n$ arbitrary outliers, $\{Y_\ell\}_{\ell=n+1}^{2n}$ has at least 
$(1-0.75\eta-\epsilon)n\ge(1-\eta) n$ 
data points satisfying $\tilde Y_\ell\le Q_\eta$. This implies the claim. 

\quad

\underline{\textbf{STEP 3}:} When pruning the second batch in $\Algorithm$  \ref{algo:naive:pruning:linear-reg}, the first batch is used only to compute
$\tilde\bb^{(0)}$ (the selected number of samples $m=(1-\eta)n$ is fixed). Using Steps 1-2, independence between $\{(y_\ell,\bx_\ell)\}_{\ell=1}^n$ and $\{(y_\ell,\bx_\ell)\}_{\ell=n+1}^{2n}$ and conditioning imply that on an event $\calE$ of probability at least $1-e^{-K/5.4}-\exp(-\eta n/1.8)$, we have
$
\tilde Y_{m}^*\le Q_\eta
$
and 
$\Vert \tilde\bb^{(0)}-\bb^*\Vert_2\le3r$.

In Step 3 we work on the event $\calE$. For all 
$\ell=n+1,\ldots,2n$,  
\begin{align}
|\tilde\xi_\ell(\tilde\bb^{(0)})|&\le|\tilde\xi_\ell|+\Vert\tilde\bx_\ell\Vert_2\Vert\tilde\bb^{(0)}-\bb^*\Vert_2.
\end{align}
Hence $\tilde R_m^*=(
|\tilde\xi_\ell(\tilde\bb^{(0)})|\vee\Vert\tilde\bx_\ell\Vert_2)_{\ell=m}^*\le Q_\eta+Q_\eta(3r)$.

We conclude the proof by summarizing the conclusions of Steps 1-3 as: on an event $\calE_0$ of probability at least $1-e^{-K/5.4}-\exp(-\eta n/1.8)$, property \ref{lemma:high:prob:pruned:data:eq0}  holds. 
\end{proof}

\begin{tcolorbox}
\begin{setup}\label{setup:pruned:sample}
We now divide the pruned data set $\{(\tilde y_\ell,\tilde\bx_\ell)\}_{\ell=1}^{m}$ outputted in $\Algorithm$ \ref{algo:naive:pruning:linear-reg} into 
$\calK$ buckets of same size 
$m/\calK=B$ indexed by the partition 
$\bigcup_{i\in[\calK]}\tilde B_i=[m]$. Similarly, the unobserved data set $\{(y_\ell,\bx_\ell)\}_{\ell=1}^{m}$ is split into $\calK$ buckets 
of same size $B$ indexed by the partition 
$\bigcup_{i\in[\calK]}B_i=[m]$. For $\eta\in(0,1/2]$,  we assume $m=(1-\eta)n$ is divisible by $\calK$ without loss on generality. 

\quad

Let $\bb\in\re^p$. With some abuse of notation, the corresponding residuals will be denoted by 
$\tilde\xi_\ell(\bb):=\tilde y_\ell-\langle\tilde\bx_\ell,\bb\rangle$ and $\xi_\ell(\bb):=y_\ell-\langle\bx_\ell,\bb\rangle$.  We also define
\begin{align}
\tilde\bz_i(\bb):=\frac{1}{B}\sum_{\ell\in \tilde B_i}\tilde\xi_\ell(\bb)\tilde\bx_\ell,
\quad\mbox{and}\quad
\bz_i(\bb):=\frac{1}{B}\sum_{\ell\in B_i}\xi_\ell(\bb)\bx_\ell.
\end{align}  
\end{setup}
\end{tcolorbox}

We now establish high-probability bounds satisfied by the pruned data set outputted by $\Algorithm$ \ref{algo:naive:pruning:linear-reg}. 
\begin{lemma}[Pruned sample: Multiplier Process at $\bb^*$]\label{lemma:high:prob:pruned:data:v1}
Grant Set-up \ref{setup:pruned:sample}. Let $(\alpha_1,C_{\alpha_1})$ and $(\alpha_2,C_{\alpha_2})$ satisfying \eqref{lemma:count:block:emp:process:alpha} respectively.  
Suppose that 
\begin{align}
o&\le(\alpha_1\vee\alpha_2)K.
\end{align}

\quad

Then, setting $r_1\ge r_{\xi\bx,n,K}(F_1(\mbB_2))$, on a 
$\{(y_\ell,\bx_\ell)\}_{\ell=1}^n\cup\{(y_\ell,\bx_\ell)\}_{\ell=n+1}^{2n}$-measurable event $\calE_1$ of  probability at least $1-e^{-K/C_{\alpha_1}}-e^{-K/C_{\alpha_2}}$, 
\begin{align}
\sup_{\bv\in\mbB_2}\sum_{i\in[\calK]}\unit_{\{|\langle\tilde\bz_i(\bb^*),\bv\rangle|\ge C_{\alpha_1} r_1\}}&\le4\alpha_1\calK,
\tag{$\mathtt{MP1}(\alpha_1,r_1)$}\label{lemma:high:prob:pruned:data':eq1}\\ 
\sup_{\bv\in\mbB_2}\sum_{i\in[\calK]}\unit_{\{\langle\tilde\bz_i(\bb^*),\bv\rangle\ge C_{\alpha_2}r_1\}}&\le 4\alpha_2\calK. 
\tag{$\mathtt{MP2}(\alpha_2,r_1)$}
\label{lemma:high:prob:pruned:data':eq2}
\end{align}
\end{lemma}
\begin{proof}
In the proof we will only use that $\Algorithm$ \ref{algo:naive:pruning:linear-reg} removes 
$\eta n$ data points from the second batch. In the following, let $S$ be the index set of buckets without outliers in the pruned sample $\{(\tilde y_\ell,\tilde\bx_\ell)\}_{\ell=1}^m$. 

From Lemma \ref{lemma:count:block:emp:process}, on an  event 
$\calE_1'$ of probability at least $1-e^{-K/C_{\alpha_1}}$,  
\begin{align}
\sup_{\bv\in\mbB_2}\sum_{i\in[K]}\unit_{\{|\langle\bz_i(\bb^*),\bv\rangle|\ge C_{\alpha_1} r_1\}}\le
\alpha_1 K. 
\end{align}
Similarly, on an event $\calE_2'$ of probability at least $1-e^{-K/C_{\alpha_2}}$, 
\begin{align}
\sup_{\bv\in\mbB_2}\sum_{i\in[K]}\unit_{\{\langle\bz_i(\bb^*),\bv\rangle\ge C_{\alpha_2} r_1\}}\le \alpha_2 K.
\end{align}
We now work on the event 
$\calE_1'\cap\calE_2'$ of probability at least 
$1-e^{-K/C_{\alpha_1}}-e^{-K/C_{\alpha_2}}$.  

By assumption $m=(1-\eta)n$, $|[m]\setminus S|\le o$ and 
$o\le\alpha_1 K$. Thus, 
\begin{align}
\sup_{\bv\in\mbB_2}\sum_{i\in[\calK]}\unit_{\{|\langle\tilde\bz_i(\bb^*),\bv\rangle|\ge C_{\alpha_1} r_1\}}&\le
\sup_{\bv\in\mbB_2}\sum_{i\in S}\unit_{\{|\langle\bz_i(\bb^*),\bv\rangle|\ge C_{\alpha_1} r_1\}}+o\\
&\le \sup_{\bv\in\mbB_2}\sum_{i\in[K]}\unit_{\{|\langle\bz_i(\bb^*),\bv\rangle|\ge C_{\alpha_1} r_1\}}+o\\
&\le2\alpha_1 K=2\alpha_1\frac{\calK}{(1-\eta)}\le4\alpha_1\calK,
\end{align}
implying \ref{lemma:high:prob:pruned:data':eq1}.  Using that $o\le \alpha_2K$, very similar arguments as above imply \ref{lemma:high:prob:pruned:data':eq2}. 
\end{proof}

The next lemma follows from very similar arguments of Lemma \ref{lemma:high:prob:pruned:data:v1} but using  the lower bound of Proposition \ref{prop:quad:proc:block:upper:heavy} optimized at $\theta>0$. We omit the proof.
\begin{lemma}[Pruned sample: Quadratic Process Lower bound]\label{lemma:high:prob:pruned:data:v2}
Grant Assumption \ref{assump:L4:L2} and Set-up \ref{setup:pruned:sample}. Let $\alpha_3\in(0,1)$ and $C_{\alpha_3}>0$ satisfying
\eqref{lemma:count:block:emp:process:alpha} and $C>0$ be an absolute constant as in Lemma \ref{lemma:trunc:quad:rademacher:comp}. Let $r_{n,K}$ as in Proposition \ref{prop:quad:proc:block:upper:heavy}. Suppose that 
\begin{align}
o&\le\alpha_3 K,\\
C_{\alpha_3} r_{n,K}+2L^2\sqrt{\frac{CC_{\alpha_3} p\log p}{n}}&\le\frac{1}{2}.
\end{align}

\quad

Then on a $\{(y_\ell,\bx_\ell)\}_{\ell=1}^n\cup\{(y_\ell,\bx_\ell)\}_{\ell=n+1}^{2n}$-measurable event $\calE_2$ of probability at least $1-e^{-K/C_{\alpha_3}}$, given $\bv\in\mbB_\Pi$, for at least $(1-4\alpha_3)\calK$ of the buckets,
\begin{align}
\frac{1}{B}\sum_{\ell\in \tilde B_i}\langle\tilde\bx_\ell,\bv\rangle^2\ge\frac{1}{2}
\tag{$\mathtt{QP}l(\alpha_3)$}
\label{lemma:high:prob:pruned:data':eq3}
\end{align}
\end{lemma}

Finally, the next lemma follows from very similar arguments to Lemma \ref{lemma:high:prob:pruned:data:v1}, but using Corollary \ref{cor:prod:proc:block:upper:heavy}. We also omit the proof. 
\begin{lemma}[Pruned sample: Product Process Upper bound]\label{lemma:high:prob:pruned:data:v3}
Grant Assumption \ref{assump:L4:L2} and Set-up \ref{setup:pruned:sample}. Let $\rho\in(0,1/2]$, 
$\alpha_4\in(0,1)$ and constant $C_{\alpha_4}>0$ satisfying  \eqref{lemma:count:block:emp:process:alpha}. 
Suppose that 
\begin{align}
o&\le\alpha_4K.
\end{align}
Let $C_\rho':=1+\sqrt{2/\rho}$ and $C>0$ be an absolute constant in Lemma \ref{lemma:trunc:quad:rademacher:comp}. Let $r_{n,K}$ as in Proposition \ref{prop:quad:proc:block:upper:heavy}.

Then, setting
\begin{align}
r_2:=2C_{\alpha_4}r_{n,K}+
CC_{\alpha_4} C_\rho'\frac{p\log p}{2n}
+2L^2\sqrt{\frac{CC_{\alpha_4} p\log p}{2n}}, 
\label{lemma:high:prob:pruned:data:v3:r2}
\end{align}
on a $\{(y_\ell,\bx_\ell)\}_{\ell=1}^n\cup\{(y_\ell,\bx_\ell)\}_{\ell=n+1}^{2n}$-measurable event $\calE_3$ of probability at least $1-e^{-\rho n/1.8}-e^{-K/C_{\alpha_4}}$, given $[\bu,\bv]\in\mbB_\Pi\times\mbB_\Pi$, for at least 
$
[1-(3\alpha_4+0.75\rho)]\calK
$
buckets,
\begin{align}
\frac{1}{B}\sum_{\ell\in \tilde B_i}\left(\langle\tilde\bx_\ell,\bu\rangle\langle\tilde\bx_\ell,\bv\rangle-\langle\bu,\bv\rangle_\Pi\right)\le r_2.
\tag{$\mathtt{PP}u(\alpha_4,\rho,r_2)$}
\label{lemma:high:prob:pruned:data':eq4}
\end{align} 
\end{lemma}

\section{Pre-algorithms \& deterministic arguments}

\begin{tcolorbox}
In all this section, we work within the Set-up \ref{setup:pruned:sample} and on the event 
$\calE_0\cap\calE_1\cap\calE_2\cap\calE_3$ where, for the pruned sample $\{(\tilde y_\ell,\tilde\bx_\ell)\}_{\ell\in[m]}$ and initialization $\tilde\bb^{(0)}$, the boundedness property \ref{lemma:high:prob:pruned:data:eq0} and the uniform properties
\ref{lemma:high:prob:pruned:data':eq1}, 
\ref{lemma:high:prob:pruned:data':eq2},  \ref{lemma:high:prob:pruned:data':eq3} and \ref{lemma:high:prob:pruned:data':eq4} all hold. The arguments in this section are purely deterministic.  
\end{tcolorbox}

\subsection{The benchmark combinatorial problem}\label{ss:benchmark:problem}

The negation of property \ref{lemma:high:prob:pruned:data':eq1} leads to the following benchmark problem. Given data set $\tilde\mbZ:=\{\tilde\bz_i\}_{i=1}^{\calK}$, $k\in[\calK]$ and $R>0$, denote by 
$\RFHP(\tilde\mbZ,k,R)$ the problem 
\begin{maxi}
  {(\theta,\bv,\bq)\in\re\times R\mbB_2\times\re^\calK}{\theta}{}{}{}
  \addConstraint{\bq_i|\langle\tilde\bz_i,\bv\rangle|}{\ge\bq_i\theta,}{i=0,\ldots \calK}
  \addConstraint{\sum_{i=1}^\calK\bq_i}{> \calK-k}
  \addConstraint{\bq_i}{\in\{0,1\},\quad}{i=0,\ldots \calK.}
 \end{maxi}
It turns out that the above problem can be solved approximately using the Multiplicative Update Algorithm together with a rounding algorithm \cite{2012karnin:liberty:lovett:schwartz:weinstein,2020lei:luh:venkat:zhang}. We highlight the following immediate fact.
\begin{fact}
$\RFHP(\tilde\mbZ,k,R)$ is feasible iff there exists 
$(\theta,\bv)\in\re\times R\mbB_2$ satisfying the  property \ref{equation:emp:quantile:neg:P1}$(\tilde\mbZ,\theta,\bv,k)$ defined as:
\begin{align}
\sum_{i=1}^{\calK}\unit_{\{|\langle\tilde\bz_i,\bv\rangle|>\theta\}}> \calK-k.
\tag{$\neg\mathtt{P1}$}
\label{equation:emp:quantile:neg:P1}
\end{align}
\end{fact} 

The geometrical interpretation of problem 
$\RFHP(\tilde\mbZ,k,R)$ is of finding two symmetrical hyperplanes with orthogonal direction $\bv$ from the origin maximizing the margin across at least $\calK-k$ data points. If we can approximately solve it, using a simple order statistics argument, one can obtain for some $k'>k$, a relaxed solution of the \emph{one-sided} problem:
\begin{maxi}
  {(\theta,\bv,\bq)\in\re\times R\mbB_2\times\re^\calK}{\theta}{}{}{}
  \addConstraint{\bq_i\langle\tilde\bz_i,\bv\rangle}{\ge\bq_i\theta,}{i=0,\ldots \calK}
  \addConstraint{\sum_{i=1}^\calK\bq_i}{> \calK-k'}
  \addConstraint{\bq_i}{\in\{0,1\},\quad}{i=0,\ldots \calK.}
 \end{maxi}
This problem identifies the hyperplane with positive angle across most data points. 
The above benchmark problem is the negation of property  \ref{lemma:high:prob:pruned:data':eq2} (for normalized length $R=1$).

\subsection{Multiplicative Weights Update \& Spherical Rounding}\label{ss:mw:round}

The main purpose of this section is to state and analyze $\Algorithm$ \ref{algo:MW} constructed to obtain an relaxed solution of $\RFHP(\tilde\mbZ,k,R)$ given data set 
$\tilde\mbZ:=\{\tilde\bz_i\}_{i=1}^{\calK}$,  $k\in[\calK]$ and $R>0$. All we need to assume are that 
$\RFHP(\tilde\mbZ,k,R)$ is feasible and we know a lower bound $\tilde r_1>0$ on the optimal margin.\footnote{Later we show that we can adapt to this parameter.} Precisely, we assume the optimal solution $(\theta_{\tilde\mbZ,k,R},\bv_{\tilde\mbZ,k,R})$ satisfies property \ref{equation:emp:quantile:neg:P1}$(\tilde\mbZ,\theta_{\tilde\mbZ,k,R},\bv_{\tilde\mbZ,k,R})$ and we know $\tilde r_1$ such that $\theta_{\tilde\mbZ,k,R}\ge\tilde r_1$. $\Algorithm$ \ref{algo:MW} then outputs a margin-direction pair 
$(\hat\theta,\bv)\in\re\times R\mbB_2$ satisfying \ref{equation:emp:quantile:neg:P1}$(\tilde\mbZ,\hat\theta,\bv,k')$ for some $k'>k$ and 
$\hat\theta\ge\sa\theta_{\tilde\mbZ,k,R}$ for some $\sa\in(0,1)$.

Different building blocks are needed to analyze $\Algorithm$ \ref{algo:MW}. We run MWU with cost associated with the spectrum of the data matrix with points $\{\tilde\bz_i\}_{i=1}^{\calK}$. Its output 
$\bfM\in\calM(R\mbS_2)$ is rounded into a direction $\bv\in R\mbB_2$ using a spherical distribution via $\Algorithm$ \ref{algo:ROUND:spherical:order}.  
In the same algorithm, a simple order statistics is used to obtain the margin-direction pair 
$(\hat\theta,\hat\bv)\in\re\times R\mbB_2$ where $\hat\bv\in\{\bv,-\bv\}$. For specific applications, either linear regression or mean estimation, we show in later sections that $\hat\bv$ satisfies ``good descent'' properties.

\begin{algorithm}[H]
\caption{$\MW(\calD,U,S,k,k',R,\tilde r_1)$}
\label{algo:MW}
\textbf{Input}: data points $\calD:=\{\tilde\bz_i\}_{i=1}^{\calK}$, upper bound $U>0$, simulation sample size $S\in\mathbb{N}$, lower bound $\tilde r_1>0$,  $k,k'\in[\calK]$ and $R>0$. 

\textbf{Output}: margin-direction pair 
$(\hat\theta,\hat\bv)\in\re\times R\mbB_2$.
 
\quad 
 
\begin{algorithmic}[1]
  \small
  \STATE Set 
$
T:=\left\lfloor
\frac{40U(1-2k/\calK)}{\tilde r_1^2}
\right\rfloor.
$
 	\STATE Set $w^{(1)}:=\frac{1}{\calK}\unit_\calK$.
	\FOR{$t\in[T]$} 
	\STATE Set matrix $\bfZ^{(t)}\in\re^{[\calK]\times p}$ with $i$th row $\bfZ^{(t)}_i:=\sqrt{w^{(t)}_i}\tilde\bz_i^\top$. Set 
$\bfM^{(t)}:=(\bfZ^{(t)})^\top\bfZ^{(t)}$.
	\STATE Solve 
$
\bv^{(t)}:=\argmax_{\bv\in R\mbS_2} \bv^\top\bfM^{(t)}\bv.
$
	\STATE For each $i\in[\calK]$, set 
	$\tau_i^{(t)}:=\langle\tilde\bz_i,\bv^{(t)}\rangle^2$.
	\STATE Set $\hat w_i^{(t+1)}\leftarrow w_i^{(t)}\left(1-\frac{1}{2U}\tau_i^{(t)}\right)$ for each $i\in[\calK]$.
	\STATE Normalize by setting $\hat w_i^{(t+1)}\leftarrow \hat w_i^{(t+1)}/\sum_{i=1}^\calK\hat w_i^{(t+1)}$ for each $i\in[\calK]$.
	\STATE Compute the \emph{Kullback-Leibler projection}
	$$	w^{(t+1)}:=\argmin_{w\in\Delta_{\calK,\calK-2k}}\KL(w\Vert \hat w^{(t+1)}).
	$$
	\ENDFOR
	\STATE Set $\bfM:=\frac{1}{T}\sum_{t=1}^T\bv^{(t)}(\bv^{(t)})^\top$.
	\STATE $(\hat\theta,\hat\bv)\leftarrow\Round\left(\calD,\bfM,S,k,k'\right)$.
 	\RETURN $(\hat\theta,\hat\bv)$.
\end{algorithmic}
\end{algorithm}

\begin{algorithm}[H]
\caption{$\Round(\calD,\bfM,S,k,k')$}
\label{algo:ROUND:spherical:order}
\textbf{Input}: data points $\calD:=\{\tilde\bz_i\}_{i=1}^{\calK}$, symmetric matrix $\bfM\in\re^{p\times p}$, simulation sample size $S\in\mathbb{N}$,  
$k,k'\in[\calK]$.

\quad

\textbf{Output}: margin-direction pair 
$(\hat\theta,\hat\bv)$.
 
\begin{algorithmic}[1]
  \small
 	\STATE Compute square-root $\bfM^{1/2}$. 
	\STATE Sample independently $\{\btheta_\ell\}_{\ell=1}^S$ from the uniform distribution over the unit sphere $\mbS_2$ on $\re^p$.
	\STATE Set $\bv_\ell:=\bfM^{1/2}\btheta_\ell$ and $Z_i(\ell):=|\langle\tilde\bz_i,\bv_\ell\rangle|$ for all $\ell\in[S]$ and $i\in[\calK]$.
	\STATE For each $\ell\in[S]$ compute the order statistics: $Z_1^\sharp(\ell)\ge\ldots\ge Z_\calK^\sharp(\ell)$. 
	\STATE Set $\ell_*\in\argmax_{\ell\in[S]}Z^\sharp_{\calK-k'}(\ell)$, 
	$\bv:=\bv_{\ell_*}$ and 
	$\hat\theta:=Z^\sharp_{\calK-k'}(\ell_*)$.
 	\STATE Compute the order statistics $W_1^\sharp\ge\ldots\ge W_\calK^\sharp$ of the sequence $W_i:=\langle\tilde\bz_i,\bv\rangle$ for $i\in[\calK]$. 
 	\IF{$W^\sharp_{\calK-k-k'}>\hat\theta$}
 	\STATE Set $\hat\bv:=\bv$.
 	\ELSE
 	\STATE Set $\hat\bv:=-\bv$.
 	\ENDIF
 	\RETURN $(\hat\theta,\hat\bv)$.
\end{algorithmic}
\end{algorithm}

We recall a online regret bound for the multiplicative weight algorithm with restricted distributions \cite{2012arora:hazan:kale,2012karnin:liberty:lovett:schwartz:weinstein}. We stated for the particular case of Algorithm \ref{algo:MW}.
\begin{lemma}[Theorem 2.4 in \cite{2012arora:hazan:kale}]\label{lemma:regret:bound}
Set 
$
w^{(\ell)}:=[w^{(\ell)}_1 \cdots w^{(\ell)}_\calK]^\top
$
and 
$
\tau^{(\ell)}:=[\tau^{(\ell)}_1 \cdots \tau^{(\ell)}_\calK]^\top.
$
Suppose that 
\begin{align}
\sup_{t\ge1}\sup_{\ell\in[t]}\sup_{i\in[\calK]}\tau_i^{(\ell)}\le U.
\end{align}

Then, for all $t\ge1$ and for all $w\in\Delta_{\calK,\calK-2k}$,
\begin{align}
\frac{1}{t}\sum_{\ell=1}^t\langle w^{(\ell)},\tau^{(\ell)}\rangle
\le\frac{1.5}{t}\sum_{\ell=1}^t\langle w,\tau^{(\ell)}\rangle
+\frac{2U}{t}\KL(w\Vert w^{(1)}).
\label{lemma:regret:bound:equation}
\end{align}
\end{lemma}

\begin{lemma}\label{lemma:MWU:from:below:linear:reg}
Suppose that 
\begin{itemize}
\item[\rm (i)] $\max_{i\in[\calK]}\Vert\tilde\bz_i\Vert_2^2\le U$ for some $U>0$,
\item[\rm (ii)] $\RFHP(\tilde\mbZ,k,R)$ is feasible with optimal solution $(\theta_{\tilde\mbZ,k,R},\bv_{\tilde\mbZ,k,R})$.
\item[\rm (iii)] $\theta_{\tilde\mbZ,k,R}\ge\tilde r_1$, for some $\tilde r_1>0$.  
\end{itemize}

Instantiate $\Algorithm$ \ref{algo:MW} with inputs 
$\calD=\tilde\mbZ$ and $U=\max_{i\in[\calK]}\Vert\tilde\bz_i\Vert_2^2$.

Then, for $D:=\theta_{\tilde\mbZ,k,R}^2/6$, 
$\Algorithm$ \ref{algo:MW}   produces matrix 
$\bfM\in\calM(R\mbS_2)$ satisfying the quantile property 
\ref{equation:emp:quantile:neg:P:M}$(\tilde\mbZ,D,\bfM,2k)$ defined by
\begin{align}
\sum_{i=1}^{\calK}\unit_{
\left\{
\llangle\tilde\bz_i\tilde\bz_i^\top,\bfM\rrangle>D 
\right\}}
>\calK-2k.
\tag{$\neg\mathtt{P1'}$}
\label{equation:emp:quantile:neg:P:M}
\end{align}
\end{lemma}
\begin{proof}
By (i) $\tau_i^{(\ell)}\le U<\infty$ for all $i,\ell$. Thus, we may apply recursion \eqref{lemma:regret:bound:equation} in Lemma \ref{lemma:regret:bound}. Let us denote locally $(\theta,\bv):=(\theta_{\tilde\mbZ,k,R},\bv_{\tilde\mbZ,k,R})$. By item (iii), 
$\theta\ge\tilde r_1>0$.

\underline{LOWER BOUND}: Define the index set 
$\calI:=\{i\in[\calK]:
|\langle\tilde\bz_i,\bv\rangle|>\theta\}$. Item (ii) ensures that $|\calI|\ge \calK-k$ and, since $w^{(\ell)}\in\Delta_{\calK,\calK-2k}$ for all $\ell\in[t]$, we have 
$
\sum_{i\in\calI}w_i^{(\ell)}\ge1-\frac{k}{2k}=0.5.
$
We conclude that, for any $t\ge1$,
\begin{align}
\frac{1}{t}\sum_{\ell=1}^t\langle w^{(\ell)},\tau^{(\ell)}\rangle&=
\frac{1}{t}\sum_{\ell=1}^t\sum_{i=1}^\calK w_i^{(\ell)}\langle\tilde\bz_i,\bv^{(\ell)}\rangle^2\\
&\ge\frac{1}{t}\sum_{\ell=1}^t\sum_{i=1}^\calK w_i^{(\ell)}\langle\tilde\bz_i,\bv\rangle^2\\
&\ge\frac{1}{t}\sum_{\ell=1}^t\sum_{i\in\calI}w_i^{(\ell)}\langle\tilde\bz_i,\bv\rangle^2\ge0.5\theta^2,
\label{lemma:MWU:lower}
\end{align}
where the first inequality uses that
$\bv^{(\ell)}\in\argmax_{\bu\in R\mbS_2}\bu^\top\bfM^{(\ell)}\bu$.

\underline{UPPER BOUND}: We show the following claim: for any $t\ge1$, there exists $\calI_t\subset[\calK]$ of size $|\calI_t|\ge \calK-2k$ such that for any $i\in \calI_t$, there exists $w:=w[t,i]\in \Delta_{\calK,\calK-2k}$ such that 
\begin{align}
\sum_{\ell=1}^t\langle w,\tau^{(\ell)}\rangle
\le \sum_{\ell=1}^t\tau^{(\ell)}_i
=\sum_{\ell=1}^t\langle\tilde\bz_i,\bv^{(\ell)}\rangle^2.\label{lemma:MWU:upper1}
\end{align}
Indeed, fix $t\ge1$ and set
$
\alpha_{j}^{(t)}:=\sum_{\ell=1}^t\tau^{(\ell)}_j
$
for all $j\in[\calK]$. Let $\calI_t:=\{i\in[\calK]:\alpha_{i}^{(t)}\ge(\alpha_{\calK-2k}^{(t)})^\sharp\}$. By construction, $|\calI_t|= \calK-2k$. Fix $i\in\calI_t$ and let $w$ be the uniform distribution over $J_i:=\{j\in[\calK]:\alpha_{j}^{(t)}\le\alpha_{i}^{(t)}\}$. Since $i\in\calI_t$ we have that $|J_i|\ge2k$ and hence $w\in\Delta_{\calK,\calK-2k}$. Finally, by construction, 
\begin{align}
\sum_{\ell=1}^t\langle w,\tau^{(\ell)}\rangle
=\sum_{j=1}^\calK w_j\alpha_{j}^{(t)}
=\sum_{j=1}^\calK w_j(\alpha_{j}^{(t)})^\sharp
\le \sum_{j\in J_i}^\calK w_j\alpha_{i}^{(t)} =\alpha_{i}^{(t)}=\sum_{\ell=1}^t\tau^{(\ell)}_i,
\end{align}
implying claim \eqref{lemma:MWU:upper1}. 

Given $t\ge1$, let $\calI_t$ as given by claim 
\eqref{lemma:MWU:upper1}. From Lemma \ref{lemma:KL} in the Appendix, for any $i\in\calI_t$ and some $w[t,i]\in\Delta_{\calK,\calK-2\kappa}$ as in claim \eqref{lemma:MWU:upper1} we have 
\begin{align}
\frac{2U}{t}\KL(w[t,i]\Vert w^{(1)})\le\frac{10U(1-2k/\calK)}{t}.
\label{lemma:MWU:upper2}
\end{align}

\underline{RAPPING UP}: Joining the bounds in \eqref{lemma:MWU:lower}, \eqref{lemma:MWU:upper1}, \eqref{lemma:MWU:upper2} with the regret bound \eqref{lemma:regret:bound:equation}, we conclude that: for all $t\ge1$ and all $i\in\calI_t$, 
\begin{align}
\frac{\theta^2}{3}-\frac{20U(1-2k/\calK)}{3t}\le\frac{1}{t}\sum_{\ell=1}^t\langle\tilde\bz_i,\bv^{(\ell)}\rangle^2=\frac{\alpha_{i}^{(t)}}{t}.
\end{align}

Recall that $\theta\ge \tilde r_1$. Thus, the LHS of the previous displayed inequality is at least
$
\frac{\theta^2}{6}
$
after
$
T:=\left\lfloor
\frac{40U(1-2k/\calK)}{\tilde r_1^2}
\right\rfloor
$
iterations. Hence, for all $i\in\calI_T$, 
\begin{align}
\frac{\theta^2}{6}\le \frac{\alpha_{i}^{(T)}}{T}=\frac{1}{T}\sum_{\ell=1}^T\langle\tilde\bz_i,\bv^{(\ell)}\rangle^2
=\llangle\tilde\bz_i\tilde\bz_i^\top,\bfM\rrangle,
\end{align} 
where $\bfM:=\frac{1}{T}\sum_{\ell=1}^T\bv^{(\ell)}(\bv^{(\ell)})^\top$. As 
$|\calI_T|=\calK-2k$ and $\bfM\in\calM(R\mbS_2)$, the claim is proved.
\end{proof}

\begin{lemma}[Random spherical rounding: boosted confidence]\label{lemma:MW:from:below:lin:reg:boost}
Grant assumptions in Lemma \ref{lemma:MWU:from:below:linear:reg}. 

Let $\varphi\in(0,\pi/2)$ and 
$k'\in[\calK]$ such that $\mathfrak{p}\in(0,1)$ where
$$
\mathfrak{p}:=
\frac{2\varphi}{\pi}\left(1-2(\nicefrac{k}{\calK})\right)-\left(1-(\nicefrac{k'}{\calK})\right).
$$

Then  with probability (on the randomness of $\{\btheta_\ell\}_{\ell\in[S]}$) of at least $1-e^{-\frac{\mathfrak{p}^2S}{7.72}}$, 
$\Algorithm$  \ref{algo:ROUND:spherical:order} (inputted in $\Algorithm$ \ref{algo:MW}) produces margin $\hat\theta\ge(\cos\varphi)\theta_{\tilde\mbZ,k,R}/\sqrt{6}$ and direction 
$\bv\in R\mbB_2$ satisfying property \ref{equation:emp:quantile:neg:P1}$(\tilde\mbZ,\hat\theta,\bv,k')$, that is, 
\begin{align}
\sum_{i=1}^{\calK}\unit_{\left\{|\langle\tilde\bz_i,\bv\rangle|>\hat\theta
\right\}}> \calK-k'.
\end{align} 
\end{lemma}
\begin{proof}
We use the local notation
$\theta:=(\cos\varphi)\theta_{\tilde\mbZ,k,R}/\sqrt{6}$. Let $\btheta$ be an random variable with the uniform distribution over 
$\mbS_2$ and $\bv_{\btheta}:=\bfM^{1/2}\btheta$. Define $Z_i:=|\langle\tilde\bz_i,\bv_{\btheta}\rangle|$ and $Z:=Z_{\calK-k'}^\sharp$. Recall the notations $Z_i(\ell)=|\langle\tilde\bz_i,\bv_{\ell}\rangle|$ and $Z(\ell):=Z_{\calK-k'}^\sharp(\ell)$ for 
$\ell\in[S]$ in $\Algorithm$  \ref{algo:ROUND:spherical:order}.

By the one-sided Bernstein's inequality, we have that, for all $t\ge0$, with probability at least $1-e^{-t}$,
\begin{align}
\frac{1}{S}\sum_{\ell=1}^S
\unit_{\{Z(\ell)>\theta\}}
\ge\prob(Z>\theta)-\sigma\sqrt{\frac{2t}{S}}-\frac{t}{3S},
\end{align}
where 
$
\sigma^2:=\esp(\unit_{\{Z>\theta\}}-\prob(Z>\theta))^2\le\prob(Z>\theta).
$
As $Z=Z^\sharp_{\calK-k'}$, 
\begin{align}
\prob(Z>\theta)=\prob\left(
\sum_{i=1}^\calK\unit_{\left\{|\langle\tilde\bz_i,\bv_{\btheta}\rangle|> \theta\right\}}>\calK-k'
\right)\ge\mathfrak{p}^2.
\end{align}
In the last inequality, we used  Proposition \ref{prop:unif:round} and  property \ref{equation:emp:quantile:neg:P:M}$(\tilde\mbZ,D,\bfM,2k)$ with $D:=\theta_{\tilde\mbZ,k,R}^2/6$ ensured by Lemma \ref{lemma:MWU:from:below:linear:reg}. 

Setting $t:=c^2\mathfrak{p}^2S$ we thus conclude that with probability at least $1-e^{-c^2\mathfrak{p}^2S}$, 
\begin{align}
\max_{\ell\in[S]}\unit_{\{Z(\ell)>\theta\}}\ge\mathfrak{p}^2\left(1-c\sqrt{2\mathfrak{p}^2}-\frac{c^2}{3}\right)\ge \mathfrak{p}^2
\left(1-c\sqrt{2}-\frac{c^2}{3}\right)\ge\mathfrak{p}^2/2,
\end{align}
for small enough $c\in(0,1)$. The choice $c=0.36$ suffices. The rest of the proof will happen in this event. 

By the previous display, one has 
$\hat\theta:=Z^\sharp_{\calK-k'}(\ell_*)=Z(\ell_*)>\theta$ where 
$\ell_*\in\argmax_{\ell\in[S]}Z(\ell)$. This shows that $(\hat\theta,\bv)=(\hat\theta,\bv_{\ell_*})$, as returned by $\Algorithm$  \ref{algo:ROUND:spherical:order} (inputted in $\Algorithm$ \ref{algo:MW}), satisfies the claim of the lemma.
\end{proof}

\subsection{Solving the outer loop combinatorial problem}\label{ss:largest:residual}

We next set $\tilde\mbZ:=\{\tilde\bz_i(\bb)\}_{i=1}^\calK$ for some $\bb\neq\bb^*$. For reasons to be made clearer later, we will tune our algorithm with $R:=\mu^2(\mbB_2)$ and omit the dependence on $R$ for convenience. The next lemma formalizes the fact that under the structural conditions \ref{lemma:high:prob:pruned:data':eq1} and \ref{lemma:high:prob:pruned:data':eq2}, problem
$\RFHP(\{\tilde\bz_i(\bb)\}_{i=1}^\calK,k)$ is feasible when 
$\bb\neq\bb^*$ for small enough $k$.

\begin{lemma}[Two-sided feasibility \& margin-\emph{distance} \emph{lower} bound]\label{lemma:RFHP:feasibility}
Suppose that 
\begin{itemize}
    \item[\rm (i)] \ref{lemma:high:prob:pruned:data':eq1} holds for some $\alpha_1\in(0,1)$. Let $k:=4\alpha_1\calK$.
    \item[\rm (ii)] \ref{lemma:high:prob:pruned:data':eq3} holds for some $\alpha_3\in(0,1)$.
\end{itemize}

Let $\bb\neq\bb^*$ satisfying: 
\begin{itemize}
    \item[\rm (iv)] For some 
$(\theta,\bv,k')\in\re\times\mu^2(\mbB_2)\mbS_2\times[\calK]$, property \ref{equation:emp:quantile:neg:P1}$(\{\tilde\bz_i(\bb)\}_{i=1}^\calK,\theta,\bv,k')$ holds, i.e., 
\begin{align}
\sum_{i=1}^{\calK}\unit_{\{|\langle\tilde\bz_i(\bb),\bv\rangle|>\theta\}}> \calK-k'.
\end{align}
In particular, $\RFHP(\tilde\bz_i(\bb)\}_{i=1}^m,k')$ is feasible with optimal value, say, 
$\vartheta$.
\item[\rm (v)] 
$4(\alpha_1+\alpha_3)\calK<\calK$.
\item[\rm (vi)] For some $\sa\in(0,1]$, 
$\theta\ge\sa\vartheta$.
\end{itemize}
Then 
\begin{itemize}
\item[\rm (a)] 
$
\sa\left[(\nicefrac{1}{2})\Vert\bb-\bb^*\Vert_2-C_{\alpha_1}\mu^2(\mbB_2)r_1\right]\le \theta.
\label{lemma:condition:RFHP:step:dir:a}
$
\end{itemize}

In particular, for any $\bb\neq\bb^*$, 
$\RFHP(\{\tilde\bz_i(\bb)\}_{i=1}^\calK,k)$ is feasible and, its optimal solution, denoted as 
$(\theta_{\bb,k},\bv_{\bb,k})$, satisfies \ref{equation:emp:quantile:neg:P1}$(\{\tilde\bz_i(\bb)\}_{i=1}^\calK,\theta_{\bb,k},\bv_{\bb,k},k)$ with margin satisfying
$$
\theta_{\bb,k}\ge-C_{\alpha_1}\mu^2(\mbB_2)r_1+(\nicefrac{1}{2})\Vert\bb^*-\bb\Vert_2.
$$
\end{lemma}
\begin{proof}
For simplicity we give a proof for $(\theta_{\bb,k},\bv_{\bb,k})$. The proof is the same for any $(\theta,\bv,k')$ satisfying conditions (iv)-(vi). 

\underline{STEP 1}: An upper bound on the optimal value is trivial: for any $(\theta,\bv,\bq)$ satisfying the constraints of $\RFHP(\{\tilde\bz_i(\bb)\}_{i=1}^\calK,k)$ it follows from Cauchy-Schwarz that
$\theta_{\bb,k}\le\max_{i\in[\calK]}\Vert\tilde\bz_i(\bb)\Vert_2<\infty$. 

\underline{STEP 2:} setting $r_1:=r_{\xi\bx,n,K}(F_1(\mbB_2))$, we now prove the lower bound 
$
\theta_{\bb,k}\ge(\nicefrac{1}{2})\Vert\bb-\bb^*\Vert_2-C_{\alpha_1}\mu^2(\mbB_2)r_1. 
$
\ref{lemma:high:prob:pruned:data':eq1} applied to the vector 
$\bv:=\nicefrac{\mu^2(\mbB_2)(\bb^*-\bb)}{\Vert\bb^*-\bb\Vert_2}$ implies that
$
\langle\tilde\bz_i(\bb^*),\bv\rangle\ge-C_{\alpha_1}\mu^2(\mbB_2)r_1
$
for more than $\calK-k$ buckets $i$'s, for which 
\begin{align}
|\langle\tilde\bz_i(\bb),\bv\rangle|\ge\langle\tilde\bz_i(\bb),\bv\rangle
&=\langle\tilde\bz_i(\bb^*),\bv\rangle
+\frac{1}{B}\sum_{\ell\in\tilde B_i}\langle\tilde\bx_\ell,\bb^*-\bb\rangle\langle\tilde\bx_\ell,\bv\rangle\\
&\ge-C_{\alpha_1}\mu^2(\mbB_2)r_1+\frac{1}{\mu^2(\mbB_2)}\Vert\bb^*-\bb\Vert_2\frac{1}{B}\sum_{\ell\in\tilde B_i}\langle\tilde\bx_\ell,\bv\rangle^2.
\end{align}
\ref{lemma:high:prob:pruned:data':eq3} implies that
\begin{align}
\frac{1}{B}\sum_{\ell\in\tilde B_i}\langle\tilde\bx_\ell,\bv\rangle^2
\ge\frac{1}{2}\cdot\frac{\mu^4(\mbB_2)\Vert\bb^*-\bb\Vert_\Pi^2}{\Vert\bb^*-\bb\Vert_2^2},
\end{align}
for more than $(1-4\alpha_3)\calK$ buckets $i$'s. We thus conclude that for more than $(1-\alpha_3)\calK-k\ge1$ buckets, 
\begin{align}
|\langle\tilde\bz_i(\bb),\bv\rangle|\ge -C_{\alpha_1}\mu^2(\mbB_2)r_1+(\nicefrac{\mu^2(\mbB_2)}{2})\frac{\Vert\bb^*-\bb\Vert_\Pi^2}{\Vert\bb^*-\bb\Vert_2}
\ge -C_{\alpha_1}\mu^2(\mbB_2)r_1+(\nicefrac{1}{2})\Vert\bb^*-\bb\Vert_2.
\end{align}
In other words, the feasible set of 
$\RFHP(\{\tilde\bz_i(\bb)\}_{i=1}^\calK,k)$ contains the point 
$(\theta,\bv,\bq)\in\re\times R\mbB_2\times\{0,1\}^\calK$ with 
$\theta:=-C_{\alpha_1}\mu^2(\mbB_2)r_1+(\nicefrac{1}{2})\Vert\bb^*-\bb\Vert_2$ for some 
$\bq\in\{0,1\}^\calK$. By maximality, one must have 
$\theta_{\bb,k}\ge-C_{\alpha_1}\mu^2(\mbB_2)r_1+(\nicefrac{1}{2})\Vert\bb^*-\bb\Vert_2$. 
\end{proof}

\subsection{Computing the outer loop descent direction w.r.t. $\langle\cdot,\cdot\rangle_\Pi$}\label{ss:descent:direction}

\begin{lemma}[One-sided feasibility \& margin-\emph{angle} \emph{upper} bound]\label{lemma:condition:RFHP:step:dir}
Grant assumptions of Lemma \ref{lemma:RFHP:feasibility} and  
additionally assume: 
\begin{itemize}
    \item[\rm (iii)] \ref{lemma:high:prob:pruned:data':eq4} holds for some $\alpha_4\in(0,1)$ and $\rho\in(0,1/2]$. 
\item[\rm (vii)] 
$
r_2\le\frac{\sa}{4C_{\alpha_4}\mu^2(\mbB_2)\Vert\bfSigma\Vert}.
$
\item[\rm (viii)] 
$
\Vert\bb-\bb^*\Vert_2\ge 4
\left(\frac{\sa+1}{\sa}\right)C_{\alpha_1}\mu^2(\mbB_2)r_1.
$
\end{itemize}

Then
\begin{itemize}
\item[\rm (b)] Let $k'':=(2\alpha_4+0.75\rho)\calK+k+k'$. There exists $\bv'\in\{-\bv,\bv\}$ such that property 
\ref{equation:emp:quantile:neg:MP2}$(\{\tilde\bz_i(\bb)\}_{i=1}^\calK,\theta,\bv',k'')$, defined below, holds:
\begin{align}
\sum_{i=1}^{\calK}\unit_{\{\langle\tilde\bz_i(\bb),\bv'\rangle>\theta\}}> \calK-k''.
\tag{$\neg\mathtt{MP2}$}
\label{equation:emp:quantile:neg:MP2}\end{align}
\end{itemize}

Additionally to the assumptions of Lemma \ref{lemma:RFHP:feasibility} and (iii),(vii)-(viii), assume:
\begin{itemize}
\item[\rm (ix)] \ref{lemma:high:prob:pruned:data':eq2} holds for some $\alpha_2\in(0,1)$. Let $k_0:=4\alpha_2\calK$.
\end{itemize}

Then, for any $\bb\neq\bb^*$ satisfying (iv)-(viii) and 
\begin{itemize}
\item[\rm (x)] $k''\le \calK/3$ and $k_0\le\calK/3$,
\end{itemize}
one also has
\begin{itemize}
\item[\rm (c)] 
$
\theta\le C_{\alpha_2}\mu^2(\mbB_2)r_1+\langle\bb^*-\bb,\bv'\rangle_\Pi+C_{\alpha_4} \mu^2(\mbB_2)\Vert\bfSigma\Vert r_2\Vert\bb^*-\bb\Vert_2.
$
\end{itemize}

Suppose additionally that, instead of (viii), one has 
\item[\rm (viii')] 
$
\Vert\bb-\bb^*\Vert_2\ge A \mu^2(\mbB_2)r_1
$
where 
$$
A:=\left[4
\left(\frac{\sa+1}{\sa}\right)C_{\alpha_1}\right]
\bigvee
\left[(\nicefrac{8}{\sa})(\sa C_{\alpha_1}+C_{\alpha_2})\right].
$$

Then one also has
\begin{itemize}
\item[\rm (d)] 
$
\langle\bv',\bb-\bb^*\rangle_\Pi\le-\frac{\sa}{8}\Vert\bb-\bb^*\Vert_2.
$
\end{itemize}
\end{lemma}
\begin{remark}[The need for a margin-\emph{angle} upper bound]
The concept of \emph{distance-estimate} was shown to be the sufficient property in robust mean estimation when using the benchmark combinatorial problem in Section \ref{ss:benchmark:problem} \cite{2019cherapanamjeri:flammarion:bartlett, 2020lei:luh:venkat:zhang}. As it will be clearer in the following, we emphasize that, in our analysis based on least-squares methodology with unknown $\bfSigma$, a margin-\emph{distance} upper bound is \emph{not enough} to obtain the optimal rate for robust linear regression. In order to obtain the optimal rate and breakdown point with respect to the condition number 
$\kappa=\mu^2(\mbB_2)\Vert\bfSigma\Vert$, the margin-\emph{angle} upper bound in item (c) of Lemma \ref{lemma:condition:RFHP:step:dir} is crucially needed. Notice that such margin-angle upper bound follows from the one-sided benchmark problem (item (b) above). Differently, the margin-distance lower bound in item (a) follows from the two-sided benchmark problem. 
\end{remark}

\begin{proof}
\underline{Proof of (b)}: by \ref{lemma:high:prob:pruned:data':eq4}, \begin{align}
\frac{1}{B}\sum_{\ell\in\tilde B_i}\langle\tilde\bx_\ell,\bb-\bb^*\rangle\langle\tilde\bx_\ell,\bv\rangle
-\langle\bb-\bb^*,\bv\rangle_\Pi
\le C_{\alpha_4} r_2\Vert\bb-\bb^*\Vert_\Pi\Vert\bv\Vert_\Pi,
\end{align}
for more than $(1-(2\alpha_4+0.75\rho))\calK$ buckets $i$'s. Let $S$ denote such index set and define:
\begin{align}
G_{\bb^*}&:=\{i\in[\calK]:|\langle\tilde\bz_i(\bb^*),\bv\rangle|\le C_{\alpha_1}\mu^2(\mbB_2)r_1\},\\
B_{\bb}^+&:=\{i\in[\calK]:\langle\tilde\bz_i(\bb),\bv\rangle> \theta\},\\
B_{\bb}^-&:=\{i\in[\calK]:\langle\tilde\bz_i(\bb),-\bv\rangle> \theta\}.
\end{align}
We consider two cases.

\begin{description}
\item[Case 1:] $\langle\bb-\bb^*,\bv\rangle_\Pi\le0$. 
Given $i\in B_{\bb}^-\cap S$, 
\begin{align}
\langle\tilde\bz_i(\bb^*),\bv\rangle&=\langle\tilde\bz_i(\bb),\bv\rangle
+\frac{1}{B}\sum_{\ell\in\tilde B_i}\langle\tilde\bx_\ell,\bb-\bb^*\rangle\langle\tilde\bx_\ell,\bv\rangle\\
&\le-\theta+C_{\alpha_4} r_2
\Vert\bb-\bb^*\Vert_\Pi\Vert\bv\Vert_\Pi\\
&\stackrel{\rm (a)}{<}\sa C_{\alpha_1}\mu^2(\mbB_2)r_1+\left(-\frac{\sa}{2}+C_{\alpha_4}\Vert\bfSigma\Vert_2\mu^2(\mbB_2)r_2\right)\Vert\bb-\bb^*\Vert_2\\
&\stackrel{\rm (vii),(viii)}{<}-C_{\alpha_1}\mu^2(\mbB_2)r_1,
\end{align}
where we used from (viii) that $\Vert\bb-\bb^*\Vert_2\ge4(\nicefrac{\sa+1}{\sa})C_{\alpha_1}\mu^2(\mbB_2)r_1$.

We thus concluded that $B_{\bb}^-\cap S\subset G_{\bb^*}^c$.  This and the facts
\begin{itemize}
\item $|S|\ge[1-(2\alpha_4+0.75\rho)]\calK$ by \ref{lemma:high:prob:pruned:data':eq4},
\item $|G_{\bb^*}|>\calK-k$ by  \ref{lemma:high:prob:pruned:data':eq1},
\item $B_{\bb}^+$ and $B_{\bb}^-$ are disjoint because 
$\theta>0$, by (a) and (viii). Also, $|B_{\bb}^+|+|B_{\bb}^-|>\calK-k'$ by (iv), 
\end{itemize}
imply that $|B_{\bb}^+|>\calK-(2\alpha_4+0.75\rho)\calK-k-k'$. 

\item[Case 2:] $\langle\bb-\bb^*,\bv\rangle_\Pi>0$. By exchanging $\bv$ with $-\bv$ and $B_{\bb}^-$ by $B_{\bb}^+$ a similar argument shows that 
$|B_{\bb}^-|>\calK-(2\alpha_4+0.75\rho)\calK-k-k'$.
\end{description}

\underline{Proof of (c)}: By (b) and (x), one has
$
\langle\tilde\bz_i(\bb),\bv'\rangle\le\theta
$
for less $\calK/3$ of buckets $i$'s.
By \ref{lemma:high:prob:pruned:data':eq2} and (x), one has 
$
\langle\tilde\bz_i(\bb^*),\bv'\rangle> C_{\alpha_2}\mu^2(\mbB_2)r_1
$
for less $\calK/3$ of buckets $i$'s. Finally, by \ref{lemma:high:prob:pruned:data':eq4} and (x), for less than $\calK/3$ of buckets $i$'s one has 
\begin{align}
\frac{1}{B}\sum_{\ell\in\tilde B_i}\langle\tilde\bx_\ell,\bb^*-\bb\rangle\langle\tilde\bx_\ell,\bv'\rangle\le \langle\bb^*-\bb,\bv'\rangle_\Pi
+ C_{\alpha_4} r_2\Vert\bb-\bb^*\Vert_\Pi\Vert\bv'\Vert_\Pi.
\end{align}
By the pigeonhole principle, there is at least one bucket $i$ for which
$\langle\tilde\bz_i(\bb^*),\bv'\rangle>\theta$, 
$\langle\tilde\bz_i(\bb^*),\bv'\rangle\le C_{\alpha_2}\mu^2(\mbB_2)r_1$  and the previous display all hold. Thus
\begin{align}
\theta&<\langle\tilde\bz_i(\bb),\bv'\rangle\\
&=\langle\tilde\bz_i(\bb^*),\bv'\rangle+
\frac{1}{B}\sum_{\ell\in\tilde B_i}\langle\tilde\bx_\ell,\bb^*-\bb\rangle\langle\tilde\bx_\ell,\bv'\rangle\\
&\le C_{\alpha_2}\mu^2(\mbB_2)r_1+\langle\bb^*-\bb,\bv'\rangle_\Pi+C_{\alpha_4} \mu^2(\mbB_2)\Vert\bfSigma\Vert r_2\Vert\bb^*-\bb\Vert_2,
\end{align}
entailing the claim. 

\underline{Proof of (d)}: we join the upper bound (c) and the lower bound (a). Using (vii) and (viii'), so that $\Vert\bb-\bb^*\Vert_2\ge(8/\sa)(\sa C_{\alpha_1}+C_{\alpha_2})\mu^2(\mbB_2) r_1$, and rearranging the displayed inequality finishes the proof.
\end{proof}

From now on fix the parameters
$
\alpha_1=1/96, 
$
$
\alpha_2=0.08, 
$
$
\alpha_3=0.239, 
$
$
\alpha_4=1/144,
$
$
c_{\alpha_1}:=\frac{1}{4},
$
$
\rho=1/36,
$
$\sa:=0.0128$
and $\varphi:=0.49\pi$. 
In order to satisfy \eqref{lemma:count:block:emp:process:alpha}, 
it suffices to take 
$
C_{\alpha_1}=2525.26, 
$
$
C_{\alpha_2}=192.4, 
$
$
C_{\alpha_3}=51.9
$
and
$
C_{\alpha_4}=4330.
$

\begin{corollary}[Good descent properties]\label{cor:step:dir}
Let $\eta\in(0,1/2]$ and suppose that: 
\begin{itemize}
\item[\rm (i)] 
\ref{lemma:high:prob:pruned:data:eq0}, \ref{lemma:high:prob:pruned:data':eq1}, \ref{lemma:high:prob:pruned:data':eq3},  \ref{lemma:high:prob:pruned:data':eq4} and
\ref{lemma:high:prob:pruned:data':eq2} all hold.
    \item[\rm (ii)] 
Assume 
$
r_2\le\frac{1}{312.5C_{\alpha_4}\mu^2(\mbB_2)\Vert\bfSigma\Vert}.
$
\item[\rm (iii)] Given $\bb\in\re^p$, assume
$
\Vert\bb-\bb^*\Vert_2\ge 8\cdot10^5 \mu^2(\mbB_2)r_1.
$
\end{itemize}

Let $(\hat\theta,\hat\bv)$ be the output of  
$\MW(\calD,U,S_1,k,k',R,\tilde r_1)$, namely, 
$\Algorithm$ \ref{algo:MW} with 
inputs $\calD=\{\tilde\bz_i(\bb)\}_{i\in[\calK]}$ and $U=\max_{i\in[\calK]}\Vert\tilde\bz_i(\bb)\Vert_2^2$, $k=4\alpha_1\calK$, $k'=c_{\alpha_1}\calK$ (with $c_{\alpha_1}=1/4$), $\tilde r_1=C_{\alpha_1}\mu^2(\mbB_2)r_1$ and $R=\mu^2(\mbB_2)$. 

Then on an event of probability (on the randomness of $\{\btheta_{\ell}\}_{\ell\in[S_1]}$) of at least $1-e^{-\frac{S_1}{353}}$, one has
\begin{align}
\sum_{i=1}^{\calK}\unit_{\left\{|\langle\tilde\bz_i(\bb),\bv\rangle|>\hat\theta
\right\}}> \calK-k'
\quad{and}\quad
\hat\theta>\sa\theta_{\bb,k},
\label{cor:step:dir:0}
\end{align}
where $k=4\alpha_1\calK$ and $k'=c_{\alpha_1}\calK$. Moreover,  
\begin{align}
&(\nicefrac{\sa}{2})\Vert\bb-\bb^*\Vert_2-2525.26\mu^2(\mbB_2)r_1\le\hat\theta,\label{cor:step:dir:1}\\
&\hat\theta\le \langle\bb^*-\bb,\hat\bv\rangle_\Pi+4330\mu^2(\mbB_2)\Vert\bfSigma\Vert r_2\Vert\bb-\bb^*\Vert_2+192.4\mu^2(\mbB_2)r_1,\label{cor:step:dir:1'}\\
&\langle\hat\bv,\bb-\bb^*\rangle_\Pi\le-\frac{1}{625}\Vert\bb-\bb^*\Vert_2.
\label{cor:step:dir:2}
\end{align}
\end{corollary}
\begin{proof}
Setting $k=4\alpha_1=\calK/24$ and $k':=c_{\alpha_1}\calK=\calK/4$ with the parameters displayed before the corollary, one checks that $\sa\le\cos(\varphi)/\sqrt{6}$, 
$\mathfrak{p}\ge0.148$ and
all conditions of Lemmas \ref{lemma:MWU:from:below:linear:reg},  \ref{lemma:MW:from:below:lin:reg:boost}, \ref{lemma:RFHP:feasibility} and \ref{lemma:condition:RFHP:step:dir} hold. In particular, 
$
\Vert\bb-\bb^*\Vert_2\ge 4\mu^2(\mbB_2)C_{\alpha_1}r_1
$ 
implying condition (iii) of Lemma \ref{lemma:MWU:from:below:linear:reg} with $\tilde r_1:=C_{\alpha_1}\mu^2(\mbB_2)r_1$.

We now work on the event of probability $1-e^{-\frac{\mathfrak{p}^2S_1}{7.72}}$ for which the claim of Lemma \ref{lemma:MW:from:below:lin:reg:boost}
is true. By Lemma \ref{lemma:MW:from:below:lin:reg:boost},
\eqref{cor:step:dir:0} is satisfied; these are assumptions (iv) and (vi) of Lemma \ref{lemma:RFHP:feasibility} for 
$(\hat\theta,\bv)$. All other assumptions of such lemma hold, yielding \eqref{cor:step:dir:1}.

All conditions of Lemma \ref{lemma:RFHP:feasibility} and conditions (iii), (vii)-(viii), (viii') of Lemma \ref{lemma:condition:RFHP:step:dir} hold so there must exist 
$\bv'\in\{-\bv,\bv\}$ satisfying
$
\sum_{i=1}^{\calK}\unit_{\{\langle\tilde\bz_i(\bb),\bv'\rangle>\hat\theta\}}> \calK-k'',
$
that is, (b) of such lemma. By this property, the order statistics in 
$\Algorithm$ \ref{algo:ROUND:spherical:order} implies $\hat\bv=\bv'$. All additional conditions of Lemma \ref{lemma:condition:RFHP:step:dir} hold, yielding \eqref{cor:step:dir:1'}-\eqref{cor:step:dir:2}.
\end{proof}

We finalize this section showing we have a sufficiently small stepsize and descent direction, assuming one has an sufficiently good estimate of $\bfSigma\hat\bv$. We will show in the next section how to construct it (without knowing or estimating the covariance matrix $\bfSigma$). 

We first complement Lemma \ref{lemma:condition:RFHP:step:dir} and Corollary \ref{cor:step:dir} with additional results. Like Lemma \ref{lemma:RFHP:feasibility} and unlike Lemma \ref{lemma:condition:RFHP:step:dir} and Corollary \ref{cor:step:dir}, the next two results do not to assume that
$\mu^2(\mbB_2)r_1\lesssim\Vert\bb-\bb^*\Vert_2$ nor $\hat\theta>0$. Lemma \ref{lemma:stepsize:descent:direction} does assume, however, (d") which is stronger than (d). Also, the stepsize in Lemma \ref{lemma:condition:RFHP:step:dir} is 
$\hat\theta+C_{\alpha_1}\mu^2(\mbB_2)r_1$ instead of $\hat\theta$. These slightly more general results are only used in case one of the iterates follows within the statistical error before the final iteration. We need them to avoid Cauchy-Schwarz when upper bounding $\langle\hat\bv,\bb^*-\bb\rangle_\Pi$. Hence, we can attain the optimal rate and breakdown point with respect to the condition number $\kappa$. See proof of Theorem \ref{thm:rate:known:rate} in Section \ref{s:master:rate}. We omit the proof of Lemma \ref{lemma:condition:RFHP:step:dir-looser} as it is very similar to the proof of (c) in Lemma \ref{lemma:condition:RFHP:step:dir}. 

\begin{lemma}[Looser margin-angle upper bound]\label{lemma:condition:RFHP:step:dir-looser}
Grant items (i) and (iv) of Lemma \ref{lemma:RFHP:feasibility}, item (iii) of Lemma \ref{lemma:condition:RFHP:step:dir} and
additionally assume:
\begin{itemize}
    \item[\rm (x')] $k''\le \calK/3$ for $k'':=(2\alpha_4+0.75\rho)\calK+k+k'$. 
\end{itemize}

Then
\begin{itemize}
\item[\rm (c')] 
$
\theta\le C_{\alpha_2}\mu^2(\mbB_2)r_1+|\langle\bb^*-\bb,\bv\rangle_\Pi|+C_{\alpha_4} \mu^2(\mbB_2)\Vert\bfSigma\Vert r_2\Vert\bb^*-\bb\Vert_2.
$
\end{itemize}
\end{lemma}
 
\begin{lemma}[Descent direction]\label{lemma:stepsize:descent:direction}
Let $\bb\neq\bb^*$ and $(\hat\theta,\hat\bv)$ be the output of $\Algorithm$ \ref{algo:MW}
with inputs $\calD=\{\tilde\bz_i(\bb)\}_{i\in[\calK]}$ and $U=\max_{i\in[\calK]}\Vert\tilde\bz_i(\bb)\Vert_2^2$. Assume  there exist positive constants $(\sa,\sa_1,\sa_2,\sa_3,\sa_4)$ such that
\begin{itemize}
\item[\rm (a")] 
$(\nicefrac{\sa}{2})\Vert\bb-\bb^*\Vert_2
\le\hat\theta+\sa_1\mu^2(\mbB_2)r_1$. 
\item[\rm (b")] 
$
\hat\theta\le \langle\bb^*-\bb,\hat\bv\rangle_\Pi
+\sa_2\mu^2(\mbB_2)\Vert\bfSigma\Vert r_2\Vert\bb^*-\bb\Vert_2
+\sa_3\mu^2(\mbB_2)r_1.
$
\item[\rm (d")] 
$
\left(
\sa_4\Vert\bb^*-\bb\Vert_2
\right)
\bigvee
\left(
(\sa_1+\sa_3)\mu^2(\mbB_2)r_1
\right)
\le\langle\bb^*-\bb,\hat\bv\rangle_\Pi.
$
\end{itemize}

Suppose further:
\begin{itemize}
\item We know an estimate 
$\hat\bmu$ of $\bfSigma\hat\bv$ satisfying 
$
\Vert\hat\bmu-\bfSigma\hat\bv\Vert_2\le\Delta
$
for some $\Delta\in(0,1)$.
\item $\Delta<\frac{\sa}{16}$.
\item $\sa_2(\kappa r_2)\le\frac{1}{4}$.
\end{itemize}

Let 
$
c_*:=\frac{\sa}{8(2+\sa_2)\kappa^2(\kappa^2+\Delta^2)},
$
and
$
\bb^+:=\bb+c_*(\hat\theta+\sa_1\mu^2(\mbB_2)r_1)\hat\bmu.
$

Then
$$
\Vert\bb^+-\bb^*\Vert_2^2\le
\left(
1-\frac{\sa^2}{32(2+\sa_2)\kappa^2(\kappa^2+\Delta^2)}
\right)\Vert\bb-\bb^*\Vert_2^2.
$$
\end{lemma}
\begin{proof}
Let us denote $\theta:=\hat\theta+\sa_1\mu^2(\mbB_2)r_1$. We first note that 
\begin{align}
\Vert\bb^+-\bb^*\Vert_2^2&=
\Vert\bb-\bb^*\Vert_2^2
+2c_*\theta\langle\hat\bmu,\bb-\bb^*\rangle
+c_*^2\theta^2\Vert\hat\bmu\Vert_2^2\\
&\le \Vert\bb-\bb^*\Vert_2^2
+2c_*\theta\langle\bfSigma\hat\bv,\bb-\bb^*\rangle
+2c_*^2\theta^2\Vert\bfSigma\hat\bv\Vert_2^2\\
&+2c_*\theta\langle\hat\bmu-\bfSigma\hat\bv,\bb-\bb^*\rangle
+2c_*^2\theta^2\Vert\hat\bmu-\bfSigma\hat\bv\Vert_2^2.
\end{align}
For ease of notation, we let $A_{\bb}:=\langle\hat\bv,\bb-\bb^*\rangle_\Pi$ and 
$D_{\bb}:=\Vert\bb^*-\bb\Vert_2$. 

We next bound the first-order terms in $c_*$. We have 
\begin{align}
T_1&:=2c_*\theta\langle\bfSigma\hat\bv,\bb-\bb^*\rangle+2c_*\theta\langle\hat\bmu-\bfSigma\hat\bv,\bb-\bb^*\rangle\\
&\stackrel{\rm (b")}{\le}2c_*\theta A_{\bb}
+2c_*\left(
-A_{\bb}
+\sa_2\kappa r_2D_{\bb}
+(\sa_1+\sa_3)\mu^2(\mbB_2)r_1
\right)
\Delta D_{\bb}\\
&=2c_*A_{\bb}(\theta-\Delta D_{\bb})
+2c_*((\sa_1+\sa_3)\mu^2(\mbB_2)r_1+\sa_2\kappa r_2D_{\bb})\Delta D_{\bb}. 
\end{align}

Since $A_{\bb}<0$, $\theta\ge\frac{\sa}{2}D_{\bb}$ and $\Delta<\frac{\sa}{2}$, the first term above satisfies
\begin{align}
2c_*A_{\bb}(\theta-\Delta D_{\bb})
&\le2c_*((\nicefrac{\sa}{2})-\Delta)D_{\bb}A_{\bb}. 
\end{align}
As for the second term, using (d"), it is upper bounded by
\begin{align}
2c_*(-A_{\bb}+\sa_2\kappa r_2D_{\bb})\Delta D_{\bb}
=2c_*(-A_{\bb})\Delta D_{\bb}
+2c_*\sa_2(\kappa r_2)\Delta D_{\bb}^2.
\end{align}
Using $\Delta<\frac{\sa}{4}$ and $A_{\bb}\le-\sa_4D_{\bb}$, we thus conclude that
\begin{align}
T_1\le-2c_*\left(
(\nicefrac{\sa}{2})-2\Delta)\sa_4
-\sa_2(\kappa r_2)\Delta 
\right)D_{\bb}^2\le -\frac{\sa}{2}c_*D_{\bb}^2, 
\end{align}
since $2\Delta\sa_4+\sa_2(\kappa r_2)\Delta\le\frac{\sa}{4}$.

We next bound the second-order terms in $c_*$. From (b") and (d"), Cauchy-Schwarz and 
$\Vert\hat\bv\Vert_2\le\mu^2(\mbB_2)$, we have 
$
\theta\le 2(-A_{\bb})+\sa_2\kappa D_{\bb}
\le (2\kappa+\sa_2\kappa)D_{\bb}. 
$
We thus have 
\begin{align}
T_2:=2c_*^2\theta^2\Vert\bfSigma\hat\bv\Vert_2^2
+2c_*^2\theta^2\Vert\hat\bmu-\bfSigma\hat\bv\Vert_2^2
\le2c_*^2(2\kappa+\sa_2\kappa)^2D_{\bb}^2
\left(
\kappa^2+\Delta^2
\right).
\end{align}

We thus conclude that 
\begin{align}
\Vert\bb^+-\bb^*\Vert_2^2\le\left[
1-(\nicefrac{\sa}{2})c_*
+2c_*^2(2\kappa+\sa_2\kappa)^2
(\kappa^2+\Delta^2)
\right]\Vert\bb-\bb^*\Vert_2^2.
\end{align}
Optimizing on $c_*$ entails the claim. 
\end{proof}

\subsection{Estimating the outer loop descent direction}

In Sections \ref{ss:largest:residual} and \ref{ss:descent:direction}, the estimated direction $\hat\bv$ in $\Algorithm$ \ref{algo:MW} is a descent direction with respect to the conditioned inner product 
$\langle\cdot,\cdot\rangle_\Pi=\langle\bfSigma(\cdot),\cdot\rangle$. Still, we cannot use it as is without knowing $\bfSigma$. We note however that, if an upper estimate of $\Vert\bfSigma\Vert$ is available, all we need is an estimate of 
$\bfSigma\hat\bv$. In this section, we show that the property  \ref{lemma:high:prob:pruned:data':eq4}, already shown to be satisfied by the pruned data set $\{\tilde\bx_\ell\}_{\ell=1}^m$, is enough to estimate $\bfSigma\hat\bv$ by means of robust mean estimation. 

For ease of reference, we make some definitions. 
\begin{tcolorbox}
\begin{setup}\label{setup:pruned:sample:mean}
Given $\bmu\in\re^p$, we define
\begin{align}
\tilde\bz_i(\hat\bv,\bmu):=\frac{1}{B}\sum_{\ell\in \tilde B_i}(\langle\tilde\bx_\ell,\hat\bv\rangle\tilde\bx_\ell-\bmu),
\quad\mbox{and}\quad
\bz_i(\hat\bv,\bmu):=\frac{1}{B}\sum_{\ell\in B_i}(\langle\tilde\bx_\ell,\hat\bv\rangle\bx_\ell
-\bmu). 
\end{align}  
\end{setup}
\end{tcolorbox}
The following statement is immediate from Lemma \ref{lemma:high:prob:pruned:data:v3} evaluated at 
$\hat\bv$. We state it for ease of reference.
\begin{corollary}[Pruned sample: Noise Process at 
$\hat\bv$]\label{cor:high:prob:pruned:data:mean}
Grant Assumption \ref{assump:L4:L2} and Set-ups \ref{setup:pruned:sample} and \ref{setup:pruned:sample:mean}. Let $\bar\rho\in(0,1/2]$, $\bar\alpha_4\in(0,1)$ and constant $C_{\bar\alpha_4}>0$ satisfying  \eqref{lemma:count:block:emp:process:alpha}. 
Suppose that 
\begin{align}
o&\le\bar\alpha_4K.
\end{align}
Let $C_{\bar\rho}':=1+\sqrt{2/\bar\rho}$ and $C>0$ be an absolute constant in Lemma \ref{lemma:trunc:quad:rademacher:comp}. Let $r_{n,K}$ as in Proposition \ref{prop:quad:proc:block:upper:heavy}.
Let $\bar\alpha_1\ge(3\bar\alpha_4+0.75\bar\rho)/4$ and $\bar\alpha_2\ge(3\bar\alpha_4+0.75\bar\rho)/2$ and let $(C_{\bar\alpha_1},C_{\bar\alpha_2})$ satisfy  \eqref{lemma:count:block:emp:process:alpha} with respect to $(\bar\alpha_1,\bar\alpha_2)$. 
Set $\bar r_1:=\Vert\bfSigma\Vert \bar r_2$ where
\begin{align}
\bar r_2:=2C_{\bar\alpha_4}r_{n,K}+
CC_{\bar\alpha_4} C_{\bar\rho}'\frac{p\log p}{2n}
+2L^2\sqrt{\frac{CC_{\bar\alpha_4} p\log p}{2n}}.
\label{lemma:high:prob:pruned:data:mean:r2}
\end{align}

Then on a $\{(y_\ell,\bx_\ell)\}_{\ell=1}^n\cup\{(y_\ell,\bx_\ell)\}_{\ell=n+1}^{2n}$-measurable event $\bar\calE_3$ of probability at least $1-e^{-\bar\rho n/1.8}-e^{-K/C_{\bar\alpha_4}}$,
\begin{align}
\sup_{\bv\in\mbB_2}\sum_{i\in[\calK]}\unit_{\{|\langle\tilde\bz_i(\hat\bv,\bfSigma\hat\bv),\bv\rangle|\ge C_{\bar\alpha_1} \bar r_1\}}&\le4\bar\alpha_1\calK,
\tag{$\mathtt{NP1}(\bar\alpha_1,\bar r_1)$}\label{cor:high:prob:pruned:data':mean:eq1}\\ 
\sup_{\bv\in\mbB_2}\sum_{i\in[\calK]}\unit_{\{\langle\tilde\bz_i(\hat\bv,\bfSigma\hat\bv),\bv\rangle\ge C_{\bar\alpha_2}\bar r_1\}}&\le 4\bar\alpha_2\calK. 
\tag{$\mathtt{NP2}(\bar\alpha_2,\bar r_1)$}
\label{cor:high:prob:pruned:data':mean:eq2}
\end{align}
\end{corollary}

\begin{tcolorbox}
\begin{setup}\label{setup:final}
In all this section, we work within the Set-up \ref{setup:pruned:sample} and \ref{setup:pruned:sample:mean} and on the event 
$\calE:=\calE_0\cap\calE_1\cap\calE_2\cap\calE_3\cap\bar\calE_3$ where, for the pruned sample $\{(\tilde y_\ell,\tilde\bx_\ell)\}_{\ell\in[m]}$ and initialization $\tilde\bb^{(0)}$, the boundedness property \ref{lemma:high:prob:pruned:data:eq0} and the uniform properties
\ref{lemma:high:prob:pruned:data':eq1}, 
\ref{lemma:high:prob:pruned:data':eq2},  \ref{lemma:high:prob:pruned:data':eq3}, \ref{lemma:high:prob:pruned:data':eq4},
\ref{cor:high:prob:pruned:data':mean:eq1}
and \ref{cor:high:prob:pruned:data':mean:eq2} all hold. The arguments in this section are purely deterministic.  
\end{setup}
\end{tcolorbox}

\subsubsection{Solving the inner loop combinatorial problem}

When $\tilde\mbZ:=\{\tilde\bz_i(\hat\bv,\bmu)\}_{i=1}^\calK$ for some $\bmu\neq\bfSigma\hat\bv$, problem $\RFHP(\tilde\mbZ,\bar k,1)$ for some $\bar k\in[\calK]$ becomes an parametrized instance of the Furthest Hyperplane Problem used in prior work for robust mean estimation \cite{2012karnin:liberty:lovett:schwartz:weinstein}. For simplicity we will omit the length $R=1$ in the following. The next lemma states that under the structural condition \ref{cor:high:prob:pruned:data':mean:eq1}, problem $\RFHP(\{\tilde\bz_i(\hat\bv,\bmu)\}_{i=1}^\calK,k)$ is feasible. 

\begin{lemma}[Two-sided feasibility \& margin-distance lower bound]\label{lemma:RFHP:feasibility:mean}
Suppose that \ref{cor:high:prob:pruned:data':mean:eq1}  holds and, for $\bar k:=4\bar\alpha_1\calK$, 
$$
\bar k<\calK.
$$

\quad

Then, for any $\bmu\neq\bfSigma\hat\bv$, $\RFHP(\{\tilde\bz_i(\hat\bv,\bmu)\}_{i=1}^\calK,\bar k)$ is feasible; in particular, its optimal solution 
$(\theta_{\hat\bv,\bmu,\bar k},\bv_{\hat\bv,\bmu,\bar k})$ satisfies \ref{equation:emp:quantile:neg:P1}$(\{\tilde\bz_i(\hat\bv,\bmu)\}_{i=1}^\calK,\theta_{\hat\bv,\bmu,\bar k},\bv_{\hat\bv,\bmu,\bar k},
\bar k)$ with margin
$$
\theta_{\hat\bv,\bmu,\bar k}\ge-C_{\bar\alpha_1}\bar r_1+\Vert\bfSigma\hat\bv-\bmu\Vert_2.
$$
\end{lemma}
\begin{proof}
\underline{STEP 1}: An upper bound on the optimal value is trivial: for any $(\theta,\bv,\bq)$ satisfying the constraints of $\RFHP(\{\tilde\bz_i(\hat\bv,\bmu)\}_{i=1}^\calK,\bar k)$ it follows from Cauchy-Schwarz that
$\theta_{\hat\bv,\bmu,\bar k}\le\max_{i\in[\calK]}\Vert\tilde\bz_i(\hat\bv,\bmu)\Vert_2<\infty$. 

\underline{STEP 2:} we now prove the lower bound 
$
\theta_{\hat\bv,\bmu,\bar k}\ge\Vert\bmu-\bfSigma\hat\bv\Vert_2-C_{\bar\alpha_1}\bar r_1. 
$
\ref{cor:high:prob:pruned:data':mean:eq1} applied to the unit vector 
$\bv:=\nicefrac{\bfSigma\hat\bv-\bmu}{\Vert\bfSigma\hat\bv-\bmu\Vert_2}$ implies that
$
\langle\tilde\bz_i(\hat\bv,\bfSigma\hat\bv),\bv\rangle\ge-C_{\bar\alpha_1}\bar r_1
$
for more than $\calK-\bar k$ buckets $i$'s, for which 
\begin{align}
|\langle\tilde\bz_i(\hat\bv,\bfSigma\hat\bv),\bv\rangle|\ge\langle\tilde\bz_i(\hat\bv,\bmu),\bv\rangle
&=\langle\tilde\bz_i(\hat\bv,\bfSigma\hat\bv),\bv\rangle
+\langle\bfSigma\hat\bv-\bmu,\bv\rangle\\
&\ge-C_{\bar\alpha_1}\bar r_1+\Vert\bfSigma\hat\bv-\bmu\Vert_2.
\end{align}
In other words, the feasible set of 
$\RFHP(\{\tilde\bz_i(\hat\bv,\bmu)\}_{i=1}^\calK,\bar k)$ contains the point 
$(\theta,\bv,\bq)$ with 
$\theta:=-C_{\bar\alpha_1}\bar r_1+\Vert\bfSigma\hat\bv-\bmu\Vert_2$ for some 
$\bq\in\{0,1\}^\calK$. By maximality, one must have 
$\theta_{\hat\bv,\bmu,\bar k}\ge-C_{\bar\alpha_1}\bar r_1+\Vert\bfSigma\hat\bv-\bmu\Vert_2$. 
\end{proof}

\subsubsection{Computing the inner loop descent direction}\label{ss:descent:direction:mean}
Let $\bmu\in\re^p$ be a current point that is far from $\bfSigma\hat\bv$. 

\begin{lemma}\label{lemma:condition:RFHP:margin:mean}
Suppose that 
\begin{itemize}
    \item[\rm (i)] \ref{cor:high:prob:pruned:data':mean:eq1} holds for some $\bar\alpha_1\in(0,1)$. Let $\bar k:=4\bar\alpha_1\calK$.
\end{itemize}

Let $\bmu\neq\bfSigma\hat\bv$ satisfying: 
\begin{itemize}
    \item[\rm (ii)] For some 
$(\bar\theta,\bar\bv,\bar k')\in\re\times\mbS_2\times[\calK]$,  property \ref{equation:emp:quantile:neg:P1}$(\{\tilde\bz_i(\hat\bv,\bmu)\}_{i\in[\calK]},\bar\theta,\bar\bv,\bar k')$ holds, i.e., 
\begin{align}
\sum_{i=1}^{\calK}\unit_{\{|\langle\tilde\bz_i(\hat\bv,\bmu),\bar\bv\rangle|>\theta\}}> \calK-\bar k'.
\end{align}
In particular, $\RFHP(\{\tilde\bz_i(\hat\bv,\bmu)\}_{i\in[\calK]},\bar k')$ is feasible with optimal value, say, $\bar\vartheta$.
\item[\rm (iii)] 
$\bar k<\calK$.
\item[\rm (iv)] For some $\bar\sa\in(0,1]$, 
$\bar\theta\ge\bar\sa\bar\vartheta$.
\end{itemize}
Then 
\begin{itemize}
\item[\rm (a)] 
$
\bar\sa\left(\Vert\bmu-\bfSigma\hat\bv\Vert_2-C_{\bar\alpha_1}\bar r_1\right)\le \bar\theta\le\Vert\bmu-\bfSigma\hat\bv\Vert_2+C_{\bar\alpha_1}\bar r_1.
\label{lemma:condition:RFHP:step:dir:a:mean}
$
\end{itemize}
\end{lemma}
\begin{proof}
From (ii), maximality and (iv), 
$\bar\sa\bar\vartheta\le\bar\theta\le\bar\vartheta$. We skip the proof of the lower bound 
$\bar\vartheta\ge\Vert\bfSigma\hat\bv-\bmu\Vert_2-C_{\bar\alpha_1}\bar r_1$ as it is proved in the same say as in STEP 2 of the proof of Lemma \ref{lemma:RFHP:feasibility:mean} using (i) and condition $\bar k<\calK$ in (iii). 

Next, we prove the upper bound 
$\bar\vartheta\le\Vert\bfSigma\hat\bv-\bmu\Vert_2+C_{\bar\alpha_1}\bar r_1$. Assume by contradiction that $\bar\vartheta>\Vert\bfSigma\hat\bv-\bmu\Vert_2+C_{\bar\alpha_1}\bar r_1$. By maximality, this implies that  there must exist 
$\bar\theta'\in\re$, $\bar\bv'\in\mbS_2$ and 
$\bar\bq'\in\{0,1\}^\calK$ satisfying the constraints of 
$\RFHP(\{\tilde\bz_i(\hat\bv,\bmu)\}_{i\in[\calK]},\bar k')$ such that 
$\bar\theta'>\Vert\bfSigma\hat\bv-\bmu\Vert_2+C_{\bar\alpha_1}\bar r_1$. In particular, for more than $\calK-\bar k'$ buckets $i$'s,  
$
|\langle\tilde\bz_i(\hat\bv,\bmu),\bar\bv'\rangle|\ge\bar\theta'
$, implying 
\begin{align}
|\langle\tilde\bz_i(\hat\bv,\bfSigma\hat\bv),\bar\bv'\rangle|
\ge|\langle\tilde\bz_i(\hat\bv,\bmu),\bar\bv'\rangle|
-\left|\langle\bfSigma\hat\bv-\bmu,\bar\bv'\rangle
\right|\ge\bar\theta'-\Vert\bmu-\bfSigma\hat\bv\Vert_2\ge C_{\bar\alpha_1}\bar r_1.
\end{align}
This contradicts  \ref{cor:high:prob:pruned:data':mean:eq1} in (i),  finishing the proof of (a).
\end{proof}

\begin{lemma}[One-sided feasibility \& margin-\emph{distance} upper bound]\label{lemma:condition:RFHP:step:dir:mean}
Grant assumptions of Lemma \ref{lemma:condition:RFHP:margin:mean} and  
additionally assume:
\begin{itemize} 
\item[\rm (v)] $
\Vert\bmu-\bfSigma\hat\bv\Vert_2\ge A \bar r_1
$
with
$$
A:=\left[
\left(\frac{\bar\sa+1}{\bar\sa}\right)C_{\bar\alpha_1}\right]
\bigvee
\left[(\nicefrac{2}{\bar\sa})(\bar\sa C_{\bar\alpha_1}+C_{\bar\alpha_2})\right].
$$
\end{itemize}

Then
\begin{itemize}
\item[\rm (b)] Let $\bar k'':=\bar k+\bar k'$. There exists $\bar\bv'\in\{-\bar\bv,\bar\bv\}$ such that property
\ref{equation:emp:quantile:neg:NP2}$(\hat\bv,\bmu,\bar\theta,\bar\bv',\bar k'')$, defined below, holds:
\begin{align}
\sum_{i=1}^{\calK}\unit_{\{\langle\tilde\bz_i(\hat\bv,\bmu),\bar\bv'\rangle>\bar\theta\}}> \calK-\bar k''.
\tag{$\neg\mathtt{NP2}$}
\label{equation:emp:quantile:neg:NP2}
\end{align}
\end{itemize}

Additionally to the assumptions of Lemma \ref{lemma:condition:RFHP:margin:mean} and (v), assume:
\begin{itemize}
\item[\rm (vi)] \ref{cor:high:prob:pruned:data':mean:eq2} holds for some $\bar\alpha_2\in(0,1)$. Let $\bar k_0:=4\bar\alpha_2\calK$.
\end{itemize}

Then, for any $\bmu\neq\bfSigma\hat\bv$ satisfying (ii)-(v) and 
\begin{itemize}
\item[\rm (vii)] $\bar k''\le \calK/2$ and $\bar k_0\le\calK/2$,
\end{itemize}
one also has
\begin{itemize}
\item[\rm (c)] 
$
\langle\bar \bv',\bmu-\bfSigma\hat\bv\rangle\le-\frac{\bar\sa}{2}\Vert\bmu-\bfSigma\hat\bv\Vert_2.
$
\end{itemize}
\end{lemma}
\begin{proof}
\underline{Proof of (b)}: define the sets
\begin{align}
G_{\hat\bv}&:=\{i\in[\calK]:|\langle\tilde\bz_i(\hat\bv,\bfSigma\hat\bv),\bar\bv\rangle|\le C_{\bar\alpha_1}\bar r_1\},\\
B_{\bmu}^+&:=\{i\in[\calK]:\langle\tilde\bz_i(\hat\bv,\bmu),\bar\bv\rangle> \bar\theta\},\\
B_{\bmu}^-&:=\{i\in[\calK]:\langle\tilde\bz_i(\hat\bv,\bmu),-\bar\bv\rangle> \bar\theta\}.
\end{align}
We consider two cases.

\begin{description}
\item[Case 1:] $\langle\bmu-\bfSigma\hat\bv,\bar\bv\rangle\le0$. 
Given $i\in B_{\bmu}^-$, 
\begin{align}
\langle\tilde\bz_i(\hat\bv,\bfSigma\hat\bv),\bar\bv\rangle&=\langle\tilde\bz_i(\hat\bv,\bmu),\bar\bv\rangle
+\langle\bmu-\bfSigma\hat\bv,\bar\bv\rangle\\
&\le-\bar\theta\\
&\stackrel{\rm (a)}{<}\sa C_{\bar\alpha_1}\bar r_1-\bar\sa\Vert\bmu-\bfSigma\hat\bv\Vert_2\\
&\stackrel{\rm (v)}{<}-C_{\bar\alpha_1}\bar r_1,
\end{align}
where we used from (v) that 
$\Vert\bmu-\bfSigma\hat\bv\Vert_2\ge(\nicefrac{1+\bar\sa}{\bar\sa})C_{\bar\alpha_1}\bar r_1$.

We thus concluded that $B_{\bmu}^-\subset G_{\hat\bv}^c$.  This and the facts
\begin{itemize}
\item $|G_{\hat\bv}|>\calK-\bar k$ by  \ref{cor:high:prob:pruned:data':mean:eq1},
\item $B_{\bmu}^+$ and $B_{\bmu}^-$ are disjoint because 
$\bar\theta>0$, by (a) and (v). Also, $|B_{\bmu}^+|+|B_{\bmu}^-|>\calK-\bar k'$ by (ii), 
\end{itemize}
imply that $|B_{\bmu}^+|>\calK-\bar k-\bar k'$. 

\item[Case 2:] $\langle\bmu-\bfSigma\hat\bv,\bar\bv\rangle>0$. By exchanging $\bar\bv$ with $-\bar\bv$ and $B_{\bmu}^-$ by $B_{\bmu}^+$ a similar argument shows that 
$|B_{\bmu}^-|>\calK-\bar k-\bar k'$.
\end{description}

\underline{Proof of (c)}: By (b) and (vii), one has
$
\langle\tilde\bz_i(\hat\bv,\bmu),\bar\bv'\rangle\le\bar\theta
$
for less $\calK/2$ of buckets $i$'s.
By \ref{cor:high:prob:pruned:data':mean:eq2} and (vii), one has 
$
\langle\tilde\bz_i(\hat\bv,\bfSigma\hat\bv),\bar\bv'\rangle> C_{\bar\alpha_2}\bar r_1
$
for less $\calK/2$ of buckets $i$'s. 
By the pigeonhole principle, there is at least one bucket $i$ for which
$\langle\tilde\bz_i(\hat\bv,\bmu),\bar\bv'\rangle>\bar\theta$ and
$\langle\tilde\bz_i(\hat\bv,\bfSigma\hat\bv),\bar\bv'\rangle\le C_{\bar\alpha_2}\bar r_1$ hold. Thus
\begin{align}
\bar\sa\left[\Vert\bmu-\bfSigma\hat\bv\Vert_2-C_{\bar\alpha_1}\bar r_1\right]&\stackrel{\rm (a)}{\le}
\bar\theta\\
&<\langle\tilde\bz_i(\hat\bv,\bmu),\bar\bv'\rangle\\
&=\langle\tilde\bz_i(\hat\bv,\bfSigma\hat\bv),\bar\bv'\rangle+
\langle\bfSigma\hat\bv-\bmu,\bar\bv'\rangle\\
&\le C_{\bar\alpha_2}\bar r_1+\langle\bfSigma\hat\bv-\bmu,\bar\bv'\rangle.
\end{align}
Using (v), so that $\Vert\bmu-\bfSigma\hat\bv\Vert_2\ge(\nicefrac{2}{\bar\sa})(\bar\sa C_{\bar\alpha_1}+C_{\bar\alpha_2})\bar r_1$, and rearranging the displayed inequality finishes the proof. 
\end{proof}

We conclude this section with the following corollary. From now on, we fix the parameters 
$\bar\alpha_4=1/96$, $\bar\rho=1/24$ so that
$
\bar\alpha_1=1/64 
$
and
$
\bar\alpha_2=1/8
$
satisfy the conditions of Corollary \ref{cor:high:prob:pruned:data:mean}. We also set 
$
c_{\bar\alpha_1}:=\frac{1}{4},
$
$\bar\sa:=0.0128$
and $\bar\varphi:=0.49\pi$. 
In order to satisfy \eqref{lemma:count:block:emp:process:alpha}, 
it suffices to take 
$
C_{\bar\alpha_1}=1666.68, 
$
and
$
C_{\bar\alpha_2}=110.
$

\begin{corollary}\label{cor:step:dir:mean}
Let $\eta\in(0,1/2]$ and suppose that: 
\begin{itemize}
\item[\rm (i)] 
\ref{lemma:high:prob:pruned:data:eq0}, \ref{cor:high:prob:pruned:data':mean:eq1} and
\ref{cor:high:prob:pruned:data':mean:eq2} all hold.
\item[\rm (ii)] Given $\bmu\in\re^p$, assume
$
\Vert\bmu-\bfSigma\hat\bv\Vert_2\ge 1.32\cdot10^5\bar r_1.
$
\end{itemize}

Let $(\tilde\theta,\tilde\bv)$ be the output of 
$\MW(\calD,U,S_2,\bar k,\bar k',1,\tilde r_1)$, that is, $\Algorithm$ \ref{algo:MW}  with inputs 
$\calD=\{\tilde\bz_i(\hat\bv,\bmu)\}_{i\in[\calK]}$ and $U=\max_{i\in[\calK]}\Vert\tilde\bz_i(\hat\bv,\bmu)\Vert_2^2$, $\bar k=4\bar\alpha_1\calK$ and $\bar k'=c_{\bar\alpha_1}\calK$, $R=1$ and $\tilde r_1=C_{\bar\alpha_1}\bar r_1$. 

Then on an event of probability (on the randomness of $\{\btheta_\ell\}_{\ell\in[S_2]}$) of at least $1-e^{-\frac{S_2}{669}}$, one has
\begin{align}
\sum_{i=1}^{\calK}\unit_{\left\{|\langle\tilde\bz_i(\hat\bv,\bmu),\bar\bv\rangle|>\tilde\theta
\right\}}> \calK-\bar k'
\quad{and}\quad
\tilde\theta>\bar\sa\theta_{\hat\bv,\bmu,\bar k}.
\label{cor:step:dir:mean:0}
\end{align}
 Moreover,  
\begin{align}
&(\nicefrac{1}{78.125})\Vert\bmu-\bfSigma\hat\bv\Vert_2-1666.68\bar r_1\le\tilde\theta,\label{cor:step:dir:mean:1}\\
&\tilde\theta\le\Vert\bmu-\bfSigma\hat\bv\Vert_2+1666.68\bar r_1,\label{cor:step:dir:mean:1'}\\
&\langle\tilde\bv,\bmu-\bfSigma\hat\bv\rangle_2\le-\frac{1}{156.25}\Vert\bmu-\bfSigma\hat\bv\Vert_2.
\label{cor:step:dir:mean:2}
\end{align}
\end{corollary}
\begin{proof}
Setting $\bar k=4\bar\alpha_1$ and $\bar k':=c_{\bar\alpha_1}\calK$ with the parameters displayed before the corollary, one checks that 
$\bar\sa\le\cos(\bar\varphi)/\sqrt{6}$, 
$\mathfrak{p}=0.1075$ and
all conditions of Lemmas \ref{lemma:MWU:from:below:linear:reg},  \ref{lemma:MW:from:below:lin:reg:boost}, \ref{lemma:condition:RFHP:margin:mean} and \ref{lemma:condition:RFHP:step:dir:mean} hold. In particular, 
$
\Vert\bmu-\bfSigma\hat\bv\Vert_2\ge 2C_{\bar\alpha_1}\bar r_1
$ 
implying condition (iii) of Lemma \ref{lemma:MWU:from:below:linear:reg} with $\tilde r_1:=C_{\bar\alpha_1}\bar r_1$.

We now work on the event of probability $1-e^{-\frac{\mathfrak{p}^2S_2}{7.72}}$ for which the claim of Lemma \ref{lemma:MW:from:below:lin:reg:boost}
is true. By Lemma such lemma,
\eqref{cor:step:dir:mean:0} is satisfied; these are assumptions (ii) and (iv) of Lemma \ref{lemma:condition:RFHP:margin:mean} for 
$(\tilde\theta,\bar\bv)$. All other assumptions of such lemma hold, yielding \eqref{cor:step:dir:mean:1}-\eqref{cor:step:dir:mean:1'}.

All conditions of Lemma \ref{lemma:condition:RFHP:margin:mean} and conditions (v) of Lemma \ref{lemma:condition:RFHP:step:dir:mean} holds so there must exist 
$\bar\bv'\in\{-\bar\bv,\bar\bv\}$ satisfying
$
\sum_{i=1}^{\calK}\unit_{\{\langle\tilde\bz_i(\hat\bv,\bmu),\bar\bv'\rangle>\tilde\theta\}}> \calK-\bar k'',
$
that is, (b) of such lemma. By this property, the order statistics in 
$\Algorithm$ \ref{algo:ROUND:spherical:order} implies $\tilde\bv=\bar\bv'$. All additional conditions of Lemma \ref{lemma:condition:RFHP:step:dir:mean} hold, yielding \eqref{cor:step:dir:mean:2}.
\end{proof}

\begin{lemma}\label{lemma:stepsize:descent:direction:mean}
Let 
$
\bar c_*:=1.045752\cdot10^{-06}
$
and $\Delta_0:=1.093597\cdot10^{-12}$.

Let $\bmu\in\re^p$ and grant assumptions (i)-(ii) of Corollary \ref{cor:step:dir:mean} which guarantees the outputted margin-direction pair $(\tilde\theta,\tilde\bv)$ produced by $\Algorithm$ \ref{algo:MW}.

Let 
$
\bmu^+:=\bmu+\bar c_*\tilde\theta\tilde\bv.
$

Then, on the same event of Corollary \ref{cor:step:dir:mean},  
$$
\Vert\bmu^+-\bfSigma\hat\bv\Vert_2^2\le\left(
1-\bar\Delta_0
\right)
\Vert\bmu-\bfSigma\hat\bv\Vert_2^2.
$$
\end{lemma}
\begin{proof}
One has
\begin{align}
\Vert\bmu^+-\bfSigma\hat\bv\Vert_2^2&\le
\Vert\bmu-\bfSigma\hat\bv\Vert_2^2
+2\bar c_*\tilde\theta\langle\tilde\bv,\bmu^+-\bfSigma\hat\bv\rangle
-\Vert\bmu^+-\bmu\Vert_2^2\\
&=\Vert\bmu-\bfSigma\hat\bv\Vert_2^2
+2\bar c_*\tilde\theta\langle\tilde\bv,\bmu-\bfSigma\hat\bv\rangle+E,
\end{align}
where by, Young's inequality,
$$
E:=2\bar c_*\tilde\theta\langle\tilde\bv,\bmu^+-\bmu\rangle-\Vert\bmu^+-\bmu\Vert_2^2
\le \bar c_*^2\tilde\theta^2.
$$
	
Item (ii) and \eqref{cor:step:dir:mean:1} of Corollary \ref{cor:step:dir:mean} imply
$
\hat\theta\ge0.00017\Vert\bmu-\bfSigma\hat\bv\Vert_2.
$
Hence, from \eqref{cor:step:dir:mean:2}, one gets
$
2\tilde\theta\langle\tilde\bv,\bmu-\bfSigma\hat\bv\rangle
\le -2.176\cdot10^{-6}\Vert\bmu-\bfSigma\hat\bv\Vert_2^2.
$
Items (ii) and \eqref{cor:step:dir:mean:1'} imply 
$
\tilde\theta\le1.02
\Vert\bmu-\bfSigma\hat\bv\Vert_2.
$

We thus conclude that 
\begin{align}
\Vert\bmu^+ -\bfSigma\hat\bv\Vert_2^2
\le \left[
1 - 2.176\cdot10^{-6}\bar c_* +
1.0404\bar c_*^2
\right]
\Vert\bmu-\bfSigma\hat\bv\Vert_2^2.
\end{align}
Minimizing over $\bar c_*$ yields the claim.  
\end{proof}

\section{Master algorithm}\label{s:master:rate}

We now present two master algorithms for robust regression. The first one, $\Algorithm$ \ref{algo:robust:regression:main}, assumes knowledge of 
\begin{align}
r_1:=2\sqrt{\frac{\tr(\bfXi)}{n}} \bigvee
\sqrt{\frac{\Vert\bfXi\Vert K}{n}}.
\end{align}
Note that $(\tr(\bfXi),\Vert\bfXi\Vert)$ requires information of the noise level, a difficult quantity to robustly estimate in practice. We present the alternative $\Algorithm$ \ref{algo:robust:regression:main:adaptive} in the next section, assuming no knowledge of 
$(\tr(\bfXi),\Vert\bfXi\Vert)$. Both algorithms assume knowledge  of the minimum and maximal eigenvalues 
$(\mu^{-2}(\mbB_2),\Vert\bfSigma\Vert)$ of the covariance matrix $\bfSigma$. As mentioned in the introduction, these values can be effectively replaced by their robust estimates (up to absolute constants).  

For ease of reference, we recall previously defined constants and rates. We recall the parameters
$
\alpha_1=1/96, 
$
$
\alpha_2=0.08, 
$
$
\alpha_3=0.239, 
$
$
\alpha_4=1/144,
$
$
c_{\alpha_1}:=\frac{1}{4},
$
$
\rho=1/36,
$
$\sa:=0.0128$, 
$
C_{\alpha_1}=2525.26, 
$
$
C_{\alpha_2}=192.4, 
$
$
C_{\alpha_3}=51.9
$
and
$
C_{\alpha_4}=4330.
$ 
Also, we recall $\bar\rho=1/24$, 
$
\bar\alpha_1=1/64, 
$
$
\bar\alpha_2=1/8,
$
$\bar\alpha_4=1/96$, 
$
c_{\bar\alpha_1}:=\frac{1}{4},
$
$\bar\sa:=\sa$, 
$
C_{\bar\alpha_1}=1666.68, 
$
and
$
C_{\bar\alpha_2}=110.
$
Finally, we recall the rates
\begin{align}
r&:=2\mu^2(\mbB_2)\sqrt{12\frac{K}{n}\tr(\bfXi)},\\
r_{n,K}&:=C L^2\sqrt{\frac{p\log p}{n}}
\bigvee L^2\sqrt{\frac{K}{n}},\\
r_2&:=2C_{\alpha_4}r_{n,K}+
CC_{\alpha_4} C_{\rho}'\frac{p\log p}{2n}
+2L^2\sqrt{\frac{CC_{\alpha_4} p\log p}{2n}},\\
\bar r_2&:=2C_{\bar\alpha_4}r_{n,K}+
CC_{\bar\alpha_4} C_{\bar\rho}'\frac{p\log p}{2n}
+2L^2\sqrt{\frac{CC_{\bar\alpha_4} p\log p}{2n}},\\
\bar r_1&:=\Vert\bfSigma\Vert \bar r_2, 
\end{align}
where $C>0$ is an absolute constant in Lemma \ref{lemma:trunc:quad:rademacher:comp} in the Appendix and $C_\alpha':=1+\sqrt{2/\alpha}$ for any 
$\alpha\in(0,1)$. Recall the constants
$
\bar c_*:=1.045752\cdot10^{-06}
$
and $\Delta_0:=1.093597\cdot10^{-12}$.
Define the constants $c_*:=\frac{\sa}{8(2+C_{\alpha_4})\kappa^2(\kappa^2+\Delta^2)}$ and
\begin{align}
\Delta:=263876.1\bar r_1\bigvee(1+\kappa)^2e^{-T_2\bar\Delta_0},
\end{align}
where $T_2\in\mathbb{N}$ is to be defined in the following. 

\begin{algorithm}[H]
\caption{$\RobustRegression(\calD,T_1,T_2,K,\eta,S_1,S_2,r_1)$}
\label{algo:robust:regression:main}
\textbf{Input}: sample $\calD:=\{(\tilde y_\ell,\tilde\bx_\ell)\}_{\ell=1}^n\cup\{(\tilde y_\ell,\tilde\bx_\ell)\}_{i=n+1}^{2n}$, outer and inner number of iterations $(T_1,T_2)$, number of buckets $K$, quantile probability $\eta\in(0,1/2)$, 
outer and inner simulation sample sizes 
$(S_1,S_2)$, optimal rate $r_1>0$.

\quad

\textbf{Output}: $\hat\bb(r_1)$. 
 
\begin{algorithmic}[1]
  \small 
   	\STATE Set $(\tilde\bb^{(0)},\tilde R_m,
	\{(\tilde y_\ell,\tilde\bx_\ell)\}_{\ell=1}^m)\leftarrow\Pruning(\calD,K,\eta)$.
	\STATE Split $\{(\tilde y_\ell,\tilde\bx_\ell)\}_{\ell=1}^{m}$ into $\calK$ buckets of same size $m/\calK=B$ indexed by the partition $\bigcup_{i\in[\calK]}\tilde B_i=[m]$. 
 	\STATE Set 
 	$\bb^1\leftarrow\tilde\bb^{(0)}$.
 	\FOR{$t=1:T_1$}
 	\STATE For $i\in[\calK]$, compute 
$
\tilde\bz_i(\bb^t):=\frac{1}{B}\sum_{\ell\in\tilde B_i}(\tilde\by_\ell-\langle\tilde\bx_\ell,\bb^t\rangle)\tilde\bx_\ell.
$
 	\STATE Set 
 	$$(\theta_t,\bv^t)\leftarrow\MW\left(
\{\tilde\bz_i(\bb^t)\}_{i=1}^{\calK},\max_{i\in[\calK]}\Vert\tilde\bz_i(\bb^t)\Vert_2^2,S_1,4\alpha_1\calK, c_{\alpha_1}\calK,\mu^2(\mbB_2),C_{\alpha_1}\mu^2(\mbB_2)r_1\right).
$$
\STATE $\bmu^t\leftarrow\RobustDirection(\{
\tilde\bx_\ell\}_{\ell=1}^m,\bv^t,T_2,S_2)$.	

 		\IF{$\theta^t<\nicefrac{\theta^{t-1}}{\kappa}$}	
 		\STATE Set
$
		\bb^{t+1}:=\bb^t + c_*(\theta_t+C_{\alpha_1}\mu^2(\mbB_2)r_1)\bmu^t.
		$
 		\ELSE 
 		\STATE $\bb^{t+1}\leftarrow\bb^t$.
\ENDIF 		
 	\ENDFOR
 	\RETURN $\bb^{T_1}$.
 \end{algorithmic}
\end{algorithm}

\begin{algorithm}[H]
\caption{$\RobustDirection(\calD,\hat\bv,T_2,S_2)$}
\label{algo:robust:direction:estimation}
\textbf{Input}: pruned feature sample $\calD:=\{\tilde\bx_\ell\}_{\ell=1}^m$, direction $\hat\bv$, inner number of iterations $T_2$, inner simulation sample size $S_2$. 

\quad

\textbf{Output}: estimate $\hat\bmu(\hat\bv)$ of $\bfSigma\hat\bv$. 
 
\begin{algorithmic}[1]
  \small 
	\STATE Set $\bmu^{1}\leftarrow (1,0,\ldots,0)\in\re^p$.
		\FOR{$\tau=1:T_2$}
 	\STATE For $i\in[\calK]$, compute 
$
\tilde\bz_i(\hat\bv,\bmu^{\tau}):=\frac{1}{B}\sum_{\ell\in\tilde B_i}\left(\langle\tilde\bx_\ell,\hat\bv\rangle\tilde\bx_\ell-\bmu^{\tau}
\right).
$
 	\STATE Set 
 	$$
 	(\tilde\theta_\tau,\tilde\bv^\tau)\leftarrow\MW\left(
\{\tilde\bz_i(\hat\bv,\bmu^{\tau})\}_{i=1}^{\calK}, \max_{i\in[\calK]}\Vert\tilde\bz_i(\hat\bv,\bmu^{\tau})\Vert_2^2,S_2,4\bar\alpha_1\calK,c_{\bar\alpha_1}\calK,1,C_{\bar\alpha_1}\bar r_1\right).
$$
 		\IF{$\tilde\theta^t<\tilde\theta^{t-1}$}	
 		\STATE Set
 		$
		\bmu^{\tau+1}:=\bmu^{\tau} +\bar c_*\tilde\theta_\tau\tilde\bv^\tau.
		$
 		\ELSE 
 		\STATE $\bmu^{\tau+1}\leftarrow\bmu^{\tau}$.
\ENDIF 		
 	\ENDFOR	
 	\RETURN $\hat\bmu^{T_2}$.
 \end{algorithmic}
\end{algorithm}

We introduce the assumptions in the next two main results. Let $K\in[n]$, $\eta\in(0,1/2)$ and $m=(1-\eta)n$. Grant Assumptions \ref{assump:L4:L2} and \ref{assump:contamination}
and assume that 
\begin{align}
K&>\max\{4,(\alpha_1\vee\alpha_2)^{-1},\alpha_3^{-1},\alpha_4^{-1},\bar\alpha_4^{-1}\}o,\label{condition1}\\
\eta&\ge4\epsilon,\label{condition2}
\end{align}
and
\begin{align}
L^2\left(
7\sqrt{\frac{K\log(24)}{n}} + 4\sqrt{\frac{p}{n}}
\right)&\le\frac{1}{2},\label{condition3}\\
C_{\alpha_3}r_{n,K}+2L^2\sqrt{\frac{CC_{\alpha_3}p\log p}{n}}&\le\frac{1}{2},\label{condition4}\\
\mu^2(\mbB_2)\Vert\bfSigma\Vert r_2&\le\frac{1}{312.5C_{\alpha_4}}.\label{condition5}
\end{align}

\begin{proposition}[Inner loop convergence]\label{prop:rate:known:rate:inner:loop}
Grant Grant Assumptions \ref{assump:L4:L2} and \ref{assump:contamination} and \eqref{condition1}, \eqref{condition2}, \eqref{condition3}, \eqref{condition4} and \eqref{condition5}. Instantiate $\Algorithm$ \ref{algo:robust:regression:main} with inputs
$\calD:=\{(\tilde y_\ell,\tilde\bx_\ell)\}_{\ell=1}^n\cup
\{(\tilde y_\ell,\tilde\bx_\ell)\}_{\ell=n+1}^{2n}$, $K$, $\eta$, integers 
$(T_1,T_2,S_1,S_2)$ and $r_1>0$. Recall the event $\calE$ in Set-up \ref{setup:final}. 

Then, for any $t\in[T_1]$, after $T_2$ iterations after $\RobustDirection(\{\tilde\bx_\ell\}_{\ell=1}^m,\bv^t,T_2,S_2)$ is queried, there is an event  $\calS(t,T_2)$ of probability at least $1-T_2e^{-\frac{S_2}{669}}$, such that on the event 
$\calE\cap\calS(t,T_2)$, 
\begin{align}
\Vert\bmu^{t,T_2}-\bfSigma\bv^t\Vert_2
\le263876.1\bar r_1\bigvee(1+\kappa)^2e^{-T_2\bar\Delta_0}.
\end{align}
\end{proposition}
\begin{proof}
On $\calE$, conditions \ref{lemma:high:prob:pruned:data:eq0}, \ref{cor:high:prob:pruned:data':mean:eq1} and
\ref{cor:high:prob:pruned:data':mean:eq2} all hold.

Fix outer iteration $t\in[T_1]$. Let $\{\bar\btheta_{\ell,\tau}\}_{\ell\in[S_2],\tau\in[T_2]}$ be the simulated iid sequence from the uniform distribution over $\mbS_2$ (independent of the label-feature data set) during the inner query of $\Algorithm$  \ref{algo:robust:direction:estimation} fir given $t\in[T_1]$. In the following, our arguments are over the event
 $\calE\cap\calS(t,T_2)$ where  $\calS(t,T_2)$ is an event of probability at least $1-T_2e^{-S/669}$ over the randomness of $\{\bar\btheta_{\ell,\tau}\}_{\ell\in[S_2],\tau\in[T_2]}$ conditioned on $\{\bx_\ell\}_{\ell=1}^n\cup\{\bx_\ell\}_{\ell=n+1}^{2n}$. 
We denote by $\{\bmu^{t,\tau}\}_{\tau\in[T_2]}$ the iterates during the call of $\Algorithm$ \ref{algo:robust:direction:estimation} for $t\in[T_1]$. 

Let 
$
\calG:=\{\bmu\in\re^p:\Vert\bmu-\bfSigma\bv^t\Vert_2
<1.32\cdot10^{5}\bar r_1\}.
$
We consider two cases. 
\begin{description}
\item[Case 1:] There is $\tau\in[T_2]$ such that 
$\bmu^{t,\tau}\in\calG$. If 
$\bmu^{t,T_2}\in\calG$ we are done, so suppose 
$\bmu^{t,T_2}\notin\calG$. By Corollary \ref{cor:step:dir:mean}, 
\begin{align}
(\nicefrac{1}{78.125})\Vert\bmu^{t,T_2}-\bfSigma\bv^t\Vert_2-1666.68\bar r_1&\le\tilde\theta_{T_2},\\
\tilde\theta_t&\le \Vert\bmu^{t,\tau}-\bfSigma\bv^t\Vert_2+1666.68\bar r_1.
\end{align}
As $\tilde\theta_{T_2}\le\tilde\theta_\tau$ by construction, 
\begin{align}
\Vert\bmu^{t,T_2}-\bfSigma\bv^t\Vert_2\le 
78.125\Vert\bmu^{t,\tau}-\bfSigma\bv^t\Vert_2
+79.125\bar r_1
\le263876.1\bar r_1. 
\end{align}

\item[Case 2:] For all $\tau\in[T_2]$, 
$\bmu^{t,\tau}\notin\calG$ so Lemma \ref{lemma:stepsize:descent:direction:mean} applies across all iterations after an union bound. Thus
$$
\Vert\bmu^{t,T_2}-\bfSigma\bv^t\Vert_2^2\le
e^{-T_2\bar\Delta_0}\Vert\tilde\bmu^{t,1}-\bfSigma\bv^t\Vert_2^2
\le (1+\kappa)^2e^{-T_2\bar\Delta_0}.
$$
\end{description}
Independence between $\{\btheta_{\ell,t}\}_{\ell\in[S],t\in[\mathsf{T}]}$ and $\{\bx_\ell\}_{\ell=1}^n\cup\{\bx_\ell\}_{\ell=n+1}^{2n}$ and an union bound finishes the proof.
\end{proof}

Next, we additionally assume that the sample size and inner loop number of iterations $T_2$ are large enough so that
\begin{align}
\Delta:=263876.1\bar r_1\bigvee(1+\kappa)^2e^{-T_2\bar\Delta_0}\le\frac{\sa}{16}.\label{condition6}
\end{align}
For ease of reference, we define the constants 
$C_*:=8\cdot10^5$, 
$
D_*:=1/625
$
and 
$
E_*:=\max\{C_*,(C_{\alpha_1}+C_{\alpha_2})/D_*\}=(C_{\alpha_1}+C_{\alpha_2})/D_*\ge17*10^5. 
$
For ease of reference, we also define, given input parameters $(K,\eta,T_1,T_2,S_1,S_2)$, the failure probability
\begin{align}
\delta&:=e^{-\frac{K}{5.4}}
+e^{-\frac{\eta n}{1.8}}
+e^{-\frac{K}{C_{\alpha_1}}}
+e^{-\frac{K}{C_{\alpha_2}}}
+e^{-\frac{K}{C_{\alpha_3}}}
+e^{-\frac{\rho n}{1.8}}
+e^{-\frac{K}{C_{\alpha_4}}}
+e^{-\frac{\bar\rho n}{1.8}}
+e^{-\frac{K}{C_{\bar\alpha_4}}}\\
&+T_1T_2e^{-\frac{S_2}{669}}
+T_1e^{-\frac{S_1}{353}}.\label{failure:prob}
\end{align}

\begin{theorem}[Outer loop convergence]\label{thm:rate:known:rate}
Grant Grant Assumptions \ref{assump:L4:L2} and \ref{assump:contamination} and \eqref{condition1}, \eqref{condition2}, \eqref{condition3}, \eqref{condition4},  \eqref{condition5} and \eqref{condition6}.  Instantiate $\Algorithm$ \ref{algo:robust:regression:main} with inputs
$\calD:=\{(\tilde y_\ell,\tilde\bx_\ell)\}_{\ell=1}^n\cup
\{(\tilde y_\ell,\tilde\bx_\ell)\}_{\ell=n+1}^{2n}$, $K$, $\eta$, integers 
$(T_1,T_2,S_1,S_2)$ and $r_1>0$. Recall the event $\calE$ in Set-up \ref{setup:final}.

\quad

Then after $T_1$ iterations of $\Algorithm$ \ref{algo:robust:regression:main}, on the event $\calE\cap\bigcap_{t\in[T_1]}\calS(t,T_2)$ of probability at least 
$
1-\delta
$, 
\begin{align}
\Vert\bb^{T_1}-\bb^*\Vert_2\le\left[\left(\frac{2C_{\alpha_2}}{\sa\kappa}
+\frac{2E_*}{\sa}\left(1+\frac{1}{312.5\kappa}\right)
+5050.52\right)\mu^2(\mbB_2)r_1\right]\bigvee3r e^{-\frac{\sa T_1}{8c_*}}.
\end{align}
\end{theorem}
\begin{proof}
On $\calE\cap\bigcap_{t\in[T_1]}\calS(t,T_2)$,  conditions \ref{lemma:high:prob:pruned:data:eq0}, \ref{lemma:high:prob:pruned:data':eq1}, 
\ref{lemma:high:prob:pruned:data':eq2},  \ref{lemma:high:prob:pruned:data':eq3}, \ref{lemma:high:prob:pruned:data':eq4},
\ref{cor:high:prob:pruned:data':mean:eq1}
and \ref{cor:high:prob:pruned:data':mean:eq2} all hold. By an union bound and independence, we have that $\calE\cap\bigcap_{t\in[T_1]}\calS(t,T_2)$ has probability at least as given in the statement of the theorem. We next state our arguments on the event $\calE\cap\bigcap_{t\in[T_1]}\calS(t,T_2)$. 

Let $\{\btheta_{\ell,t}\}_{\ell\in[S_1],t\in[T_1]}$ be the simulated iid sequence from the uniform distribution over $\mbS_2$ (independent of the label-feature data set) during queries of the outer loop of $\Algorithm$  \ref{algo:robust:regression:main}. 
Let $\calI\subset[T_1]$ be the set of iterations for which 
$
\Vert\bb^t-\bb^*\Vert_2
<E_*\mu^2(\mbB_2)r_1
$.

We consider two cases. 
\begin{description}
\item[Case 1:] There is $t\in\calI$. If 
$T_1\in\calI$ we are done, so suppose $T_1\notin\calI$. By construction $\theta_{T_1}\le\theta_t/\kappa$. From Corollary \ref{cor:step:dir} applied to $\bb^{T_1}$,
Lemma \ref{lemma:condition:RFHP:step:dir-looser} applied to $(\bb^t,\theta_t,\bv^t)$, 
$
C_{\alpha_4}\mu^2(\mbB_2)\Vert\bfSigma\Vert r_2\le1/312.5
$ and $\Vert\bv\Vert_2\le\mu^2(\mbB_2)$, we get 
\begin{align}
\Vert\bb^{T_1}-\bb^*\Vert_2&\le \frac{2}{\sa}\theta_{T_1}+2*2525.26\mu^2(\mbB_2)r_1\\
&\le \frac{2}{\sa\kappa}\left(
C_{\alpha_2}\mu^2(\mbB_2)r_1+|\langle\bb^*-\bb^t,\bv\rangle_\Pi|+C_{\alpha_4} \mu^2(\mbB_2)\Vert\bfSigma\Vert r_2\Vert\bb^*-\bb^t\Vert_2
\right)\\
&+5050.52\mu^2(\mbB_2)r_1\\
&\le \left(\frac{2C_{\alpha_2}}{\sa\kappa}+5050.52\right)\mu^2(\mbB_2)r_1
+\frac{2}{\sa\kappa}\cdot\kappa\Vert\bb^t-\bb^*\Vert_2
+\frac{2}{\sa\kappa}\cdot\frac{\Vert\bb^t-\bb^*\Vert_2}{312.5},\\
&<\left(\frac{2C_{\alpha_2}}{\sa\kappa}
+\frac{2E_*}{\sa}\left(1+\frac{1}{312.5\kappa}\right)
+5050.52\right)\mu^2(\mbB_2)r_1.
\end{align}

\item[Case 2:] $\calI=\emptyset$. In particular, 
$\Vert\bb^t-\bb^*\Vert_2\ge C_*\mu^2(\mbB_2)r_1$ for all $t\in[T_1]$ Corollary \ref{cor:step:dir} applies across all iterations after an union bound. In particular, for all $t\in[T_1]$, 
$
\langle\bv^t,\bb^*-\bb^t\rangle_\Pi\ge D_*\Vert\bb^*-\bb^t\Vert_2\ge 
D_*E_*\mu^2(\mbB_2)r_1\ge 
(C_{\alpha_1}+C_{\alpha_2})\mu^2(\mbB_2)r_1.
$
Additionally, $\sa_2(\kappa r_2)\le\frac{1}{4}$ and, by Proposition \ref{prop:rate:known:rate:inner:loop}, for all $t\in[T_1]$, 
\begin{align}
\Vert\bmu^{t,T_2}-\bfSigma\bv^t\Vert_2
\le263876.1\bar r_1\bigvee(1+\kappa)^2e^{-T_2\bar\Delta_0}=:\Delta\le\frac{\sa}{16}.
\end{align} 
Hence, can invoke Lemma \ref{lemma:stepsize:descent:direction} across all iterations (after an union bound) with $\sa_1:=2525.26=C_{\alpha_1}$, 
$\sa_2:=4330=C_{\alpha_4}$, 
$\sa_3:=192.4=C_{\alpha_2}$ and $\sa_4:=D_*$. We get
$$
\Vert\bb^{T_1}-\bb^*\Vert_2^2\le
e^{-\frac{\sa T_1}{4c_*}}\Vert\tilde\bb^{(0)}-\bb^*\Vert_2^2 
\le 9r^2 e^{-\frac{\sa T_1}{4c_*}}.
$$
\end{description}
Independence between $\{\btheta_{\ell,t}\}_{\ell\in[S_1],t\in[T_1]}$, $\{\bar\btheta_{\ell,\tau}\}_{\ell\in[S_2],\tau\in[T_2]}$ and $\{(y_\ell,\bx_\ell)\}_{\ell=1}^n\cup\{(y_\ell,\bx_\ell)\}_{\ell=n+1}^{2n}$, an union bound finishes the proof.
\end{proof}

\subsection{Adaptation to $r_1$}
\label{ss:adaptive:algo}

We now present algorithm $\Algorithm$  \ref{algo:robust:regression:main:adaptive} which is adaptive to the noise level. Here we assume to know a (loose) upper bound of 
$\Vert\bfXi\Vert$. In this setting, larger values for $r_1$ and $r$. Defining, for every $\zeta>0$, 
\begin{align}
r_1(\zeta):=\left(2\sqrt{\frac{p}{n}} \bigvee
\sqrt{\frac{ K}{n}}\right)\sqrt{\zeta}
\quad\mbox{and}\quad
r(\zeta):=2\mu^2(\mbB_2)\sqrt{12\frac{pK}{n}}\sqrt{\zeta}, 
\end{align}
we let, only in this section, $r_1:=r_1(\Vert\bfXi\Vert)$ and $r:=r(\Vert\bfXi\Vert)$. Since $\tr(\bfXi)\le p\Vert\bfXi\Vert$, these values of $(r_1,r)$ satisfy the conditions of Section \ref{s:master:rate}. We only assume a loose upper bound for $\Vert\bfXi\Vert$ and assume, without loss on generality, that $\Vert\bfXi\Vert\ge1$. 

We will need the following rate definition: given fixed $T_1\in\mathbb{N}$, for every $\zeta>0$, let 
\begin{align}
R(\zeta):=\left[\left(\frac{2C_{\alpha_2}}{\sa\kappa}
+\frac{2E_*}{\sa}\left(1+\frac{1}{312.5\kappa}\right)
+5050.52\right)\mu^2(\mbB_2)r_1(\zeta)\right]\bigvee3r(\zeta) e^{-\frac{\sa T_1}{8c_*}}.
\end{align}

\begin{algorithm}[H]
\caption{$\AdaptiveRobustRegression(\calD,T_1,T_2,K,\eta,S_1,S_2,\zeta_0,\gamma)$}
\label{algo:robust:regression:main:adaptive}
\textbf{Input}: sample $\calD:=\{(\tilde y_\ell,\tilde\bx_\ell)\}_{\ell=1}^n\cup\{(\tilde y_\ell,\tilde\bx_\ell)\}_{\ell=n+1}^{2n}$, outer and inner number of iterations $(T_1,T_2)$, number of buckets $K$, quantile probability $\eta\in(0,1/2)$, outer and inner simulation sample sizes 
$(S_1,S_2)$, $\gamma\in(0,1)$ and $\zeta_0>0$ satisfying $\gamma\zeta_0\ge \Vert\bfXi\Vert$.

\quad

\textbf{Output}: $\hat\bb$. 
 
\begin{algorithmic}[1]
  \small
\STATE Set $M:=\lceil\log_{\gamma^{-1}}(\zeta_0)\rceil$. 
\FOR{$\ell\in[M]$}
\STATE Set $\zeta_\ell:=\gamma^\ell \zeta_0$.
\STATE Set $\hat\bb(\zeta_\ell)\leftarrow\RobustRegression(\calD,T_1,T_2,K,\eta,S_1,S_2,r_1(\zeta_\ell))$.
\ENDFOR
\STATE Set $\hat\ell\leftarrow\max\left\{
\ell\in[M]:\bigcap_{j\in[\ell]}\mbB_2\left(\hat\bb(\zeta_\ell),R(\zeta_\ell)\right)\neq\emptyset
\right\}$.
\STATE Set $\hat\bmu\leftarrow\hat\bb(\zeta_{\hat\ell})$.
\RETURN $\hat\bmu$.
\end{algorithmic}
\end{algorithm}

We conclude with the following result. The arguments are standard and based on Lepski's method. See for instance \cite{2020dalalyan:minasyan}. 
\begin{theorem}[Noise level adaptive estimation]\label{thm:rate:adapt:r}
Grant Assumptions \ref{assump:L4:L2} and \ref{assump:contamination} and \eqref{condition1}, \eqref{condition2}, \eqref{condition3}, \eqref{condition4},  \eqref{condition5} and \eqref{condition6}. Let $\delta\in(0,1)$ as in \eqref{failure:prob}. Suppose that 
$\Vert\bfXi\Vert\ge1$. 

Then the output $\hat\bb$ of Algorithm  \ref{algo:robust:regression:main:adaptive} 
with inputs $\calD:=\{(\tilde y_\ell,\tilde\bx_\ell)\}_{\ell=1}^n\cup\{(\tilde y_\ell,\tilde\bx_\ell)\}_{\ell=n+1}^{2n}$, $(T_1,T_2,K,\eta,S_1,S_2)$,  $\gamma\in(0,1)$ and $\zeta_0>0$ such that $\Vert\bfXi\Vert<\zeta_0\gamma$ satisfies with probability at least $1-M\delta$, 
\begin{align}
\Vert\hat\bb-\bb^*\Vert_2\le
3R(\Vert\bfXi\Vert/\gamma). 
\end{align} 
\end{theorem}
\begin{proof}
Let $\ell^*:=\max\{\ell\in\mathbb{N}:\zeta_\ell\ge \Vert\bfXi\Vert\}$. Since 
$1\le\Vert\bfXi\Vert<\gamma\zeta_0$, we have 
$\ell_*\in[M]$ and
$
\zeta_0> \zeta_{\ell^*}\ge \Vert\bfXi\Vert\ge \zeta_{\ell^*}\gamma.
$ 
Define the event
$
\Omega_\ell:=\{\bb^*\in\mbB_2(\hat\bb(\zeta_\ell),R(\zeta_\ell))\}$ for all $\ell\in[M]$. For all $\ell\le\ell^*$, $\zeta_\ell\ge \zeta_{\ell^*}\ge \Vert\bfXi\Vert$; in particular, $r_1(\zeta_\ell)\ge r_1(\Vert\bfXi\Vert)$ and $r(\zeta_\ell)\ge r(\Vert\bfXi\Vert)$ so all the conditions of Theorem  \ref{thm:rate:known:rate} apply for such $\ell$. Precisely, we infer from such theorem that  $\prob(\Omega_\ell)\ge1-\delta$ for all 
$\ell\le\ell^*$. By an union bound, with probability at least $1-M\delta$, we must have 
$
\bb^*\in\cap_{\ell=1}^{\ell^*}\mbB_2(\hat\bb(\zeta_\ell),R(\zeta_\ell)).
$
The argument as follows will occur on this event. 

By maximality of $\hat\ell$, $\hat\ell\ge\ell^*$. In particular, $\zeta_{\hat\ell}\le \zeta_{\ell^*}$; therefore there must exist $\bb$ such that
$$
\bb\in \mbB_2(\hat\bb(\zeta_{\hat\ell}),R(\zeta_{\hat\ell}))\cap
\mbB_2(\hat\bmu(\zeta_{\ell^*}),R(\zeta_{\ell^*})). 
$$
This and triangle inequality implies that
\begin{align}
\Vert\hat\bb-\hat\bb(\zeta_{\ell^*})\Vert_2
\le \Vert\hat\bb-\bb\Vert_2+\Vert\hat\bb(\zeta_{\ell^*})-\bb\Vert_2
\le R(\zeta_{\hat\ell})+R(\zeta_{\ell^*})\le2R(\zeta_{\ell^*}).  
\end{align}
Using that 
$\bb^*\in\mbB_2(\hat\bb(\zeta_{\ell^*}),R(\zeta_{\ell^*}))$, 
$\zeta_{\ell^*}\le \Vert\bfXi\Vert/\gamma$ and that $\zeta\mapsto R(\zeta)$ is non-decreasing, we finally obtain that
\begin{align}
\Vert\hat\bb-\bb^*\Vert_2
\le \Vert\hat\bb-\hat\bb(\zeta_{\ell^*})\Vert_2
+\Vert\hat\bb^*-\hat\bb(\zeta_{\ell^*})\Vert_2
\le3R(\zeta_{\ell^*})\le 3R({\Vert\bfXi\Vert/\gamma}). 
\end{align}
This finishes the proof. 
\end{proof}

Let $\gamma\in(0,1)$ and $\zeta_0>0$ such that 
$\Vert\bfXi\Vert<\gamma\zeta_0$ and $M:=\lceil\log_{\gamma^{-1}}(\zeta_0)\rceil$. Let $\delta_0\in(0,1)$ be the desired probability of failure. In the following, $\sfC>1$ is a constant that may change from line to line. From the conditions of Theorem \ref{thm:rate:adapt:r}, it is straightforward to check that if the sample size and contamination fraction satisfy
\begin{align}
n&\ge(\sfC L^4\kappa^2p\log p)\bigvee(\sfC L^4\kappa^2\log(\nicefrac{\sfC M}{\delta_0})),\\
\epsilon&\le\frac{1}{\sfC L^4\kappa^2}, 
\end{align}
then, for any $T_1\in\mathbb{N}$, tuning $\Algorithm$ \ref{algo:robust:regression:main:adaptive} such that
\begin{align}
K&\ge\log\left(\nicefrac{\sfC M}{\delta_0}\right)\bigvee (\sfC  o),\\
\eta&\ge \frac{\sfC \log\left(\nicefrac{\sfC M}{\delta_0}\right)}{n}\bigvee (4\epsilon),\\
T_2&\ge \sfC \log\kappa,\\
S_1&\ge \sfC \log\left(\nicefrac{\sfC M T_1}{\delta_0}\right),\\
S_2&\ge \sfC\log\left(\nicefrac{\sfC M T_1 T_2}{\delta_0}\right),
\end{align}
the following estimate holds with probability at least $1-\delta_0$:
\begin{align}
\Vert\hat\bb-\bb^*\Vert_2^2&\lesssim
\mu^4(\mbB_2)\Vert\bfXi\Vert
\left(
\frac{p+K}{n}
\right)\bigvee
\left(
\frac{pK}{n}e^{-T_1/(\sfC c_*)}
\right)\\
&\lesssim \mu^4(\mbB_2)\Vert\bfXi\Vert
\left(
\frac{p}{n}+\frac{\log\left(\nicefrac{\sfC M}{\delta_0}\right)}{n}+\epsilon
\right)\bigvee
\left(
p\left(
\frac{\log\left(\nicefrac{\sfC M}{\delta_0}\right)}{n}+\epsilon
\right)
e^{-T_1/(\sfC c_*) }
\right).
\end{align}
Note that we can tune $T_1$ (without knowledge of 
$\Vert\bfXi\Vert$) as
$$
T_1\ge\sfC\kappa^4\log\left(
\frac{pK}{p+K}
\right)
$$
to obtain the optimal statistical rate. In particular, the optimal $T_1$ is independent of $\Vert\bfXi\Vert$. 

\section{Appendix}

\begin{lemma}[Lemma A.3 in \cite{2020hopkins:li:zhang}]\label{lemma:KL}
Let $p\in\Delta_{\calK,k}$ and $q$ be the uniform distribution on $[\calK]$. Then 
$\KL(p\Vert q)\le5\frac{k}{\calK}$. 
\end{lemma}

\subsection{Proof of Lemma \ref{lemma:count:combinatorial:variable:v2}}

Let $Q:=Q_{\overline{X},1-\eta/2}$. The one-sided Bernstein's inequality applied to 
$\unit_{\overline X > Q}$ implies: for any $t\ge0$, with probability at least $1-\exp(-t)$, 
\begin{align}
\frac{1}{n}\sum_{i=1}^n\unit_{\{\overline X_i > Q\}}
\le \prob(\overline X > Q) + \sigma\sqrt{\frac{2t}{n}} + \frac{t}{3n},
\label{lemma:count:combinatorial:variable:eq7}
\end{align}
where 
\begin{align}
\sigma^2:=\esp\left[\unit_{\{\overline X > Q\}}-\prob(\overline X > Q)\right]^2\le \prob(\overline X > Q).
\end{align}

By definition of quantile and that $X$ is absolute continuous, 
\begin{align}
\prob(\overline X > Q)=\eta/2.
\end{align} 
Take $t=c^2\eta n$ for $c>0$ satisfying $1/2+c+c^2/3\le3/4$, e.g., $c^2=0.563$. Then the RHS of \eqref{lemma:count:combinatorial:variable:eq7} is at most $3\eta/4$.

\subsection{Proof of Lemma \ref{lemma:count:block:variable}}
We only prove the first statement as the second is similar. Let $C_\alpha>0$ to be determined. The one-sided Bernstein's inequality applied to 
$\unit_{\hat\mu_k-\esp[X] > C_\alpha r}$ implies: for any $t\ge0$, with probability at least $1-\exp(-t)$, 
\begin{align}
\frac{1}{K}\sum_{k=1}^K\unit_{\{\hat\mu_k-\esp[X] > C_\alpha r\}}
\le \prob(\hat\mu_1-\esp[X] > C_\alpha r) + \sigma\sqrt{\frac{2t}{n}} + \frac{t}{3n},
\label{lemma:count:block:variable:eq1}
\end{align}
where 
\begin{align}
\sigma^2:=\esp\left[\unit_{\{\hat\mu_1-\esp[X] > C_\alpha r\}}-\prob(\hat\mu_1-\esp[X] > C_\alpha r)\right]^2\le \prob(\hat\mu_1-\esp[X] > C_\alpha r).
\end{align}

By definition Chebyshev's inequality and rotation invariance, 
\begin{align}
\prob(\hat\mu_1-\esp[X] > C_\alpha r)\le \frac{K}{n}\cdot\frac{\sigma_X^2}{C_\alpha^2 r^2}\le\frac{1}{C_\alpha^2},
\end{align} 
where we used definition of $r$.

Taking $t=K/C_\alpha$ in \eqref{lemma:count:block:variable:eq1} the claim is satisfied for $C_\alpha$ satisfying 
$$
\frac{1}{3C_\alpha}+\frac{1}{C_\alpha^2}+\frac{\sqrt{2}}{C_\alpha^{2/3}}\le\alpha.
$$

\subsection{Proof of Lemma \ref{lemma:count:block:emp:process}}

We only prove the first statement. Next we will take $r:= r_{X,n,K}(F)$ and a numerical constant $C_\alpha>0$ to be determined. The uniform Bernstein-type concentration inequality due to Bousquet applied to the empirical process
$
f\mapsto\sum_{k=1}^K\unit_{\{\hat\probn_{B_k}f-\probn f> C_\alpha r\}}
$
implies in particular that, for all $t\ge0$, with probability at least $1-e^{-t}$, 
\begin{align}
\sup_{f\in F}\frac{1}{K}\sum_{k\in[K]}\unit_{\hat\probn_{B_k}f-\probn f\ge C_\alpha r}\le \frac{2E}{K}+\sigma\sqrt{\frac{2t}{K}}+\frac{4t}{3K},\label{lemma:count:block:emp:process:eq1}
\end{align}
where
\begin{align}
E&:=\esp\left[\sup_{f\in F}\sum_{k\in[K]}\unit_{\{\hat\probn_{B_k}f-\probn f\ge C_\alpha r\}}\right],\\
\sigma^2&:=\sup_{f\in F}\esp\left(\unit_{\{\hat\probn_{B_k}f-\probn f\ge C_\alpha r\}}-\esp\unit_{\{\hat\probn_{B_k}f-\probn f\ge C_\alpha r\}}\right)^2\le\sup_{f\in F}\prob(\hat\probn_{B_k}f-\probn f\ge C_\alpha r).
\end{align}

Let $\varphi$ be a $2/(C_\alpha r)$-Lipschitz function such that 
$\unit_{t\ge C_\alpha r}\le\varphi(t)\le\unit_{t\ge C_\alpha r/2}$. 
Typical symmetrization-contraction arguments lead to
\begin{align}
\frac{E}{K}&\le \esp\left[\sup_{f\in F}\frac{1}{K}\sum_{k\in[K]}\varphi\left(\hat\probn_{B_k}f-\probn f\right)\right]\\
&\stackrel{\mbox{\tiny (Centering)}}{=}\esp\left[\sup_{f\in F}\frac{1}{K}\sum_{k\in[K]}\varphi\left(\hat\probn_{B_k}f-\probn f\right)-\esp[\varphi(\hat\probn_{B_k}f-\probn f)]\right]
+\sup_{f\in F}\esp[\varphi(\hat\probn_{B_k}f-\probn f)]\\
&\stackrel{\mbox{\tiny (Symmetrization)}}{\le} 2\esp\left[\sup_{f\in F}\frac{1}{K}\sum_{i\in[K]}\epsilon_k\varphi\left(\hat\probn_{B_k}f-\probn f\right)\right]+\sup_{f\in F}\esp[\varphi(\hat\probn_{B_k}f-\probn f)]\\
&\stackrel{\mbox{\tiny (Contraction)}}{\le}\frac{4}{r K}\esp\left[\sup_{f\in F}\sum_{k\in[K]}\epsilon_k(\hat\probn_{B_k}-\probn) f\right]+\sup_{f\in F}\esp[\varphi(\hat\probn_{B_k}f-\probn f)],\label{lemma:count:block:emp:process:eq2}
\end{align} 
where, by reverse symmetrization, the Rademacher complexity of the iid sequence $\{\hat\probn_{B_1} f,\ldots,\hat\probn_{B_K}f\}$ of block empirical averages may bounded by
\begin{align}
\mathscr{\tilde R}:=\esp\left[\sup_{f\in F}\sum_{k\in[K]}\epsilon_k(\hat\probn_{B_k}-\probn) f\right]
&\le2\esp\left[\sup_{f\in F}\left|\sum_{k\in[K]}\hat\probn_{B_k}f-\probn f\right|\right]\\
&=\frac{2K}{n}\esp\left[\sup_{f\in F}\left|\sum_{i\in[n]}f(X_i)-\probn f\right|\right]=\frac{2K}{n}\mathscr{D}_{X,n}(F).
\end{align}

Note that, by definition of $r_{X,n,K}(F)$,
$$
\frac{4}{C_\alpha r}\cdot\frac{\mathscr{\tilde R}}{K}\le \frac{8}{C_\alpha r}\cdot\frac{\mathscr{D}_{X,n}(F)}{n}\le \frac{8}{C_\alpha}.
$$

It remains to bound $\sup_{f\in F}\esp[\varphi(\hat\probn_{B_k}f-\probn f)]$ and $\sigma$ and gather the bounds in \eqref{lemma:count:block:emp:process:eq1}. 
In that regard, by Chebyshev's inequality and rotation invariance of variance
\begin{align}
\sup_{f\in F}\esp[\varphi(\hat\probn_{B_k}f-\probn f)]
&\le \sup_{f\in F}\prob(\hat\probn_{B_k}f-\probn f\ge C_\alpha r/2)
\le \sup_{f\in F}\frac{4\esp[(\hat\probn_{B_k}f-\probn f)^2]}{C_\alpha^2 r^2}
\le\frac{4K}{n}\cdot\frac{\sigma^2_X(F)}{C_\alpha^2 r^2},\\
\sigma^2&\le \sup_{f\in F}\prob(\hat\probn_{B_k}f-\probn f\ge C_\alpha r)\le\frac{K}{n}\cdot\frac{\sigma^2_X(F)}{C_\alpha^2 r^2}.
\end{align}
Again, by definition of $r_{X,n,K}(F)$,
\begin{align}
\frac{K}{n}\cdot\frac{\sigma_X^2(F)}{C_\alpha^2 r^2}\le\frac{1}{C_\alpha^2}. 
\end{align}

The statement of the lemma then follows by setting $t=K/C_\alpha$ in \eqref{lemma:count:block:emp:process:eq1} for a sufficiently large $C_\alpha$ satisfying
$$
\frac{8}{C_\alpha}+\frac{4}{C_\alpha^2}+\frac{\sqrt{2}}{C_\alpha^{2/3}}+\frac{4}{3C_\alpha}\le\alpha.
$$

\subsection{Proof of Proposition \ref{prop:quad:proc:block:upper:heavy}}

We prepare the ground to prove Proposition \ref{prop:quad:proc:block:upper:heavy} and assume that $\bz$ satisfies the $L^4-L^2$ norm equivalence condition for some $L>0$. Without loss on generality, we present a proof assuming $\bfSigma=\bfI_p$ as the general case can be reduce to this one. Given $R>0$, define the map 
$
\re^p\ni\bv\mapsto \bpi^R(\bv):=\bv^R:=(1\wedge\frac{R}{\Vert\bv\Vert_2})\bv.
$
In particular, for all $\bv,\bu\in\re^p$, 
$\Vert\bv^R\Vert_2=\phi_{R}(\Vert\bv\Vert_2)$ and
$
\langle\bv^R,\bu\rangle\le\langle\bv,\bu\rangle.
$
Next, define the ``truncated quadratic class'' 
$$
F^R:=\{f=\langle\bpi^R(\cdot),\bu\rangle^2:\bu\in\mbB_2\}.
$$ 
Of course, for any $k\in[K]$ and $f\in F^R$,
$$
\hat\probn_{B_k}f:=\frac{1}{B}\sum_{\ell\in B_k}\langle\bz_\ell^R,\bu\rangle^2
\quad
\mbox{and}
\quad
\probn f:=\langle\bz_\ell^R,\bu\rangle^2.
$$

\begin{definition}
Let
\begin{align}
\Re_{\bz,n,K}^R(\mbB_2)&:=\frac{\mathscr{R}_{\bz,n}(F^R)}{n}\bigvee L^2\sqrt{\frac{K}{n}}.
\end{align}
\end{definition}

\begin{corollary}[Truncated Quadratic Process]\label{cor:trunc:quad:uniform:heavy}
Let $\alpha\in(0,1)$ and any constant $C_\alpha>0$ satisfying \eqref{lemma:count:block:emp:process:alpha}.

Then letting $r:=\Re_{\bz,n,K}^R(\mbB_2)$, with probability at least $1-e^{-K/C_\alpha}$,
\begin{align}
\sup_{f\in F^R}\sum_{k=1}^K\unit_{\{
|\hat\probn_{B_k}f-\probn f|>C_\alpha\cdot r\}}\le \alpha K.
\end{align}
\end{corollary}
\begin{proof}
By $L^4-L^2$ norm equivalence, 
\begin{align}
\sigma_{\bz}(F^R)\le L^2\sigma_{\bz}^2(\mbB_2)=L^2,
\end{align}
implying that 
$r_{\bz,n,K}(F^R)\le \Re_{\bz,n,K}^R(\mbB_2)$. The claim is then immediate applying Lemma \ref{lemma:count:block:emp:process} to the class $F^R$. 
\end{proof} 	

We now aim in bounding 
\begin{align}
\mathscr {R}_{\bz,n}(F^R)=\esp\left[\left\Vert
\sum_{i\in[n]}\epsilon_i\bz_i^R\otimes\bz_i^R
\right\Vert\right].
\end{align}
We use a standard approach via the matrix Bernstein's inequality due to Minsker \cite{2017minsker}. We will need the following lemma whose proof we omit.
\begin{lemma}\label{lemma:truncated:trace}
For all $\bv\in\re^p$,
\begin{align}
\esp\langle\bz^R,\bv\rangle^4&\le L^4\langle\bv,\bv\rangle^2,\\
\esp\Vert\bz^R\Vert_2^4&\le L^4p^2,\\
\esp\langle\bz^R,\bu\rangle^2
&\ge 
\left(
1-\frac{L^4 p}{R^2}
\right)
\esp\langle\bz,\bu\rangle^2.
\end{align}
\end{lemma}

Next we set
$
\bfS_i:=\epsilon_i\bz^R\otimes\bz^R
$
and
$
\bfS:=\sum_{i=1}^n\bfS_i.
$
For all $i$, $\esp[\bfS_i]=\bf0$ and 
$\Vert\bfS_i\Vert\le R^2$. Define the ``matrix variance'' 
\begin{align}
\var(\bfS):=\sum_{i=1}^n\esp[\bfS_i\bfS_i^\top]
=n\esp[\Vert\bz^R\Vert_2^2\bz^R\otimes\bz^R].
\end{align}
We claim that 
$
\var(\bfS)\preceq nL^4p\cdot\bfI_p.
$
Indeed by Lemma \ref{lemma:truncated:trace}, for any 
$\bv\in\re^p$, 
\begin{align}
\langle\bv,\var(\bfS)\bv\rangle
=n\esp[\Vert\bz^R\Vert_2^2\langle\bv,\bz^R\rangle^2]
\le n \sqrt{\esp\Vert\bz^R\Vert_2^4}\sqrt{\esp\langle\bv,\bz^R\rangle^4}
\le nL^4p\Vert\bv\Vert_2^2.
\end{align}
In particular, 
$
\Vert\var(\bfS)\Vert\le nL^4p.
$
The effective rank of $nL^4p\cdot\bfI_p$ is $p$. The bound by Minsker \cite{2017minsker} then yields
\begin{align}
\mathscr {R}_{\bz,n}(F^R)=\esp\Vert\bfS\Vert\lesssim \sqrt{nL^4p\log p}+R^2\log p,
\end{align}
implying the lemma:
\begin{lemma}\label{lemma:trunc:quad:rademacher:comp}
For an absolute constant $C>0$,
\begin{align}
\frac{\mathscr {R}_{\bz,n}(F^R)}{n}
\le C
\left(L^2\sqrt{\frac{p\log p}{n}}
\bigvee\frac{R^2}{n}\log p
\right).
\end{align}
\end{lemma}

We finalize with the proof of Proposition \ref{prop:quad:proc:block:upper:heavy}.
\begin{proof}[Proof of Proposition \ref{prop:quad:proc:block:upper:heavy}]
\textbf{Upper bound}: We prove the first inequality. Lemma  \ref{lemma:count:combinatorial:variable:v2} applied to $X:=\Vert\bz\Vert_2-\esp\Vert\bz\Vert_2$ together with 
$
Q_{1-\rho/2}(X)\le\sqrt{2p/\rho}
$
and $\esp[\Vert\bz\Vert_2]\le\sqrt{p}$ imply that on an event $\calE_1$ of probability at least $1-e^{-\rho n/1.8}$, for at least a fraction of $1-0.75\rho$ of the $n$ data points, 
\begin{align}
\Vert\bz_\ell\Vert_2\le \left(1+\sqrt{2/\rho}\right)\sqrt{p}.
\end{align}

Using Lemma \ref{lemma:trunc:quad:rademacher:comp} with $R:=\left(1+\sqrt{2/\rho}\right)\sqrt{p}$ and Corollary \ref{cor:trunc:quad:uniform:heavy}, we have on an event $\calE_2$ of  probability at least $1-e^{-K/C_\alpha}$, for all 
$\bu\in\mbB_2$, for at least a fraction of $1-\alpha$ of the $K$ blocks (and corresponding data points), 
\begin{align}
\frac{1}{B}\sum_{\ell\in B_k}\left(\langle\bz^R_\ell,\bu\rangle^2-\Vert\bu\Vert^2_2\right)\le C_\alpha\left[
r_{n,K}\bigvee CC'_\rho\frac{p\log p}{n}
\right],
\end{align} 
where we used that $\esp[\langle\bz^R,\bu\rangle^2]\le\Vert\bu\Vert^2$.

On the event $\calE_1\cap\calE_2$ of probability 
$1-e^{-\rho n/1.8}-e^{-K/C_\alpha}$, we invoke the pigeonhole principle so that for a fraction of $1-(\alpha+0.75\rho)$ of the blocks (and corresponding data-points), both displayed inequalities hold. For these data points, $\bz_\ell^R=\bz_\ell$ so the claim of the lemma holds. 

\textbf{Lower bound}: we now prove the second inequality. By analogous argument, Lemma \ref{lemma:trunc:quad:rademacher:comp} Corollary \ref{cor:trunc:quad:uniform:heavy}, we have on an event of probability at least $1-e^{-K/C_\alpha}$, for all 
$\bu\in\mbB_2$, for at least a fraction of $1-\alpha$ of the $K$ blocks,
\begin{align}
\frac{1}{B}\sum_{\ell\in B_k}\left(
\esp\langle\bz^R,\bu\rangle^2
-\langle\bz^R_\ell,\bu\rangle^2
\right)\le C_\alpha\left[
r_{n,K}\bigvee C\frac{R^2\log p}{n}
\right].
\end{align} 
We now use the facts that
$
\langle\bz^R_\ell,\bu\rangle^2\le\langle\bz_\ell,\bu\rangle^2
$
and, by Lemma \ref{lemma:truncated:trace}, 
\begin{align}
\esp\langle\bz^R,\bu\rangle^2
&\ge 
\left(
1-\frac{L^4 p}{R^2}
\right)
\esp\langle\bz,\bu\rangle^2.
\end{align}
The proof is finished taking $R:=\sqrt{\theta p}$ for given $\theta>0$.
\end{proof}

\subsection{Proof of Corollary \ref{cor:prod:proc:block:upper:heavy}}

By the parallelogram law and Proposition \ref{prop:quad:proc:block:upper:heavy}, on a event of probability at least $1-e^{-\rho n/1.8}-2e^{-K/C_\alpha}$, for all $[\bu,\bv]\in\mbB_\probn\times\mbB_\probn$, for at least $(1-(2\alpha+0.75\rho))K$ of the blocks, 
\begin{align}
\frac{1}{B}\sum_{\ell\in B_k}
\left(
\langle\bz_\ell,\bu\rangle
\langle\bz_\ell,\bv\rangle
-\langle\bu,\bv\rangle_\probn
\right)&=
\frac{1}{4B}\sum_{\ell\in B_k}
\left(
\langle\bz_\ell,\bu+\bv\rangle^2
-\Vert\bu+\bv\Vert_\probn^2
\right)\\
&-\frac{1}{4B}\sum_{\ell\in B_k}
\left(
\langle\bz_\ell,\bu-\bv\rangle^2
-\Vert\bu-\bv\Vert_\probn^2
\right)\\
&\le \frac{\Vert\bu+\bv\Vert_\probn^2}{4}
C_\alpha\left[
r_{n,K}\bigvee CC'_\rho\frac{p\log p}{n}
\right]\\
&+\frac{\Vert\bu-\bv\Vert_\probn^2}{4}
\left(
\frac{L^4}{\theta}
+C_\alpha\left[
r_{n,K}\bigvee C\theta\frac{p\log p}{n}
\right]
\right).
\end{align} 
Optimizing over $\theta$, one gets
\begin{align}
\frac{L^4}{\theta}+\theta\frac{CC_\alpha p\log p}{n}\le2L^2\sqrt{\frac{CC_\alpha p\log p}{n}}.
\end{align}
Using that $\Vert\bu+\bv\Vert_\probn^2\le4$ and $\Vert\bu-\bv\Vert_\probn^2\le4$ finishes the proof.

\subsection{Proof sketch of Proposition \ref{prop:quad:proc:block:lower:heavy}}

The proof follows similar lines as Theorem 3.1 in \cite{2016oliveira}. It suffices to prove for the case 
$\bfSigma$ is the identity. Given $R>0$, define  
$
\bz_i^R:=\left(1\wedge\frac{R}{\Vert\bz_i\Vert_2}\right)\bz_i.
$
Fix $s>0$ and $k\in[K]$. We apply the PAC-Bayesian inequality in Proposition 3.1 in \cite{2016oliveira} with covariance matrix $\bfC:=\bfI/(Kp)$ and process
\begin{align}
Z_{\btheta,k}:=s\esp[\langle\btheta,\bz_1^R\rangle^2]
-s\sum_{\ell\in B_k}\frac{\langle\btheta,\bz_\ell^R\rangle^2}{B}
-\frac{s^2}{2B}\esp[\langle\btheta,\bz_1^R\rangle^4].
\end{align} 
By Lemma B.2 in \cite{2016oliveira}, one concludes that 
$\esp[e^{Z_{\btheta,k}}]\le1$ for all $\btheta$. By Proposition 3.1 in \cite{2016oliveira}, we deduce that, with probability at least $1-e^{-t}$, for all $\bv\in\mbS_2$, 
\begin{align}
\sum_{\ell\in B_k}\Gamma_{\bv,\bfC}\frac{\langle\btheta,\bz_\ell^R\rangle^2}{B}\ge\Gamma_{\bv,\bfC}\esp[\langle\btheta,\bz_1^R\rangle^2]
-\left(
\frac{s}{2B}\esp[\Gamma_{\bv,\bfC}\langle\btheta,\bz_1^R\rangle^4]
+\frac{Kp+2t}{2s}
\right). 
\end{align}

As in \cite{2016oliveira}, one deduces from Lemma 3.1 in that paper the estimates
\begin{align}
\sum_{\ell\in B_k}\Gamma_{\bv,\bfC}\frac{\langle\btheta,\bz_\ell^R\rangle^2}{B}&\le\sum_{\ell\in B_k}\frac{\langle\bv,\bz_\ell\rangle^2}{B}
+\sum_{\ell\in B_k}\frac{\Vert\bz_i^R\Vert_2^2}{KpB},\\ 
\Gamma_{\bv,\bfC}\esp[\langle\btheta,\bz_1^R\rangle^2]&\ge
1-\frac{L^4p}{R^2}+\frac{\esp[\Vert\bz_1^R\Vert_2^2]}{Kp},
\end{align}
and 
\begin{align}
\esp\left[\Gamma_{\bv,\bfC}\langle\btheta,\bz_i^R\rangle^4\right]
&\le 4\esp\left[\langle\bv,\bz_i^R\rangle^4\right]
+12\esp\left[\frac{\Vert\bz_i^R\Vert_2^4}{(Kp)^2}\right]
+8\esp\left[\langle\bv,\bz_i^R\rangle^2\frac{\Vert\bz_i^R\Vert_2^2}{Kp}\right]\\
&\le 8\esp\left[\langle\bv,\bz_i^R\rangle^4\right]
+16\esp\left[\frac{\Vert\bz_i^R\Vert_2^4}{(Kp)^2}\right]\\
&\le8L^4+\frac{16L^4}{K^2}. 
\end{align}

We also have, by Bernstein's inequality, with probability at least $1-e^{-t}$, 
\begin{align}
\sum_{\ell\in B_k}\frac{\Vert\bz_i^R\Vert_2^2-\esp[\Vert\bz_i^R\Vert_2^2]}{KpB}\le L^2\sqrt{\frac{2t}{K^2B}}+\frac{2R^2t}{3KpB}.	
\end{align}

Invoking an union bound and using the previous bounds, we conclude that, with probability at least $1-2e^{-t}$, for all $\bv\in\mbS_2$, 
\begin{align}
1-\sum_{\ell\in B_k}\frac{\langle\bv,\bz_\ell\rangle^2}{B}&\le L^2\sqrt{\frac{2t}{K^2B}}
+\left(\frac{L^4p}{R^2}+\frac{2R^2t}{3KpB}\right) 
+\left(\frac{4L^4s}{B}+\frac{t}{s}\right)
+\left( \frac{8L^4s}{BK^2} +\frac{Kp}{2s}\right).
\end{align}
The first term is 
$
L^2\sqrt{\frac{2t}{Kn}}.
$
Optimizing $R>0$, the second term is less than $2L^2\sqrt{\frac{2t}{3n}}$. We now choose $s=s_1s_2$ with arbitrary $s_i>0$, $i=1,2$. For any $s_2$, the third term minimized at $s_1^*$ has value $2\sqrt{\frac{4L^4s_2t}{Bs_2}}=4L^2\sqrt{\frac{Kt}{n}}$. Evaluated at $s_1^*$, the forth term minimized at $s_2^*$ has value 
$
2\sqrt{\frac{8L^4s_1^*Kp}{BK^22s_1^*}}=4L^2\sqrt{\frac{p}{n}}.
$
Using the overestimates $K\ge1$, $2/3\le2$ and $1\le2$,  finishes the proof. 

\subsection{Proof of Proposition \ref{prop:unif:round}}

For simplicity we set 
$
f_i(\bfM):=\llangle\bz_i\bz_i^\top,\bfM\rrangle.
$ The assumption implies in particular that 
the set 
$$
\calA(\bfM):=\{i\in[m]:\Vert\bfM^{1/2}\bz_i\Vert_2> \sqrt{D}\}
$$ 
has cardinality $|\calA(\bfM)|>\sb m$. 

Let $\bu:=\frac{\bfM^{1/2}\bz_i}{\Vert\bfM^{1/2}\bz_i\Vert_2}$. The angle $\angle(\btheta,\bu)$ between $\btheta$ and $\bu\in\mbS_2$ is uniformly distributed over $[0,2\pi]$. Define 
the constants $\sfC:=(\cos\varphi)^{-1}$ and $B:=\sqrt{D}/\sfC$. Define also the random variable 
$$
Z_{\btheta}:=\sum_{i=1}^m\unit_{\left\{|\langle\bfM^{1/2}\bz_i,\btheta\rangle|>B\right\}}=\sum_{i=1}^m\unit_{\left\{|\langle\bz_i,\bfM^{1/2}\btheta\rangle|>B\right\}}.
$$
For any $i\in\calA(\bfM)$, 
\begin{align}
\prob\left(|\langle\bfM^{1/2}\bz_i,\btheta\rangle|>B\right)&\ge
\prob\left(|\langle\bfM^{1/2}\bz_i,\btheta\rangle|>\frac{1}{\sfC}
\Vert\bfM^{1/2}\bz_i\Vert_2
\right)\\
&=\prob(|\langle\bu,\btheta\rangle|>1/\sfC)\\
&=\prob(|\cos\angle(\btheta,\bu)|>1/\sfC)\\
&=4\prob(0\le\angle(\btheta,\bu)<\varphi)\\
&=4\frac{\varphi}{2\pi}.
\end{align}
It follows that
$
\esp Z_{\btheta}\ge|\calA(\bfM)|\frac{2\varphi}{\pi}\ge
\frac{2\varphi\sb}{\pi} m.
$

Note that almost surely $Z_{\btheta}\le m$. From Paley-Zygmund's inequality (Proposition 3.3.1 in \cite{2000pena:gine}), for all 
$\sa\in(0,1]$,
\begin{align}
\prob\left(Z_{\btheta}>\sa m\right)\ge\frac{\left(\esp Z_{\btheta}-\sa m\right)^2}{\esp Z_{\btheta}^2}
\ge\left(\frac{2\varphi\sb}{\pi}-\sa\right)^2>0,
\end{align}
if we assume that
$
\frac{2\varphi\sb}{\pi\sa}>1.
$
In other words, with probability at least 
$\left(\frac{2\varphi\sb}{\pi}-\sa\right)^2$ the vector 
$\bv_{\btheta}:=\bfM^{1/2}\btheta\in\mbB_2$ satisfies
$
\sum_{i=1}^m\unit_{\{|\langle\bz_i,\bv_{\btheta}\rangle|>B\}}
>\sa m.
$

\bibliographystyle{plain}
\bibliography{references_robust_linear_regression}

\end{document}